\documentclass[11pt,reqno]{amsart}

\usepackage{a4wide}
\usepackage[english]{babel}
\usepackage[utf8]{inputenc}
\usepackage{amsmath,amssymb}
\usepackage{tikz}
\usepackage[hidelinks]{hyperref}

\theoremstyle{definition}
\newtheorem{thm}{Theorem}[section]
\newtheorem{lem}[thm]{Lemma}
\newtheorem{prop}[thm]{Proposition}
\newtheorem{cor}[thm]{Corollary}
\newtheorem{defi}[thm]{Definition}
\newtheorem{rem}[thm]{Remark}
\numberwithin{equation}{section}

\newcommand{\Arg}{\operatorname{Arg}}
\newcommand{\dom}{\operatorname{dom}}
\newcommand{\ran}{\operatorname{ran}}
\newcommand{\sgn}{\operatorname{sgn}}
\renewcommand{\Im}{\operatorname{Im}}
\newcommand{\bnd}{{\rm bnd}}
\newcommand{\reg}{{\rm reg}}
\renewcommand{\max}{{\rm max}}

\title[The $H^\infty$-functional calculus for bisectorial  Clifford operators]{The $H^\infty$-functional calculus for bisectorial \\ Clifford operators}

\author[Francesco Mantovani]{F. Mantovani}
\address{(FM) Politecnico di Milano, Dipartimento di Matematica, Via E. Bonardi, 9, 20133 Milano, Italy}
\email{francesco.mantovani@polimi.it}

\author[Peter Schlosser]{P. Schlosser}
\address{(PS) Politecnico di Milano, Dipartimento di Matematica, Via E. Bonardi, 9, 20133 Milano, Italy}
\email{pschlosser@math.tugraz.at}

\begin{document}

\begin{abstract}
The aim of this article is to introduce the $H^\infty$-functional calculus for unbounded bisectorial operators in a Clifford module over the algebra $\mathbb{R}_n$. This work is based on the universality property of the $S$-functional calculus, which shows its applicability to fully Clifford operators. While recent studies have focused on bounded operators or unbounded paravector operators, we now investigate unbounded fully Clifford operators and define polynomially growing functions of them. We first generate the $\omega$-functional calculus for functions that exhibit an appropriate decay at zero and at infinity. We then extend to functions with a finite value at zero and at infinity. Finally, using a subsequent regularization procedure, we can define the $H^\infty$- functional calculus for the class of regularizable functions, which in particular include functions with polynomial growth at infinity and, if $T$ is injective, also functions with polynomial growth at zero.
\end{abstract}

\maketitle

 \medskip

 \medskip

\textbf{Acknowledgements:} Peter Schlosser was funded by the Austrian Science Fund (FWF) under Grant No. J 4685-N and by the European Union--NextGenerationEU.

\section{Introduction}

In complex Banach spaces, the $H^\infty$-functional calculus is an extension of the Riesz-Dunford functional calculus \cite{RD}, also called holomorphic functional calculus,
\begin{equation}\label{Eq_Riesz_Dunford}
f(T)=\frac{1}{2\pi i}\int_{\partial U}(z-T)^{-1}f(z)dz,
\end{equation}
to functions which do not decay fast enough for the integral \eqref{Eq_Riesz_Dunford} to converge. The idea of the corresponding regularization procedure has been introduced by A. McIntosh in \cite{McI1} and further explored in various studies as \cite{MC10,MC97,MC06,MC98}. It is extensively used in the investigation of evolution equations, particularly in maximal regularity and fractional powers of differential operators. Additionally, there exists a close relation to Kato's square root problem and pseudo-differential operators. For a comprehensive exposition of this calculus, see the books \cite{Haase,HYTONBOOK1,HYTONBOOK2} and the references therein. \medskip

For quaternionic and Clifford-paravector operators on the other hand, the $H^\infty$-functional calculus has already been considered in works such as \cite{ACQS2016,CGdiffusion2018}, with additional discussions in \cite{JONAMEM,JONADIRECT}. These operators are relevant in various scientific and technological contexts. \medskip

A significant motivation for our investigation is to extend this functional calculus on the one hand to unbounded fully Clifford operators, as suggested in \cite{ADVCGKS}, and on the other hand to bisectorial operators in order to allow spectrum also on the negative real axis. One particular application is the gradient operator, considered as a scalar Clifford operator. While the quaternionic case corresponds to the three dimensional problem
\begin{equation*}
\nabla=i\frac{\partial}{\partial x}+j\frac{\partial}{\partial y}+k\frac{\partial}{\partial z},
\end{equation*}
the Clifford case allows the generalization to $n$ dimensions
\begin{equation*}
\nabla=e_1\frac{\partial}{\partial x_1}+\dots+e_n\frac{\partial}{\partial x_n}.
\end{equation*}
See also \cite{CMS24}, where this gradient operator is investigated with different boundary conditions on bounded and unbounded domains. Also the spectral theorem has recently been extended from quaternionic operators in \cite{ACK}, to fully Clifford operators in \cite{ColKim}. \medskip

The $H^\infty$-functional calculi also has a broader context  in the recently introduced fine structures on the $S$-spectrum. These are function spaces of nonholomorphic functions derived from the Fueter-Sce extension theorem, which connects slice hyperholomorphic and axially monogenic functions, the two main notions of holomorphicity in the Clifford setting. The connection is established through powers of the Laplace operator in a higher dimension, see \cite{Fueter,TaoQian1,Sce} and also the translation \cite{ColSabStrupSce}. These function spaces and the related functional calculi are introduced and studied in the recent works \cite{CDPS1,Fivedim,Polyf1,Polyf2} and the respective $H^\infty$-version are investigated in \cite{MILANJPETER,MPS23}. \medskip

Unlike complex holomorphic function theory, the noncommutative setting of the Clifford algebra allows multiple notions of hyperholomorphicity for vector fields. Consequently, different spectral theories emerge, each based on the concept of spectrum that relies on distinct Cauchy kernels. The spectral theory based on the $S$-spectrum was inspired by quaternionic quantum mechanics (see \cite{BF}), is associated with slice hyperholomorphicity, and began its development in 2006. A comprehensive introduction can be found in \cite{CGK}, with further explorations done in \cite{ACS2016,AlpayColSab2020,ColomboSabadiniStruppa2011}. Applications on fractional powers of operators are investigated in \cite{CGdiffusion2018,CG18,FJBOOK} and some results from classical interpolation theory, see \cite{BERG_INTER,BRUDNYI_INTER,Ale_INTER,TRIEBEL}, have been recently extended into this setting \cite{COLSCH}. In order to fully appreciate the spectral theory on the $S$-spectrum, we recall that it applies to sets of noncommuting operators, for example with the identification $(T_1,...,T_{2^n})\leftrightarrow T=\sum_AT_Ae_A$ it can be seen as a theory for several operators. \medskip

At this point, we also want to mention the spectral theory on the monogenic spectrum initiated in \cite{JM}. It is based on monogenic functions (functions in the kernel of the Dirac operator) and their associated Cauchy formula (see \cite{DSS}). This theory is well developed in the books \cite{JBOOK,TAOBOOK}, where the $H^\infty$-functional calculus in the monogenic setting is extensively considered. \medskip

The major difference between classical complex spectral theory and spectral theory on the $S$-spectrum is that, for bounded linear operators $T:V\rightarrow V$, acting in a Clifford module $V$, the series expansion of the resolvent operator is expressed as
\begin{equation}\label{Eq_Resolvent_sum}
\sum\nolimits_{n=0}^\infty T^ns^{-n-1}=(T^2-2s_0T+|s|^2)^{-1}(\overline{s}-T),\qquad|s|>\Vert T\Vert,
\end{equation}
where $s=s_0+s_1e_1+\dots+s_ne_n$ is a paravector. It is important to point out that the sum of the series
\eqref{Eq_Resolvent_sum} is independent of the fact that the operators $T_A$ in $T=\sum_AT_Ae_A$ commute among themselves. The reason why the value of the series \eqref{Eq_Resolvent_sum} does not simplify to the classical resolvent operator $(s-T)^{-1}$, is due to the noncommutativity $sT\neq Ts$. The explicit value of the sum \eqref{Eq_Resolvent_sum} now motivates that, even for unbounded $T$, the spectrum has to be associated with the invertibility of the operator $T^2-2s_0T+|s|^2$. This leads us to the definition of the $S$-spectrum $\sigma_S(T)$ and the $S$-resolvent set $\rho_S(T)$ in Definition~\ref{defi_S_spectrum}, see \cite{ColomboSabadiniStruppa2011,CGK} for the original definition. For every $s\in\rho_S(T)$ we then define the {\it left $S$-resolvent operator}
\begin{equation*}
S_L^{-1}(s,T):=(T^2-2s_0T+|s|^2)^{-1}\overline{s}-T(T^2-2s_0T+|s|^2)^{-1},
\end{equation*}
where, compared to \eqref{Eq_Resolvent_sum}, the operator $T$ is shifted to the left of $(T^2-2s_0T+|s|^2)^{-1}$ in order for $S_L^{-1}(s,T)$ to be everywhere defined. As a generalization of \eqref{Eq_Riesz_Dunford}, for bounded linear Clifford operators $T$ and for left slice hyperholomorphic functions $f$, the $S$-functional calculus is then defined as
\begin{equation}\label{Eq_S_functional_calculus}
f(T):=\frac{1}{2\pi}\int_{\partial U\cap\mathbb{C}_J}S_L^{-1}(s,T)ds_Jf(s).
\end{equation}
Here, $J\in\mathbb{S}$ is an arbitrary imaginary unit from \eqref{Eq_S} and $U$ is an open set such that it contains the $S$-spectrum of $T$ and $f$ is holomorphic in a neighbourhood of $\overline{U}$. Moreover, the path integral is understood in the sense \eqref{Eq_Path_integral}. A similar definition holds for right slice hyperholomorphic functions, involving the right $S$-resolvent operator
\begin{equation*}
S_R^{-1}(s,T)=(\overline{s}-T)(T^2-2s_0T+|s|^2)^{-1}.
\end{equation*}
Although all the results of the Sections~\ref{sec_Omega}~\&~\ref{sec_Extended} can be extended to right slice hyperholomorphic functions, it is still unclear how to generalize the $H^\infty$-functional calculus in Section~\ref{sec_Hinfty} to right slice  hyperholomorphic functions, see also Remark~\ref{rem_Hinfty_right}. For this reason we avoid considering right slice hyperholomorphic functions in this article at all, to make the presentation of the results clearer. \medskip

The aim of this article is now to develop a functional calculus for unbounded bisectorial operators of angle $\omega$, see Definition~\ref{defi_Bisectorial_operators}, based on the integral formula \eqref{Eq_S_functional_calculus}. Roughly speaking, these are closed operators for which the $S$-spectrum is contained in the double sector of angle $\omega$ and where the $S$-resolvent operator behaves as $\Vert S_L^{-1}(s,T)\Vert\lesssim\frac{1}{|s|}$ at zero and at infinity. This extension is a delicate issue and requires a three step procedure. \medskip

In Section~\ref{sec_Omega}, a version of the $S$-functional calculus \eqref{Eq_S_functional_calculus}, the so called $\omega$-functional calculus, is used to generate $f(T)$ for functions that suitably decay at zero and at infinity, i.e. all left slice hyperholomorphic functions in a double sector $D_\theta$ for which
\begin{equation*}
\int_0^\infty|f(te^{J\varphi})|\frac{dt}{t}<\infty,
\end{equation*}
for every $J\in\mathbb{S}$ and every angle $\phi\in(-\theta,\theta)\cup(\pi-\theta,\pi+\theta)$. Although the original operator $T$ is unbounded, the decay of the function $f$ is responsible for the resulting operator $f(T)$ in Definition \ref{defi_Omega} again being everywhere defined and bounded. Moreover, in Theorem~\ref{thm_Product_omega}, the important product rule is proven, i.e that for a left hyperholomorphic function $f$ and an intrinsic function $g$ there holds
\begin{equation*}
(gf)(T)=g(T)f(T).
\end{equation*}
An essential result, which makes this functional calculus a good choice, is that if we apply polynomials $p$ (or rational functions $\frac{p}{q}$) we get the expected result. More precisely, Corollary~\ref{cor_Composition_omega}, Proposition~\ref{prop_Product_omega_pf} and Theorem~\ref{thm_Rational_omega} prove that
\begin{equation*}
(p\circ g)(T)=p[g(T)],\qquad(pf)(T)=p[T]f(T),\qquad\text{and}\qquad\Big(\frac{p}{q}\Big)(T)=p[T]q[T]^{-1},
\end{equation*}
where the square brackets $p[T]$ is understood as formally replacing the variable $s$ in the polynomial $p(s)$ by the operator $T$, see also Definition~\ref{defi_Polynomial}. In the final part of this section, in Theorem~\ref{thm_Commutation_omega} and Corollary~\ref{cor_Commutation_omega_bounded} the problem of the commutation of $g(T)$ with other operators is investigated. \medskip

In Section~\ref{sec_Extended}, the $\omega$-functional calculus is extended to functions which admit finite limits at zero and at infinity, i.e. to functions of the form $f(s)=f_\infty+(1+s^2)^{-1}(f_0-f_\infty)+\widetilde{f}(s)$, where $f_0,f_\infty\in\mathbb{R}_n$ and $\widetilde{f}$ is a function which satisfies the assumptions of the $\omega$-functional calculus in Section~\ref{sec_Omega}. The extended $\omega$-functional calculus is then defined as the bounded operator
\begin{equation*}
f(T):=f_\infty+(1+T^2)^{-1}(f_0-f_\infty)+\widetilde{f}(T).
\end{equation*}
All the above mentioned results of the $\omega$-functional calculus in Section~\ref{sec_Omega} are in an adequate way also proven for this extended functional calculus. Moreover, the very important spectral mapping theorem, Theorem~\ref{thm_Spectral_mapping}, gives a close connection between $\sigma_S(g(T))$ and $g(\sigma_S(T))$. While for bounded operators $T\in\mathcal{B}(V)$, this connection is simply $\sigma_S(g(T))=g(\sigma_S(T))$, see \cite[Theorem 4.2.1]{CGK}, the limit points of $g$ at $0$ and at $\infty$ do not fit into this elegant equality in the case of bisectorial operators. \medskip

In Section~\ref{sec_Hinfty}, the extended $\omega$-functional calculus is finally generalized to the $H^\infty$-functional calculus of regularizable functions. These are all left slice hyperholomorphic functions $f$, for which there exists some intrinsic function $e$ which decays fast enough at $0$ and at $\infty$, such that for $e$ and for $ef$, the extended $\omega$-functional calculus of Definition~\ref{defi_Extended} is applicable. Moreover, $e(T)$ is assumed to be injective. For those functions the $H^\infty$-functional calculus is defined as
\begin{equation*}
f(T):=e(T)^{-1}(ef)(T),
\end{equation*}
and clearly gives an unbounded operator because $e(T)^{-1}$ is not everywhere defined. In particular, functions which are polynomially growing at infinity are regularizable, and if $T$ is injective also polynomial growth at zero is allowed, see Proposition~\ref{prop_Polynomial_functions}. Also here, the basic properties like product rule in Theorem~\ref{thm_Product_Hinfty}, the connection to the polynomial functional calculus in Corollary~\ref{cor_Composition_Hinfty}, Theorem~\ref{thm_Rational_Hinfty} and Proposition~\ref{prop_Product_Hinfty_pf}, as well as the commutation relations in Corollary~\ref{cor_Commutation_Hinfty} are proven for the $H^\infty$-functional calculus. Note that many identities which hold with equality for the $\omega$- and the extended functional calculus, are only operator inclusions for the $H\infty$-functional calculus. This is due to the fact that $f(T)$ is unbounded here, and the operator domain sometimes shrink under certain manipulations. A big missing part is the extension of the spectral mapping theorem to $H^\infty$-functional calculus, where it is not known whether this is possible.

\section{Preliminaries on Clifford Algebras and Clifford modules}

In this section we will fix the algebraic and functional analytic setting of this paper. In Section~\ref{sec_Slice_hyperholomorphic_functions} we introduce the notion of slice hyperholomorphic functions with values in the Clifford algebra $\mathbb{R}_n$. This algebra will also be the foundation on which the Banach modules and the $S$-spectrum will be build on in Section~\ref{sec_Clifford_modules}. \medskip

\subsection{Slice hyperholomorphic functions}\label{sec_Slice_hyperholomorphic_functions}

Let $\mathbb{R}_n$ be the real {\it Clifford algebra} over $n$ {\it imaginary units} $e_1,\dots,e_n$ which satisfy the relations
\begin{equation*}
e_i^2=-1\qquad\text{and}\qquad e_ie_j=-e_je_i,\qquad i\neq j\in\{1,\dots,n\}.
\end{equation*}
More precisely, $\mathbb{R}_n$ is given by
\begin{equation*}
\mathbb{R}_n:=\Big\{\sum\nolimits_{A\in\mathcal{A}}s_Ae_A\;\Big|\;s_A\in\mathbb{R},\,A\in\mathcal{A}\Big\},
\end{equation*}
using the index set
\begin{equation*}
\mathcal{A}:=\big\{(i_1,\dots,i_r)\;\big|\;r\in\{0,\dots,n\},\,1\leq i_1<\dots<i_r\leq n\big\},
\end{equation*}
and the {\it basis vectors} $e_A:=e_{i_1}\dots e_{i_r}$. Note that for $A=\emptyset$ the empty product of imaginary units is the real number $e_\emptyset:=1$. Furthermore, we define for every Clifford number $s\in\mathbb{R}_n$ its {\it conjugate} and its {\it absolute value} by
\begin{equation}\label{Eq_Conjugate_Norm}
\overline{s}:=\sum\nolimits_{A\in\mathcal{A}}(-1)^{\frac{|A|(|A|+1)}{2}}s_Ae_A\qquad\text{and}\qquad|s|^2:=\sum\nolimits_{A\in\mathcal{A}}s_A^2.
\end{equation}
An important subset of the Clifford numbers are the so called {\it paravectors}
\begin{equation*}
\mathbb{R}^{n+1}:=\big\{s_0+s_1e_1+\dots+s_ne_n\;\big|\;s_0,s_1,\dots,s_n\in\mathbb{R}\big\}.
\end{equation*}
For any paravector $s\in\mathbb{R}^{n+1}$, we define $\Im(s):=s_1e_1+\dots+s_ne_n$, and the conjugate and the modulus in \eqref{Eq_Conjugate_Norm} reduce to
\begin{equation*}
\overline{s}=s_0-s_1e_1-\dots-s_ne_n\qquad\text{and}\qquad|s|^2=s_0^2+s_1^2+\dots+s_n^2.
\end{equation*}
The sphere of purely imaginary paravectors with modulus $1$ is defined by
\begin{equation}\label{Eq_S}
\mathbb{S}:=\big\{s\in\mathbb{R}^{n+1}\;\big|\;s_0=0,\,|s|=1\big\}.
\end{equation}
Any element $J\in\mathbb{S}$ satisfies $J^2=-1$ and hence the corresponding hyperplane
\begin{equation*}
\mathbb{C}_J:=\big\{x+Jy\;\big|\;x,y\in\mathbb{R}\big\}
\end{equation*}
is an isomorphic copy of the complex numbers. Moreover, for every paravector $s\in\mathbb{R}^{n+1}$ we consider the corresponding {\it $(n-1)$--sphere}
\begin{equation*}
[s]:=\big\{x_0+J|\Im(s)|\;\big|\;J\in\mathbb{S}\big\}.
\end{equation*}
A subset $U\subseteq\mathbb{R}^{n+1}$ is called {\it axially symmetric}, if $[s]\subseteq U$ for every $s\in U$. \medskip

Next, we introduce the notion of slice hyperholomorphic functions $f:U\rightarrow\mathbb{R}_n$, defined on an axially symmetric open set $U\subseteq\mathbb{R}^{n+1}$.

\begin{defi}[Slice hyperholomorphic functions]\label{defi_Slice_hyperholomorphic_functions}
Let $U\subseteq\mathbb{R}^{n+1}$ be open, axially symmetric and consider
\begin{equation*}
\mathcal{U}:=\big\{(x,y)\in\mathbb{R}^2\;\big|\;x+\mathbb{S}y\subseteq U\big\}.
\end{equation*}
A function $f:U\rightarrow\mathbb{R}_n$ is called {\it left} (resp. {\it right}) {\it slice hyperholomorphic}, if there exist continuously differentiable functions $f_0,f_1:\mathcal{U}\rightarrow\mathbb{R}_n$, such that for every $(x,y)\in\mathcal{U}$:

\begin{enumerate}
\item[i)] The function $f$ admits for every $J\in\mathbb{S}$ the representation
\begin{equation}\label{Eq_Holomorphic_decomposition}
f(x+Jy)=f_0(x,y)+Jf_1(x,y),\quad\Big(\text{resp.}\;f(x+Jy)=f_0(x,y)+f_1(x,y)J\Big).
\end{equation}

\item[ii)] The functions $f_0,f_1$ satisfy the {\it compatibility conditions}
\begin{equation}\label{Eq_Symmetry_condition}
f_0(x,-y)=f_0(x,y)\qquad\text{and}\qquad f_1(x,-y)=-f_1(x,y).
\end{equation}

\item[iii)] The functions $f_0,f_1$ satisfy the {\it Cauchy-Riemann equations}
\begin{equation*}
\frac{\partial}{\partial x}f_0(x,y)=\frac{\partial}{\partial y}f_1(x,y)\qquad\text{and}\qquad\frac{\partial}{\partial y}f_0(x,y)=-\frac{\partial}{\partial x}f_1(x,y).
\end{equation*}
\end{enumerate}

The class of left (resp. right) slice hyperholomorphic functions on $U$ is denoted by $\mathcal{SH}_L(U)$ (resp. $\mathcal{SH}_R(U)$). In the special case that $f_0$ and $f_1$ are real valued, we call $f$ {\it intrinsic} and the space of intrinsic functions is denoted by $\mathcal{N}(U)$.
\end{defi}

\begin{rem}
A slightly different definition of slice hyperholomorphic functions (equivalent if $U\cap\mathbb{R}\neq\emptyset$), also referred to as slice monogenic functions, is presented in \cite{ColomboSabadiniStruppa2011}, along with the Cauchy formula and its associated properties. From a historical perspective, the quaternionic function theory was originally developed for slice monogenic functions. However, the definition used in this article is the one from \cite{CGK} and is inspired by the Fueter-Sce mapping theorem. It has in particular proven to be very useful in relation to operator theory.
\end{rem}

For slice hyperholomorphic functions we now introduce path integrals. Since it is sufficient to consider paths embedded in only one complex plane $\mathbb{C}_J$, the idea is to reduce it to a classical complex path integral.

\begin{defi}
Let $U\subseteq\mathbb{R}^{n+1}$ be open, axially symmetric and $g\in\mathcal{SH}_R(U)$, $f\in\mathcal{SH}_L(U)$. For $J\in\mathbb{S}$ and a continuously differentiable curve $\gamma:(a,b)\rightarrow U\cap\mathbb{C}_J$, we define the integral
\begin{equation}\label{Eq_Path_integral}
\int_\gamma g(s)ds_Jf(s):=\int_a^bg(\gamma(t))\frac{\gamma'(t)}{J}f(\gamma(t))dt.
\end{equation}
In the case that $a,b$ are $\infty$ or $\gamma(a)$, $\gamma(b)$ lie on the boundary $\partial U\cap\mathbb{C}_J$, the functions $g,f$ need to satisfy certain decay properties in order for the integral to exist.
\end{defi}

Let us now consider for every $s,p\in\mathbb{R}^{n+1}$ with $s\notin[p]$, the {\it left} and the {\it right Cauchy kernel}
\begin{equation}\label{Eq_Cauchy_kernel}
S_L^{-1}(s,p):=Q_s(p)^{-1}(\overline{s}-p)\qquad\text{and}\qquad S_R^{-1}(s,p):=(\overline{s}-p)Q_s(p)^{-1},
\end{equation}
with the polynomial function $Q_s(p):=p^2-2s_0p+|s|^2$. Then it turns out that

\begin{itemize}
\item[$\circ$] $S_L^{-1}(s,p)$ is right slice hyperholomorphic in $s$ and left slice hyperholomorphic in $s$;
\item[$\circ$] $S_R^{-1}(s,p)$ is left slice hyperholomorphic is $s$ and right slice hyperholomorphic in $p$.
\end{itemize}

With these functions we now recall from \cite[Lemma 2.1.27 \& Theorem 2.1.32]{CGK} the following version of the Cauchy integral formula in the Clifford algebra setting.

\begin{thm}[Cauchy integral formula]
Let $U\subseteq\mathbb{R}^{n+1}$ be open, axially symmetric and $f\in\mathcal{SH}_L(U)$ (resp. $f\in\mathcal{SH}_R(U)$). Then for every $p\in U$ there holds
\begin{subequations}
\begin{align}
f(p)&=\frac{1}{2\pi}\int_{\partial U'\cap\mathbb{C}_J}S_L^{-1}(s,p)ds_Jf(s), \label{Eq_Cauchy_integral_formula_left} \\
\bigg(\text{resp.}\quad f(p)&=\frac{1}{2\pi}\int_{\partial U'\cap\mathbb{C}_J}f(s)ds_JS_R^{-1}(s,p)\bigg). \label{Eq_Cauchy_integral_formula_right}
\end{align}
\end{subequations}
In these integrals $J\in\mathbb{S}$ is arbitrary and $U'$ is any subset with $\overline{U'}\subseteq U$ and where the boundary $\partial U'\cap\mathbb{C}_J$ consists of finitely many continuously differentiable curves.
\end{thm}

\subsection{Clifford modules}\label{sec_Clifford_modules}

For a real Banach space $V_\mathbb{R}$ with norm $\Vert\cdot\Vert_\mathbb{R}$, we define the corresponding {\it Clifford module} by
\begin{equation*}
V:=\Big\{\sum\nolimits_{A\in\mathcal{A}}v_A\otimes e_A\;\Big|\;v_A\in V_\mathbb{R},\,A\in\mathcal{A}\Big\},
\end{equation*}
and equip it with the norm
\begin{equation*}
\Vert v\Vert^2:=\sum\nolimits_{A\in\mathcal{A}}\Vert v_A\Vert_\mathbb{R}^2,\qquad v\in V.
\end{equation*}
For any vector $v=\sum_{A\in\mathcal{A}}v_A\otimes e_A\in V$ and any Clifford number $s=\sum_{B\in\mathcal{A}}s_Be_B\in\mathbb{R}_n$, we establish the left and the right scalar multiplication
\begin{align*}
sv:=&\sum\nolimits_{A,B\in\mathcal{A}}(s_Bv_A)\otimes(e_Be_A), && \textit{(left-multiplication)} \\
vs:=&\sum\nolimits_{A,B\in\mathcal{A}}(v_As_B)\otimes(e_Ae_B). && \textit{(right-multiplication)}
\end{align*}
From \cite[Lemma 2.1]{CMS24} we recall the well known properties of these products
\begin{subequations}
\begin{align}
\Vert sv\Vert&\leq 2^{\frac{n}{2}}|s|\Vert v\Vert\qquad\text{and}\qquad\Vert vs\Vert\leq 2^{\frac{n}{2}}|s|\Vert v\Vert,\qquad\text{if }s\in\mathbb{R}_n, \label{Eq_Norm_estimate_Rn} \\
\Vert sv\Vert&=\Vert vs\Vert=|s|\Vert v\Vert,\hspace{4.83cm}\text{if }s\in\mathbb{R}^{n+1}. \label{Eq_Norm_estimate_paravector}
\end{align}
\end{subequations}
Let us now introduce operators between those Clifford modules, starting with closed operators.

\begin{defi}[Closed operators]
Let $V$ be a Clifford module. A right-linear operator $T:\dom(T)\rightarrow V$ with right-linear domain $\dom(T)\subseteq V$ is called {\it closed}, if for every sequence $(v_n)_{n\in\mathbb{N}}\in\dom(T)$ and any $v,w\in V$, there holds:
\begin{equation*}
\text{If }\lim_{n\rightarrow\infty}v_n=v\text{ and }\lim_{n\rightarrow\infty}Tv_n=w\text{, then }v\in\dom(T)\text{ and }Tv=w.
\end{equation*}
We will denote the set of closed right-linear operators by $\mathcal{K}(V)$.
\end{defi}

For closed operators $T\in\mathcal{K}(V)$ and Clifford numbers $s\in\mathbb{R}_n$, we define
\begin{subequations}
\begin{align}
(sT)(v):=&s(Tv),\qquad v\in\dom(sT):=\dom(T), \label{Eq_Operator_left_multiplication} \\
(Ts)(v):=&T(sv),\qquad v\in\dom(Ts):=\big\{v\in V\;\big|\;sv\in\dom(T)\big\}. \label{Eq_Operator_right_multiplication}
\end{align}
\end{subequations}
Next we consider bounded operators, a particular subclass of closed operators.

\begin{defi}[Bounded operators]
Let $V$ be a Clifford module. An everywhere defined right-linear operator $B:V\rightarrow V$ is called {\it bounded}, if its {\it operator norm}
\begin{equation*}
\Vert B\Vert:=\sup\limits_{0\neq v\in V}\frac{\Vert Bv\Vert}{\Vert v\Vert}<\infty
\end{equation*}
is finite. We will denote the set of all bounded right-linear operators by $\mathcal{B}(V)$.
\end{defi}

Due to \eqref{Eq_Norm_estimate_Rn} and \eqref{Eq_Norm_estimate_paravector}, the operator norms of the products \eqref{Eq_Operator_left_multiplication} and \eqref{Eq_Operator_right_multiplication}, satisfy
\begin{align*}
\Vert sB\Vert&\leq 2^{\frac{n}{2}}|s|\Vert B\Vert\qquad\text{and}\qquad\Vert Bs\Vert\leq 2^{\frac{n}{2}}|s|\Vert B\Vert,\qquad\text{if }s\in\mathbb{R}_n, \\
\Vert sB\Vert&=\Vert Bs\Vert=|s|\Vert B\Vert,\hspace{4.95cm}\text{if }s\in\mathbb{R}^{n+1}.
\end{align*}
We also recall the notion of {\it operator inclusion}: For $T,S\in\mathcal{K}(V)$ we say $T\subseteq S$ if and only if
\begin{equation}\label{Eq_Operator_inclusion}
\dom(T)\subseteq\dom(S)\qquad\text{and}\qquad Tv=Sv,\qquad v\in\dom(T).
\end{equation}
In complex Banach spaces the spectrum is connected to the bounded invertibility of the operator $A-\lambda$. However, in the noncommutative Clifford setting it is due to the explicit value of the series expansion \eqref{Eq_Resolvent_sum}, and the discussion afterwords, that it is natural to connect the spectrum with the bounded invertibility of the operator
\begin{equation*}
Q_s[T]:=T^2-2s_0T+|s|^2,\qquad\text{with }\dom(Q_s[T])=\dom(T^2).
\end{equation*}
This suggests the following definition of the $S$-spectrum.

\begin{defi}[$S$-spectrum]\label{defi_S_spectrum}
For every $T\in\mathcal{K}(V)$ we define the {\it $S$-resolvent set} and the {\it $S$-spectrum} by
\begin{equation*}
\rho_S(T):=\big\{s\in\mathbb{R}^{n+1}\;\big|\;Q_s[T]^{-1}\in\mathcal{B}(V)\big\}\qquad\text{and}\qquad\sigma_S(T):=\mathbb{R}^{n+1}\setminus\rho_S(T).
\end{equation*}
Motivated by \eqref{Eq_Cauchy_kernel}, we define for every $s\in\rho_S(T)$ the {\it left} and the {\it right $S$-resolvent operator}
\begin{equation}\label{Eq_S_resolvent}
S_L^{-1}(s,T):=Q_s[T]^{-1}\overline{s}-TQ_s[T]^{-1}\qquad\text{and}\qquad S_R^{-1}(s,T):=(\overline{s}-T)Q_s[T]^{-1}.
\end{equation}
\end{defi}

In the same way as in the quaternionic setting, one proves also here that $\rho_S(T)$ is axially symmetric open, and $\sigma_S(T)$ is axially symmetric closed, see \cite[Theorem 3.1.6]{FJBOOK}.

\begin{rem}
Note, that although $T$ is a closed operator, this in general does not imply that $Q_s[T]$ is closed as well. However, since for every $s\in\rho_S(T)$ we have $Q_s[T]^{-1}\in\mathcal{B}(V)$, which implies that $Q_s[T]^{-1}$ and consequently $Q_s[T]$ is closed. This means that all $s\in\mathbb{R}^{n+1}$ for which $Q_s[T]$ is not closed, are necessarily contained in $\sigma_S(T)$.
\end{rem}

\begin{rem}
In \eqref{Eq_Cauchy_kernel} there clearly commutes $Q_s(p)^{-1}p=pQ_s(p)^{-1}$. However, if we replace $p\rightarrow T$, the operators $Q_s[T]^{-1}T\neq TQ_s[T]^{-1}$ are no longer the same, since the left hand side is defined on $\dom(T)$ and the right hand side on all of $V$. Hence, in order to have $S_L^{-1}(s,T)$ everywhere defined, we interchange $Q_s(p)^{-1}$ and $p$ in \eqref{Eq_Cauchy_kernel} before we replace $p\rightarrow T$.
\end{rem}

\begin{rem}
The main differences between classical spectral theory and the $S$-spectrum in terms of holomorphicity is that the mapping $\rho_S(T)\ni s\mapsto Q_s[T]^{-1}$ is not slice hyperholomorphic in the sense of Definition~\ref{defi_Slice_hyperholomorphic_functions}. This operator $Q_s[T]$ is only used for the definition of the $S$-spectrum. The $S$-resolvent operators $S_L^{-1}(s,T)$ and $S_R^{-1}(s,T)$ in \eqref{Eq_S_resolvent} are then the ones which are holomorphic and can be used for the $S$-functional calculus, see \eqref{Eq_S_functional_calculus}.
\end{rem}

The advantage of the bounded $S$-functional calculus \eqref{Eq_S_functional_calculus}, is that the integration path $\partial U\cap\mathbb{C}_J$ is compact and there are no problems regarding the convergence of the integral. However, for unbounded operators $T\in\mathcal{K}(V)$, the respective $S$-spectrum is potentially unbounded as well. So, integration paths which surround $\sigma_S(T)$ are not compact any more, and we have to ensure the convergence of the integral by a certain decay of the integrand. As a result, this cannot be done for any unbounded operator, but requires additional assumptions on the one hand on the location of the spectrum (to allow certain integration paths), but also on the decay of the resolvent operator (to make the integral converge). In this article we will concentrate on the class of bisectorial operators, which will be defined in the following. For every $\omega\in(0,\frac{\pi}{2})$, we consider the {\it double sector} \medskip

\begin{minipage}{0.34\textwidth}
\begin{center}
\begin{tikzpicture}
\fill[black!30] (0,0)--(1.63,0.76) arc (25:-25:1.8)--(0,0)--(-1.63,-0.76) arc (205:155:1.8);
\draw (-1.63,0.76)--(1.63,-0.76);
\draw (-1.63,-0.76)--(1.63,0.76);
\draw[->] (-2.1,0)--(2.1,0);
\draw[->] (0,-0.7)--(0,0.7);
\draw (1,0) arc (0:25:1) (0.8,-0.07) node[anchor=south] {\small{$\omega$}};
\draw (-1.3,0) node[anchor=south] {\small{$D_\omega$}};
\end{tikzpicture}
\end{center}
\end{minipage}
\begin{minipage}{0.65\textwidth}
\begin{equation}\label{Eq_Domega}
D_\omega:=\big\{s\in\mathbb{R}^{n+1}\setminus\{0\}\;\big|\;\Arg(s)\in I_\omega\big\},
\end{equation}
using the union of intervals
\begin{equation*}
I_\omega:=(-\omega,\omega)\cup(\pi-\omega,\pi+\omega).
\end{equation*}
\end{minipage}

\medskip Here, $\Arg(s)$ is the usual argument of $s$ treated as a complex number in $\mathbb{C}_J$.

\begin{defi}[Bisectorial operators]\label{defi_Bisectorial_operators}
An operator $T\in\mathcal{K}(V)$ is called {\it bisectorial of angle} $\omega\in(0,\frac{\pi}{2})$, if its $S$-spectrum is contained in the closed double sector
\begin{equation*}
\sigma_S(T)\subseteq\overline{D_\omega},
\end{equation*}
and for every $\varphi\in(\omega,\frac{\pi}{2})$ there exists $C_\varphi\geq 0$, such that the left $S$-resolvent \eqref{Eq_S_resolvent} satisfies
\begin{equation}\label{Eq_SL_estimate}
\Vert S_L^{-1}(s,T)\Vert\leq\frac{C_\varphi}{|s|},\qquad s\in\mathbb{R}^{n+1}\setminus(D_\varphi\cup\{0\}).
\end{equation}
\end{defi}

We will now show that the estimate \eqref{Eq_SL_estimate} on the left $S$-resolvent operator implies a similar estimate on the right $S$-resolvent operator.

\begin{lem}
Let $T\in\mathcal{K}(V)$ be bisectorial of angle $\omega\in(0,\frac{\pi}{2})$. Then for every $\varphi\in(\omega,\frac{\pi}{2})$, and with the constant $C_\varphi$ from \eqref{Eq_SL_estimate}, there holds
\begin{equation}\label{Eq_SR_estimate}
\Vert S_R^{-1}(s,T)\Vert\leq\frac{2C_\varphi}{|s|},\qquad s\in\mathbb{R}^{n+1}\setminus(D_\varphi\cup\{0\}),
\end{equation}
\end{lem}

\begin{proof}
For every $s\in\mathbb{R}^{n+1}\setminus(D_\varphi\cup\{0\})$ let us choose $J\in\mathbb{S}$ such that $s\in\mathbb{C}_J$. Then we can decompose the right $S$-resolvent in terms of the left $S$-resolvent in the way
\begin{align*}
S_R^{-1}(s,T)&=(\overline{s}-T)Q_s[T]^{-1}=Q_s[T]^{-1}\frac{s+\overline{s}}{2}-TQ_s[T]^{-1}-JQ_s[T]^{-1}\frac{s-\overline{s}}{2J} \\
&=\frac{1}{2}\Big(Q_s[T]^{-1}(s+\overline{s})-2TQ_s[T]^{-1}\Big)+J\Big(Q_s[T]^{-1}(\overline{s}-s)+(T-T)Q_s[T]^{-1}\Big)\frac{1}{2J} \\
&=\frac{1}{2}\big(S_L^{-1}(s,T)+S_L^{-1}(\overline{s},T)\big)+J\big(S_L^{-1}(s,T)-S_L^{-1}(\overline{s},T)\big)\frac{1}{2J},
\end{align*}
where in the second equation we used that $\frac{s+\overline{s}}{2},\frac{s-\overline{s}}{2J}\in\mathbb{R}$ are real and hence commute with $Q_s[T]^{-1}$. In this form we can estimate
\begin{equation*}
\Vert S_R^{-1}(s,T)\Vert\leq\frac{1}{2}\Big(\frac{C_\varphi}{|s|}+\frac{C_\varphi}{|\overline{s}|}\Big)+\frac{1}{2}\Big(\frac{C_\varphi}{|s|}+\frac{C_\varphi}{|\overline{s}|}\Big)=\frac{2C_\varphi}{|s|}. \qedhere
\end{equation*}
\end{proof}

\section{The $\omega$-functional calculus}\label{sec_Omega}

The aim of this section is to generalize the bounded functional calculus \eqref{Eq_S_functional_calculus} to the unbounded bisectorial operators $T$ from Definition~\ref{defi_Bisectorial_operators}. Since we have to integrate around the unbounded $S$-spectrum of $T$ we need certain decay assumptions on the involved function $f$ in order to make the integral converge.

\begin{defi}\label{defi_SH_0}
For every $\theta\in(0,\frac{\pi}{2})$ let $D_\theta$ be the double sector from \eqref{Eq_Domega}. Then we define the function spaces
\begin{align*}
\text{i)}\;\;&\mathcal{SH}_L^0(D_\theta):=\bigg\{f\in\mathcal{SH}_L(D_\theta)\;\bigg|\;f\text{ is bounded},\;\int_0^\infty|f(te^{J\phi})|\frac{dt}{t}<\infty,\;J\in\mathbb{S},\phi\in I_\theta\bigg\}, \\
\text{ii)}\;\;&\mathcal{N}^0(D_\theta):=\bigg\{g\in\mathcal{N}(D_\theta)\;\bigg|\;g\text{ is bounded},\;\int_0^\infty|g(te^{J\phi})|\frac{dt}{t}<\infty,\;J\in\mathbb{S},\phi\in I_\theta\bigg\}.
\end{align*}
\end{defi}

The next lemma proves how functions in the space $\mathcal{SH}_L^0(D_\theta)$ behave at zero and at infinity. The basic idea is taken from \cite[Lemma 9.4]{H18}.

\begin{lem}\label{lem_f_convergence}
Let $\theta\in(0,\frac{\pi}{2})$ and $f\in\mathcal{SH}_L^0(D_\theta)$. Then for every $\varphi\in(0,\theta)$, there holds
\begin{equation}\label{Eq_f_convergence}
\lim\limits_{t\rightarrow 0^+}\sup\limits_{J\in\mathbb{S},\phi\in I_\varphi}|f(te^{J\phi})|=0\qquad\text{and}\qquad
\lim\limits_{t\rightarrow\infty}\sup\limits_{J\in\mathbb{S},\phi\in I_\varphi}|f(te^{J\phi})|=0.
\end{equation}
\end{lem}

\begin{proof}
Note, that since $p\mapsto e^p$ is intrinsic on $\mathbb{R}^{n+1}$, the function $s\mapsto f(e^s)$ is left slice hyperholomorphic for $s\in\mathbb{R}^{n+1}$ with $|\Im(s)|<\theta$. Choosing now some $\psi\in(\varphi,\theta)$, the Cauchy integral formula \eqref{Eq_Cauchy_integral_formula_left} provides for any $t>0$, $\phi\in(-\varphi,\varphi)$ and $J\in\mathbb{S}$ the representation
\begin{equation}\label{Eq_f_convergence_1}
f(te^{J\phi})=f(e^{\ln(t)+J\phi})=\frac{1}{2\pi}\int_{\gamma_+\oplus\kappa_-\oplus\gamma_-\oplus\kappa_+}S_L^{-1}(s,\ln(t)+J\phi)ds_If(e^s),
\end{equation}
where the path integral is in a complex plane $\mathbb{C}_I$, $I\in\mathbb{S}$, along the curves \medskip

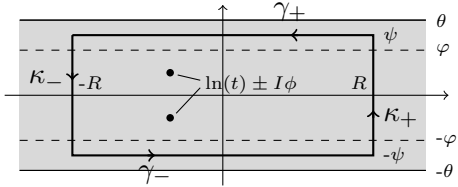
\begin{minipage}{0.44\textwidth}
\begin{center}
\begin{tikzpicture}
\fill[black!15] (-2.7,1)--(2.7,1)--(2.7,-1)--(-2.7,-1);
\draw (-2.7,1)--(2.7,1);
\draw (-2.7,-1)--(2.7,-1);
\draw[thick] (-2,0.8)--(2,0.8)--(2,-0.8)--(-2,-0.8)--(-2,0.8);
\draw[dashed] (-2.7,0.6)--(2.7,0.6);
\draw[dashed] (-2.7,-0.6)--(2.7,-0.6);
\draw (-2.7,1)--(2.7,1);
\draw (-2.7,-1)--(2.7,-1);
\draw[->] (-2.9,0)--(3,0);
\draw[->] (0,-1.2)--(0,1.2);
\draw[thick,->] (1,0.8)--(0.9,0.8);
\draw[thick,->] (-1,-0.8)--(-0.9,-0.8);
\draw[thick,->] (2,-0.3)--(2,-0.2);
\draw[thick,->] (-2,0.3)--(-2,0.2);
\draw (2,-0.3) node[anchor=west] {$\kappa_+$};
\draw (-1.97,0.2) node[anchor=east] {$\kappa_-$};
\draw (0.9,0.85) node[anchor=south] {$\gamma_+$};
\draw (-0.9,-0.8) node[anchor=north] {$\gamma_-$};
\fill[black] (-0.7,0.3) circle (0.05cm);
\fill[black] (-0.7,-0.3) circle (0.05cm);
\draw (-0.6,0.27)--(-0.3,0.2);
\draw (-0.62,-0.24)--(-0.3,0.1);
\draw (-0.37,0.15) node[anchor=west] {\tiny{$\ln(t)\pm I\phi$}};
\draw (2.05,-0.05) node[anchor=south east] {\tiny{$R$}};
\draw (-2.05,-0.05) node[anchor=south west] {\tiny{-$R$}};
\draw (2.7,0.6) node[anchor=west] {\tiny{$\varphi$}};
\draw (2.7,-0.6) node[anchor=west] {\tiny{-$\varphi$}};
\draw (2,0.8) node[anchor=west] {\tiny{$\psi$}};
\draw (2,-0.8) node[anchor=west] {\tiny{-$\psi$}};
\draw (2.7,1) node[anchor=west] {\tiny{$\theta$}};
\draw (2.7,-1) node[anchor=west] {\tiny{-$\theta$}};
\end{tikzpicture}
\end{center}
\end{minipage}
\begin{minipage}{0.55\textwidth}
\begin{align*}
\gamma_+(x)&:=-x+I\psi,\hspace{0.85cm}x\in(-R,R), \\
\gamma_-(x)&:=x-I\psi,\hspace{1.15cm}x\in(-R,R), \\
\kappa_+(\alpha)&:=R+I\alpha,\hspace{1.07cm}\alpha\in(-\psi,\psi), \\
\kappa_-(\alpha)&:=-R-I\alpha,\qquad\alpha\in(-\psi,\psi).
\end{align*}
\end{minipage}

\medskip Since $f$ is bounded, the integrals along the paths $\kappa_\pm$ vanish as $R\rightarrow\infty$ due to
\begin{align*}
\lim\limits_{R\rightarrow\infty}\bigg|\int_{\kappa_\pm}S_L^{-1}(s,\ln(t)+J\phi)ds_Jf(e^s)\bigg|&=\lim\limits_{R\rightarrow\infty}\bigg|\int_{-\psi}^\psi S_L^{-1}(\pm R\pm I\alpha,\ln(t)+J\phi)f(e^{\pm R\pm I\alpha})d\alpha\bigg| \\
&\leq\lim\limits_{R\rightarrow\infty}\frac{2\psi(|\ln(t)+J\phi|+|R+I\alpha|)}{(|\ln(t)+J\phi|-|R+I\alpha|)^2}\Vert f\Vert_\infty=0,
\end{align*}
where we used that the $S$-resolvent can for every $s\in\mathbb{C}_I$, $p\in\mathbb{C}_J$ be estimated by
\begin{subequations}
\begin{align}
|S_L^{-1}(s,p)|=\big|(p^2-2s_0p+|s|^2)^{-1}(\overline{s}-p)\big|&=\frac{|p-\overline{s}|}{|p-s_J||p-\overline{s_J}|} \label{Eq_SL_estimate_1} \\
&\leq\frac{|p|+|s|}{|p-s_J||p-\overline{s_J}|} \label{Eq_SL_estimate_2} \\
&\leq\frac{|p|+|s|}{(|p|-|s|)^2}, \label{Eq_SL_estimate_3}
\end{align}
\end{subequations}
where $s_J\in[s]\cap\mathbb{C}_J$ is the rotation of $s$ into $\mathbb{C}_J$. Hence in the limit $R\rightarrow\infty$, \eqref{Eq_f_convergence_1} reduces to
\begin{align}
f(te^{J\phi})&=\frac{1}{2\pi}\lim\limits_{R\rightarrow\infty}\int_{\gamma_+\oplus\gamma_-}S_L^{-1}(s,\ln(t)+J\phi)ds_If(e^s) \notag \\
&=\frac{1}{2\pi}\lim\limits_{R\rightarrow\infty}\sum\limits_\pm\int_{-R}^RS_L^{-1}\big(\mp x\pm I\psi,\ln(t)+J\phi\big)(-I)f\big(e^{\mp x\pm I\psi}\big)dx \notag \\
&=\frac{1}{2\pi}\lim\limits_{R\rightarrow\infty}\sum\limits_\pm\int_{e^{-R}}^{e^R}S_L^{-1}\big(\ln(\tau)\pm I\psi,\ln(t)+J\phi\big)(-I)f\big(\tau e^{\pm I\psi}\big)\frac{d\tau}{\tau}, \label{Eq_f_convergence_2}
\end{align}
where in the last line we substituted $\tau=e^{\mp x}$. Using again \eqref{Eq_SL_estimate_1}, we can also estimate the $S$-resolvent operator of this integrand by
\begin{align*}
\big|S_L^{-1}\big(\ln(\tau)\pm I\psi,\ln(t)+J\phi\big)\big|&=\frac{|\ln(t)+J\phi-\ln(\tau)\mp I\psi|}{|\ln(t)+J\phi-\ln(\tau)\mp J\psi||\ln(t)+J\phi-\ln(\tau)\pm J\psi|} \\
&\leq\frac{|\ln(\frac{t}{\tau})|+|\phi|+\psi}{|\ln(\frac{t}{\tau})+J(\phi\mp\psi)||\ln(\frac{t}{\tau})+J(\phi\pm\psi)|} \\
&\leq\frac{|\ln(\frac{t}{\tau})|+\varphi+\psi}{|\ln(\frac{t}{\tau})+J(\psi-\varphi)|^2},
\end{align*}
where in the last line we used $-\varphi<\phi<\varphi<\psi$. This allows us to estimate \eqref{Eq_f_convergence_2} by
\begin{align*}
|f(te^{J\phi})|&\leq\frac{1}{2\pi}\lim\limits_{R\rightarrow\infty}\sum\limits_\pm\int_{e^{-R}}^{e^R}\frac{|\ln(\frac{t}{\tau})|+\varphi+\psi}{|\ln(\frac{t}{\tau})+J(\psi-\varphi)|^2}|f(\tau e^{\pm I\psi})|\frac{d\tau}{\tau} \\
&=\frac{1}{2\pi}\lim\limits_{R\rightarrow\infty}\int_{e^{-R}}^{e^R}\frac{|\ln(\frac{t}{\tau})|+\varphi+\psi}{\ln(\frac{t}{\tau})^2+(\psi-\varphi)^2}\big(|f(\tau e^{I\psi})|+|f(\tau e^{-I\psi})|\big)\frac{d\tau}{\tau} \\
&=\frac{1}{2\pi}\int_0^\infty\frac{|\ln(\frac{t}{\tau})|+\varphi+\psi}{\ln(\frac{t}{\tau})^2+(\psi-\varphi)^2}\big(|f(\tau e^{I\psi})|+|f(\tau e^{-I\psi})|\big)\frac{d\tau}{\tau}.
\end{align*}
Since the right hand side of this inequality is independent of $J\in\mathbb{S}$ and $\phi\in(-\varphi,\varphi)$, it also holds for the respective supremum. Since moreover the first part of this integrand admits the $t$-independent upper bound
\begin{equation*}
\frac{|\ln(\frac{t}{\tau})|+\varphi+\psi}{\ln(\frac{t}{\tau})^2+(\psi-\varphi)^2}\leq\frac{1}{2(\psi-\varphi)}+\frac{\varphi+\psi}{(\psi-\varphi)^2},\qquad t>0,
\end{equation*}
and since the function $f\in\mathcal{SH}_L^0(D_\theta)$ satisfies the integrability conditions of Definition~\ref{defi_SH_0}, we can use the dominated convergence theorem to carry the limits $t\rightarrow 0^+$ and $t\rightarrow\infty$ inside the integral, and get
\begin{equation*}
\lim\limits_{t\rightarrow 0^+}\sup\limits_{J\in\mathbb{S},\phi\in(-\varphi,\varphi)}|f(te^{J\phi})|\leq\frac{1}{2\pi}\int_0^\infty\lim\limits_{t\rightarrow 0^+}\frac{|\ln(\frac{t}{\tau})|+\varphi+\psi}{\ln(\frac{t}{\tau})^2+(\psi-\varphi)^2}\big(|f(\tau e^{I\psi})|+|f(\tau e^{-I\psi})|\big)\frac{d\tau}{\tau}=0,
\end{equation*}
as well as
\begin{equation*}
\lim\limits_{t\rightarrow\infty}\sup\limits_{J\in\mathbb{S},\phi\in(-\varphi,\varphi)}|f(te^{J\phi})|\leq\frac{1}{2\pi}\int_0^\infty\lim\limits_{t\rightarrow\infty}\frac{|\ln(\frac{t}{\tau})|+\varphi+\psi}{\ln(\frac{t}{\tau})^2+(\psi-\varphi)^2}\big(|f(\tau e^{I\psi})|+|f(\tau e^{-I\psi})|\big)\frac{d\tau}{\tau}=0.
\end{equation*}
Hence we have proven the limits where the supremum is taken over $\phi\in(-\varphi,\varphi)$. Analogously one proves the convergence where the supremum is taken over $\phi\in(\pi-\varphi,\pi+\varphi)$, i.e.
\begin{equation*}
\lim\limits_{t\rightarrow 0^+}\sup\limits_{J\in\mathbb{S},\phi\in(\pi-\varphi,\pi+\varphi)}|f(te^{J\phi})|=0\qquad\text{and}\qquad\lim\limits_{t\rightarrow\infty}\sup\limits_{J\in\mathbb{S},\phi\in(\pi-\varphi,\pi+\varphi)}|f(te^{J\phi})|=0.
\end{equation*}
Combining these results then gives the stated convergences \eqref{Eq_f_convergence}.
\end{proof}

\begin{lem}\label{lem_Integral_vanish}
Let $\theta\in(0,\frac{\pi}{2})$, $f\in\mathcal{SH}_L^0(D_\theta)$. Consider some open interval $I\subseteq\overline{I}\subseteq I_\theta$, and a right slice hyperholomorphic function $g:\big\{s\in\mathbb{R}^{n+1}\setminus\{0\}\;\big|\;\Arg(s)\in I\big\}\rightarrow\mathbb{R}_n$. Moreover, for every $r>0$ we consider the curve
\begin{equation*}
\sigma_r(\phi):=re^{J\phi},\qquad\phi\in I.
\end{equation*}
Then the following two statements hold: \medskip

\begin{enumerate}
\item[i)] If $\sup\limits_{\phi\in I}|g(re^{J\phi})|=\mathcal{O}(\frac{1}{r})$ as $r\rightarrow 0^+$, then $\lim\limits_{r\rightarrow 0^+}\int\limits_{\sigma_r}g(s)ds_Jf(s)=0$.
\item[ii)] If $\sup\limits_{\phi\in I}|g(re^{J\phi})|=\mathcal{O}(\frac{1}{r})$ as $r\rightarrow\infty$, then $\lim\limits_{r\rightarrow\infty}\int\limits_{\sigma_r}g(s)ds_Jf(s)=0$.
\end{enumerate}
\end{lem}

\begin{proof}
i)\;\;By assumption there exists some $C_g\geq 0$ and some $r_0>0$, such that
\begin{equation*}
\sup\nolimits_{\phi\in I}|g(re^{J\phi})|\leq\frac{C_g}{r},\qquad r\leq r_0.
\end{equation*}
Hence, from Lemma~\ref{lem_f_convergence} we conclude the convergence
\begin{equation*}
\Big|\int_{\sigma_r}g(s)ds_Jf(s)\Big|=\Big|\int_Ig(re^{J\phi})\frac{re^{J\phi}J}{J}f(re^{J\phi})d\phi\Big|\leq 2^{\frac{n}{2}}|I|\frac{C_g}{r}r\sup\limits_{\phi\in I}|f(re^{J\phi})|\overset{r\rightarrow 0^+}{\longrightarrow}0.
\end{equation*}
ii)\;\;By assumption there exists some $C_g\geq 0$ and some $r_0>0$, such that
\begin{equation*}
\sup\nolimits_{\phi\in I}|g(re^{J\phi})|\leq\frac{C_g}{r},\qquad r\geq r_0.
\end{equation*}
Hence, from Lemma~\ref{lem_f_convergence} we conclude the convergence 
\begin{equation*}
\Big|\int_{\sigma_r}g(s)ds_Jf(s)\Big|=\Big|\int_Ig(re^{J\phi})\frac{re^{J\phi}J}{J}f(re^{J\phi})d\phi\Big|\leq 2^{\frac{n}{2}}|I|\frac{C_g}{r}r\sup\limits_{\phi\in I}|f(re^{J\phi})|\overset{r\rightarrow\infty}{\longrightarrow}0. \qedhere
\end{equation*}
\end{proof}

The next result ensures that the functional calculus in Definition~\ref{defi_Omega} is well defined for the classes of functions in Definition~\ref{defi_SH_0}.

\begin{thm}\label{thm_Independence}
Let $T\in\mathcal{K}(V)$ be bisectorial of angle $\omega\in(0,\frac{\pi}{2})$. Then for every $f\in\mathcal{SH}_L^0(D_\theta)$, $\theta\in(\omega,\frac{\pi}{2})$, the integral
\begin{equation}\label{Eq_Independence}
\int_{\partial D_\varphi\cap\mathbb{C}_J}S_L^{-1}(s,T)ds_Jf(s)
\end{equation}
defines a bounded Clifford operator and it neither depends on the angle $\varphi\in(\omega,\theta)$, nor on the imaginary unit $J\in\mathbb{S}$.
\end{thm}

\begin{proof}
From the estimate \eqref{Eq_SL_estimate} and the definition of the space $\mathcal{SH}_L^0(S_\theta)$ in Definition~\ref{defi_SH_0}, there follows for every $J\in\mathbb{S}$ and $\varphi\in(\omega,\theta)$ the absolute convergence of the integral
\begin{equation*}
\int_0^\infty\bigg\Vert S_L^{-1}(te^{J\varphi},T)\frac{e^{J\varphi}}{J}f(te^{J\varphi})\bigg\Vert dt\leq 2^{\frac{n}{2}}C_\varphi\int_0^\infty\frac{1}{t}|f(te^{J\varphi})|dt<\infty.
\end{equation*}
Since the integral \eqref{Eq_Independence} consists of four such integrals, one with $\varphi$ and three with $\varphi$ replaced by $-\varphi$, $\pi-\varphi$ and $\pi+\varphi$ respectively. This shows that \eqref{Eq_Independence} leads to an everywhere defined bounded operator with operator norm
\begin{equation*}
\bigg\Vert\int_{\partial D_\varphi\cap\mathbb{C}_J}S_L^{-1}(s,T)ds_Jf(s)\bigg\Vert\leq 2^{\frac{n}{2}}C_\varphi\sum\limits_{\phi\in\{\varphi,-\varphi,\pi-\varphi,\pi+\varphi\}}\int_0^\infty|f(te^{J\phi})|\frac{dt}{t}.
\end{equation*}
For the $\varphi$-independence, let $\varphi_1<\varphi_2 \in(\omega,\theta)$. For every $0<\varepsilon<R$ we consider the paths \medskip

\begin{minipage}{0.39\textwidth}
\begin{center}
\begin{tikzpicture}[scale=0.8]
\fill[black!15] (0,0)--(1.25,2.17) arc (60:-60:2.5)--(0,0)--(-1.25,-2.17) arc(240:120:2.5);
\draw (-1.25,2.17)--(1.25,-2.17);
\draw (-1.25,-2.17)--(1.25,2.17);
\fill[black!30] (0,0)--(2.27,1.06) arc (25:-25:2.5)--(0,0)--(-2.27,-1.06) arc (205:155:2.5);
\draw (-2.27,1.06)--(2.27,-1.06);
\draw (-2.27,-1.06)--(2.27,1.06);
\draw[->] (-2.7,0)--(2.9,0);
\draw[->] (0,-2.2)--(0,2.3) node[anchor=north east] {\large{$\mathbb{C}_J$}};
\draw (1,0) arc (0:35:1) (0.75,-0.09) node[anchor=south] {\tiny{$\varphi_1$}};
\draw (1.4,0) arc (0:50:1.4) (1.15,0.06) node[anchor=south] {\tiny{$\varphi_2$}};
\draw[thick] (1.48,1.76)--(0.45,0.54)--(0.58,0.41)--(1.88,1.32) arc (35:50:2.3);
\draw[thick] (1.48,-1.76)--(0.45,-0.54)--(0.58,-0.41)--(1.88,-1.32) arc (-35:-50:2.3);
\draw[thick] (-1.48,1.76)--(-0.45,0.54)--(-0.58,0.41)--(-1.88,1.32) arc (145:130:2.3);
\draw[thick] (-1.48,-1.76)--(-0.45,-0.54)--(-0.58,-0.41)--(-1.88,-1.32) arc (215:230:2.3);
\draw[thick,->] (0.58,0.41)--(0.45,0.54);
\draw[thick,->] (0.45,-0.54)--(0.58,-0.41);
\draw[thick,->] (-0.45,0.54)--(-0.58,0.41);
\draw[thick,->] (-0.58,-0.41)--(-0.45,-0.54);
\draw (0.6,0.56) node[anchor=north east] {\tiny{$\sigma$}};
\draw (0.6,-0.55) node[anchor=south east] {\tiny{$\sigma$}};
\draw (-0.6,0.55) node[anchor=north west] {\tiny{$\sigma$}};
\draw (-0.6,-0.58) node[anchor=south west] {\tiny{$\sigma$}};
\draw[thick,->] (1.85,1.3)--(1.28,0.89);
\draw[thick,->] (0.58,-0.41)--(1.28,-0.89);
\draw[thick,->] (1.45,1.73)--(1,1.2);
\draw[thick,->] (0.45,-0.54)--(1,-1.2);
\draw[thick,->] (-1.85,-1.3)--(-1.28,-0.89);
\draw[thick,->] (-0.58,0.41)--(-1.28,0.89);
\draw[thick,->] (-1.45,-1.73)--(-1,-1.2);
\draw[thick,->] (-0.45,0.54)--(-1,1.2);
\draw (1.2,0.7) node[anchor=west] {\tiny{$\gamma_{1,+}$}};
\draw (1,-0.7) node[anchor=west] {\tiny{$\gamma_{1,+}$}};
\draw (-0.9,0.7) node[anchor=east] {\tiny{$\gamma_{1,-}$}};
\draw (-1,-0.75) node[anchor=east] {\tiny{$\gamma_{1,-}$}};
\draw (1.1,1.25) node[anchor=east] {\tiny{$\gamma_{2,+}$}};
\draw (1.1,-1.25) node[anchor=east] {\tiny{$\gamma_{2,+}$}};
\draw (-1,1.2) node[anchor=west] {\tiny{$\gamma_{2,-}$}};
\draw (-1.3,-1.5) node[anchor=west] {\tiny{$\gamma_{2,-}$}};
\draw[thick,->] (1.68,1.57)--(1.625,1.625);
\draw[thick,->] (1.725,-1.525)--(1.78,-1.47);
\draw[thick,->] (-1.725,1.525)--(-1.78,1.47);
\draw[thick,->] (-1.68,-1.57)--(-1.625,-1.625);
\draw (1.68,1.57) node[anchor=west] {\tiny{$\tau$}};
\draw (1.64,-1.59) node[anchor=west] {\tiny{$\tau$}};
\draw (-1.64,1.59) node[anchor=east] {\tiny{$\tau$}};
\draw (-1.68,-1.57) node[anchor=east] {\tiny{$\tau$}};
\end{tikzpicture}
\end{center}
\end{minipage}
\begin{minipage}{0.6\textwidth}
\begin{align*}
\sigma(\varphi):=&\varepsilon e^{J\varphi},\hspace{2.52cm}\varphi\in I_{\varphi_2}\setminus\overline{I_{\varphi_1}}, \\
\tau(\varphi):=&Re^{J\varphi},\hspace{2.4cm}\varphi\in I_{\varphi_2}\setminus\overline{I_{\varphi_1}}, \\
\gamma_{k,\pm}(t):=&\pm|t|e^{-\sgn(t)J\varphi_k},\qquad t\in(-R,R)\setminus[-\varepsilon,\varepsilon], \\
\gamma_k:=&\gamma_{k,+}\oplus\gamma_{k,-},\hspace{1.45cm} k\in\{1,2\}.
\end{align*}
\end{minipage}

\medskip Since there is no spectrum in the interior of the path $\gamma_2\ominus\sigma\ominus\gamma_1\oplus\tau$, the following integrand is holomorphic and there vanishes the integral
\begin{equation*}
\int_{\gamma_2\ominus\sigma\ominus\gamma_1\oplus\tau}S_L^{-1}(s,T)ds_Jf(s)=0,
\end{equation*}
where the symbol $\oplus$ (resp. $\ominus$) means the sum (resp. difference) of integrals along the respective paths. However, since $\Vert S_L^{-1}(s,T)\Vert\leq\frac{C_\varphi}{|s|}$ by Definition~\ref{defi_Bisectorial_operators}, the integrals along $\sigma$ and $\tau$ vanish in the limits $\varepsilon\rightarrow 0^+$ and $R\rightarrow\infty$ by Lemma~\ref{lem_Integral_vanish}. Hence we end up with
\begin{equation*}
\lim\limits_{\varepsilon\rightarrow 0^+}\lim\limits_{R\rightarrow\infty}\int_{\gamma_2\ominus\gamma_1}S_L^{-1}(s,T)ds_Jf(s)=0,
\end{equation*}
which can be rewritten as the independence of the angle
\begin{equation*}
\int_{\partial D_{\varphi_1}\cap\mathbb{C}_J}S_L^{-1}(s,T)ds_Jf(s)=\int_{\partial D_{\varphi_2}\cap\mathbb{C}_J}S_L^{-1}(s,T)ds_Jf(s).
\end{equation*}
For the independence of the imaginary unit, let $I,J\in\mathbb{S}$ and $\varphi_1<\varphi_2<\varphi_3\in(\omega,\theta)$. For every $\varepsilon>0$ we define the paths \medskip

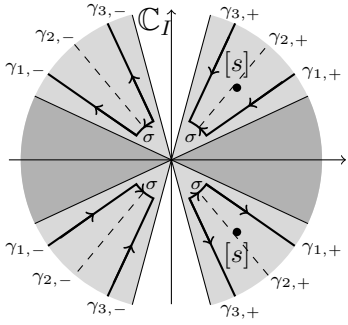
\begin{minipage}{0.39\textwidth}
\begin{center}
\begin{tikzpicture}[scale=0.8]
\fill[black!15] (0,0)--(0.65,2.41) arc (75:-75:2.5)--(0,0)--(-0.65,-2.41) arc (255:105:2.5);
\draw (-0.65,-2.41)--(0.65,2.41);
\draw (-0.65,2.41)--(0.65,-2.41);
\fill[black!30] (0,0)--(2.27,1.06) arc (25:-25:2.5)--(0,0)--(-2.27,-1.06) arc (205:155:2.5);
\draw (-2.27,-1.06)--(2.27,1.06);
\draw (-2.27,1.06)--(2.27,-1.06);
\draw[->] (-2.7,0)--(2.9,0);
\draw[->] (0,-2.4)--(0,2.5);
\draw (0.2,2.3) node[anchor=east] {\large{$\mathbb{C}_I$}};
\draw[thick] (1.06,2.27)--(0.3,0.64) arc (65:35:0.71)--(2.05,1.43);
\draw[thick] (-1.06,2.27)--(-0.3,0.64) arc (115:145:0.71)--(-2.05,1.43);
\draw[thick] (1.06,-2.27)--(0.3,-0.64) arc (-65:-35:0.71)--(2.05,-1.43);
\draw[thick] (-1.06,-2.27)--(-0.3,-0.64) arc (245:215:0.71)--(-2.05,-1.43);
\draw[thick,->] (0.47,0.53)--(0.46,0.54);
\draw[thick,->] (0.45,-0.55)--(0.46,-0.54);
\draw[thick,->] (1.85,1.3)--(1.28,0.89);
\draw[thick,->] (0.58,-0.41)--(1.28,-0.89);
\draw[thick,->] (0.96,2.05)--(0.66,1.41);
\draw[thick,->] (0.3,-0.64)--(0.66,-1.41);
\draw[thick,->] (-0.47,-0.53)--(-0.46,-0.54);
\draw[thick,->] (-0.45,0.55)--(-0.46,0.54);
\draw[thick,->] (-1.85,-1.3)--(-1.28,-0.89);
\draw[thick,->] (-0.58,0.41)--(-1.28,0.89);
\draw[thick,->] (-0.96,-2.05)--(-0.66,-1.41);
\draw[thick,->] (-0.3,0.64)--(-0.66,1.41);
\draw (2,1.43) node[anchor=west] {\tiny{$\gamma_{1,+}$}};
\draw (2,-1.5) node[anchor=west] {\tiny{$\gamma_{1,+}$}};
\draw (-1.9,1.5) node[anchor=east] {\tiny{$\gamma_{1,-}$}};
\draw (-1.9,-1.5) node[anchor=east] {\tiny{$\gamma_{1,-}$}};
\draw (1.45,2) node[anchor=west] {\tiny{$\gamma_{2,+}$}};
\draw (1.48,-2.053) node[anchor=west] {\tiny{$\gamma_{2,+}$}};
\draw (-1.4,2.05) node[anchor=east] {\tiny{$\gamma_{2,-}$}};
\draw (-1.45,-2) node[anchor=east] {\tiny{$\gamma_{2,-}$}};
\draw (1.15,2.15) node[anchor=south] {\tiny{$\gamma_{3,+}$}};
\draw (1.15,-2.25) node[anchor=north] {\tiny{$\gamma_{3,+}$}};
\draw (-1.06,2.2) node[anchor=south] {\tiny{$\gamma_{3,-}$}};
\draw (-1.06,-2.2) node[anchor=north] {\tiny{$\gamma_{3,-}$}};
\draw (0.6,0.6) node[anchor=north east] {\tiny{$\sigma$}};
\draw (0.7,-0.65) node[anchor=south east] {\tiny{$\sigma$}};
\draw (-0.65,0.6) node[anchor=north west] {\tiny{$\sigma$}};
\draw (-0.6,-0.65) node[anchor=south west] {\tiny{$\sigma$}};
\draw[dashed] (1.61,-1.92)--(0.46,-0.54);
\draw[dashed] (0.46,0.54)--(1.61,1.92);
\draw[dashed] (-1.61,1.92)--(-0.46,0.54);
\draw[dashed] (-0.46,-0.54)--(-1.61,-1.92);
\fill[black] (1.09,1.2) circle (0.07cm) node[anchor=south] {\small{$[s]$}};
\fill[black] (1.09,-1.2) circle (0.07cm) node[anchor=north] {\small{$[s]$}};
\end{tikzpicture}
\end{center}
\end{minipage}
\begin{minipage}{0.6\textwidth}
\begin{align*}
\sigma(\varphi)&:=\varepsilon e^{I\varphi},\hspace{2.36cm} \varphi\in I_{\varphi_3}\setminus\overline{I_{\varphi_1}}, \\
\gamma_{1,\pm}(t)&:=\pm|t|e^{-\sgn(t)I\varphi_1},\hspace{0.8cm} t\in\mathbb{R}\setminus[-\varepsilon,\varepsilon], \\
\gamma_{2,\pm}(t)&:=\pm|t|e^{-\sgn(t)J\varphi_2},\qquad t\in\mathbb{R}\setminus[-\varepsilon,\varepsilon], \\
\gamma_{3,\pm}(t)&:=\pm|t|e^{-\sgn(t)I\varphi_3},\hspace{0.81cm} t\in\mathbb{R}\setminus[-\varepsilon,\varepsilon], \\
\gamma_k&:=\gamma_{k,+}\oplus\gamma_{k,-},\hspace{1.28cm} k\in\{1,2,3\}.
\end{align*}

\medskip Note, that $\sigma,\gamma_1,\gamma_3$ are curves in $\mathbb{C}_I$, while $\gamma_2$ is in $\mathbb{C}_J$.
\end{minipage}

\medskip Then the Cauchy integral formula \eqref{Eq_Cauchy_integral_formula_left} gives
\begin{equation}\label{Eq_Independence_5}
f(s)=\int_{\gamma_3\ominus\sigma\ominus\gamma_1}S_L^{-1}(p,s)dp_If(p),\qquad s\in\ran(\gamma_2).
\end{equation}
Note, that the integral which closes the path at infinity vanishes because of the estimate $|S_L^{-1}(p,s)|\leq\frac{|p|+|s|}{(|p|-|s|)^2}$ from \eqref{Eq_SL_estimate_3} and the convergence in Lemma~\ref{lem_Integral_vanish}~ii). Next, we additionally consider the curves \medskip

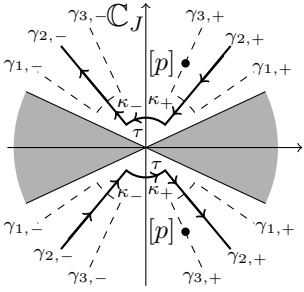
\begin{minipage}{0.34\textwidth}
\begin{center}
\begin{tikzpicture}[scale=0.8]
\fill[black!30] (0,0)--(1.99,-0.93) arc (-25:25:2.2)--(0,0)--(-1.99,-0.93) arc (205:155:2.2);
\draw (-1.99,-0.93)--(1.99,0.93);
\draw (-1.99,0.93)--(1.99,-0.93);
\draw[->] (-2.3,0)--(2.6,0);
\draw[->] (0,-2.3)--(0,2.4);
\draw (0.15,2.55) node[anchor=north east] {\large{$\mathbb{C}_J$}};
\fill[black] (0.66,1.4) circle (0.07cm) node[anchor=east] {\small{$[p]$}};
\fill[black] (0.66,-1.4) circle (0.07cm) node[anchor=east] {\small{$[p]$}};
\draw[dashed] (0.93,1.99)--(0.42,0.91) arc (65:35:1)--(1.8,1.26);
\draw[dashed] (0.93,-1.99)--(0.42,-0.91) arc (-65:-35:1)--(1.8,-1.26);
\draw[dashed] (-0.93,-1.99)--(-0.42,-0.91) arc (245:215:1)--(-1.8,-1.26);
\draw[dashed] (-0.93,1.99)--(-0.42,0.91) arc (115:145:1)--(-1.8,1.26);
\draw (1.6,1.3) node[anchor=west] {\tiny{$\gamma_{1,+}$}};
\draw (1.65,-1.35) node[anchor=west] {\tiny{$\gamma_{1,+}$}};
\draw (-1.5,1.35) node[anchor=east] {\tiny{$\gamma_{1,-}$}};
\draw (-1.5,-1.35) node[anchor=east] {\tiny{$\gamma_{1,-}$}};
\draw (1,1.9) node[anchor=south] {\tiny{$\gamma_{3,+}$}};
\draw (0.95,-1.9) node[anchor=north] {\tiny{$\gamma_{3,+}$}};
\draw (-0.95,1.9) node[anchor=south] {\tiny{$\gamma_{3,-}$}};
\draw (-0.95,-1.9) node[anchor=north] {\tiny{$\gamma_{3,-}$}};
\draw[thick] (1.41,1.69)--(0.32,0.38) arc (50:130:0.5)--(-1.41,1.69);
\draw[thick] (1.41,-1.69)--(0.32,-0.38) arc (310:230:0.5)--(-1.41,-1.69);
\draw[thick,->] (0.64,-0.77)--(0.98,-1.17);
\draw[thick,->] (1.29,1.53)--(0.91,1.08);
\draw[thick,->] (-0.64,0.77)--(-1.09,1.3);
\draw[thick,->] (-1.29,-1.53)--(-0.91,-1.08);
\draw[thick,->] (0.51,0.61)--(0.45,0.54);
\draw[thick,->] (0.45,-0.54)--(0.51,-0.61);
\draw[thick,->] (-0.51,-0.61)--(-0.45,-0.54);
\draw[thick,->] (-0.45,0.54)--(-0.51,0.61);
\draw[thick,->] (-0.2,0.455)--(-0.21,0.45);
\draw[thick,->] (0.2,-0.455)--(0.21,-0.45);
\draw (1.15,1.8) node[anchor=west] {\tiny{$\gamma_{2,+}$}};
\draw (1.18,-1.85) node[anchor=west] {\tiny{$\gamma_{2,+}$}};
\draw (-1.1,1.85) node[anchor=east] {\tiny{$\gamma_{2,-}$}};
\draw (-1,-1.8) node[anchor=east] {\tiny{$\gamma_{2,-}$}};
\draw (-0.13,0.48) node[anchor=north] {\tiny{$\tau$}};
\draw (0.18,-0.52) node[anchor=south] {\tiny{$\tau$}};
\draw (0.65,0.75) node[anchor=east] {\tiny{$\kappa_+$}};
\draw (0.65,-0.75) node[anchor=east] {\tiny{$\kappa_+$}};
\draw (-0.65,0.7) node[anchor=west] {\tiny{$\kappa_-$}};
\draw (-0.65,-0.8) node[anchor=west] {\tiny{$\kappa_-$}};
\end{tikzpicture}
\end{center}
\end{minipage}
\begin{minipage}{0.65\textwidth}
\begin{align*}
\tau(\varphi)&:=\frac{\varepsilon}{2}e^{J\varphi},\hspace{1.85cm}\varphi\in\Big(-\frac{\pi}{2},\frac{3\pi}{2}\Big)\setminus\overline{I_{\varphi_2}}, \\
\kappa_\pm(t)&:=\pm|t|e^{-\sgn(t)J\varphi_2},\quad t\in(-\varepsilon,\varepsilon)\setminus\Big[-\frac{\varepsilon}{2},\frac{\varepsilon}{2}\Big], \\
\kappa&:=\kappa_+\oplus\kappa_-.
\end{align*}
\end{minipage}

\medskip If we use the fact that $S_L^{-1}(\,\cdot\,,T)$ is a right slice hyperholomorphic in $\mathbb{R}^{n+1}\setminus D_\omega$, the Cauchy integral formula \eqref{Eq_Cauchy_integral_formula_right} gives
\begin{align}
S_L^{-1}(p,T)&=\frac{1}{2\pi}\int_{\ominus\gamma_2\ominus\kappa\ominus\tau}S_L^{-1}(s,T)ds_JS_R^{-1}(s,p) \notag \\
&=\frac{1}{2\pi}\int_{\gamma_2\oplus\kappa\oplus\tau}S_L^{-1}(s,T)ds_JS_L^{-1}(p,s),\qquad p\in\ran(\gamma_3), \label{Eq_Independence_2}
\end{align}
where in the second equation we used the relation $S_L^{-1}(p,s)=-S_R^{-1}(s,p)$ between the left and the right Cauchy kernels. Note, that the integral, which closes the path at infinity, vanishes due to the asymptotics behaviour of the integrand
\begin{equation*}
\Vert S_L^{-1}(s,T)\Vert|S_L^{-1}(p,s)|\leq\frac{C_{\varphi_1}}{|s|}\frac{|p|+|s|}{(|p|-|s|)^2}=\mathcal{O}\Big(\frac{1}{|s|^2}\Big),\qquad\text{as }|s|\rightarrow\infty,
\end{equation*}
which follows from the estimates \eqref{Eq_SL_estimate} and \eqref{Eq_SL_estimate_3}. Analogously, we obtain the formula
\begin{equation}\label{Eq_Independence_3}
0=\frac{1}{2\pi}\int_{\gamma_2\oplus\kappa\oplus\tau}S_L^{-1}(s,T)ds_JS_L^{-1}(p,s),\qquad p\in\ran(\gamma_1),
\end{equation}
where the difference with respect to \eqref{Eq_Independence_2} is, that the left hand side vanishes since every $p\in\ran(\gamma_1)$ is outside the integration path $\gamma_2\oplus\kappa\oplus\tau$. Analogously to \eqref{Eq_Independence_5} there also holds
\begin{equation}\label{Eq_Independence_6}
0=\int_{\gamma_3\ominus\sigma\ominus\gamma_1}S_L^{-1}(p,s)dp_If(p),\qquad s\in\ran(\kappa\oplus\tau),
\end{equation}
where the left hand side is zero instead of $f(s)$, since every $s\in\ran(\kappa\oplus\tau)$ is outside the integration path $\gamma_3\ominus\sigma\ominus\gamma_1$. In the upcoming calculation, let us now use \eqref{Eq_Independence_5} in the first equation, \eqref{Eq_Independence_2} and \eqref{Eq_Independence_3} in the third equation, and \eqref{Eq_Independence_6} in the fifth equation. Doing so, this gives the formula
\begin{align}
\int_{\gamma_2}S_L^{-1}(s,T)ds_Jf(s)&=\frac{1}{2\pi}\int_{\gamma_2}S_L^{-1}(s,T)ds_J\bigg(\int_{\gamma_3\ominus\sigma\ominus\gamma_1}S_L^{-1}(p,s)dp_If(p)\bigg) \notag \\
&=\frac{1}{2\pi}\int_{\gamma_3\ominus\sigma\ominus\gamma_1}\bigg(\int_{\gamma_2}S_L^{-1}(s,T)ds_JS_L^{-1}(p,s)\bigg)dp_If(p) \notag \\
&=\int_{\gamma_3}\bigg(S_L^{-1}(p,T)-\frac{1}{2\pi}\int_{\kappa\oplus\tau}S_L^{-1}(s,T)ds_JS_L^{-1}(p,s)\bigg)dp_If(p) \notag \\
&\quad+\frac{1}{2\pi}\int_{\gamma_1} \bigg(\int_{\kappa\oplus\tau}S_L^{-1}(s,T)ds_JS_L^{-1}(p,s)\bigg)dp_If(p) \notag \\
&\quad-\frac{1}{2\pi}\int_\sigma\bigg(\int_{\gamma_2}S_L^{-1}(s,T)ds_JS_L^{-1}(p,s)\bigg)dp_If(p) \notag \\
&\hspace{-2.5cm}=\int_{\gamma_3}S_L^{-1}(p,T)dp_If(p)+\frac{1}{2\pi}\int_{\kappa\oplus\tau}S_L^{-1}(s,T)ds_J\bigg(\int_{\gamma_1\ominus\gamma_3}S_L^{-1}(p,s)dp_If(p)\bigg) \notag \\
&\hspace{-2.5cm}\quad-\frac{1}{2\pi}\int_\sigma\bigg(\int_{\gamma_2}S_L^{-1}(s,T)ds_JS_L^{-1}(p,s)\bigg)dp_If(p) \notag \\
&\hspace{-2.5cm}=\int_{\gamma_3}S_L^{-1}(p,T)dp_If(p)-\frac{1}{2\pi}\int_\sigma\bigg(\int_{\kappa\oplus\tau\oplus\gamma_2}S_L^{-1}(s,T)ds_JS_L^{-1}(p,s)\bigg)dp_If(p). \label{Eq_Independence_7}
\end{align}
In the above computations we were allowed to interchange the order of integration, since it is easy to check that all the double integrals are absolute convergent. To conclude the proof it remains to show that the last integral in \eqref{Eq_Independence_7} vanishes in the limit $\varepsilon\rightarrow 0^+$. To do so, we use \eqref{Eq_SL_estimate} and \eqref{Eq_SL_estimate_2} to estimate the integrand by
\begin{equation*}
\Vert S_L^{-1}(s,T)\Vert|S_L^{-1}(p,s)|\leq\frac{C_{\varphi_1}(|p|+|s|)}{|s||s-p_J||s-\overline{p_J}|}.
\end{equation*}
With this inequality, we can now estimate all three parts of the paths $\kappa\oplus\tau\oplus\gamma_2$ in the double integral separately. For the $\tau$-part, we get
\begin{align*}
\Big\Vert&\int_\sigma\Big(\int_\tau S_L^{-1}(s,T)ds_JS_L^{-1}(p,s)\Big)dp_If(p)\Big\Vert \\
&\leq 2^{\frac{n}{2}}\int_{I_{\varphi_3}\setminus\overline{I_{\varphi_1}}}\int_{(-\frac{\pi}{2},\frac{3\pi}{2})\setminus\overline{I_{\varphi_2}}}\Big\Vert S_L^{-1}\Big(\frac{\varepsilon}{2}e^{J\phi},T\Big)\Big\Vert\Big|S_L^{-1}\Big(\varepsilon e^{I\varphi},\frac{\varepsilon}{2}e^{J\phi}\Big)\Big||f(\varepsilon e^{I\varphi})|\frac{\varepsilon^2}{2}d\phi d\varphi \\
&\leq 2^{\frac{n}{2}-1}\varepsilon^2\int_{I_{\varphi_3}\setminus\overline{I_{\varphi_1}}}\int_{(-\frac{\pi}{2},\frac{3\pi}{2})\setminus\overline{I_{\varphi_2}}}\frac{C_{\varphi_1}(\frac{\varepsilon}{2}+\varepsilon)}{\frac{\varepsilon}{2}|\frac{\varepsilon}{2}e^{J\phi}-\varepsilon e^{J\varphi}||\frac{\varepsilon}{2}e^{J\phi}-\varepsilon e^{-J\varphi}|}|f(\varepsilon e^{I\varphi})|d\phi d\varphi \\
&\leq 2^{\frac{n}{2}-1}3C_{\varphi_1}\int_{I_{\varphi_3}\setminus\overline{I_{\varphi_1}}}\int_{(-\frac{\pi}{2},\frac{3\pi}{2})\setminus\overline{I_{\varphi_2}}}\frac{1}{|\frac{1}{2}e^{J\phi}-e^{J\varphi}||\frac{1}{2}e^{J\phi}-e^{-J\varphi}|}d\phi d\varphi\sup\limits_{\varphi\in I_{\varphi_3}}|f(\varepsilon e^{I\varphi})|\overset{\varepsilon\rightarrow 0^+}{\longrightarrow}0,
\end{align*}
where this upper bound converges to zero because of Lemma~\ref{lem_f_convergence} and since the double integral is finite because it does not admit any singularities. On the other hand, also using \eqref{Eq_SL_estimate} and \eqref{Eq_SL_estimate_2}, the $\kappa\oplus\gamma_2$-part of the double integral in \eqref{Eq_Independence_7} can be estimated by
\begin{align*}
&\Big\Vert\int_\sigma\Big(\int_{\kappa\oplus\gamma_2}S_L^{-1}(s,T)ds_JS_L^{-1}(p,s)\Big)dp_If(p)\Big\Vert \\
&\leq 2^{\frac{n}{2}}\sum\limits_\pm\int_{I_{\varphi_3}\setminus\overline{I_{\varphi_1}}}\int_{\mathbb{R}\setminus[-\frac{\varepsilon}{2},\frac{\varepsilon}{2}]}\big\Vert S_L^{-1}\big(\pm|t|e^{-\sgn(t)J\varphi_2},T\big)\big\Vert \\
&\hspace{4.6cm}\times\big|S_L^{-1}\big(\varepsilon e^{I\varphi},\pm|t|e^{-\sgn(t)J\varphi_2}\big)\big||f(\varepsilon e^{I\varphi})|\varepsilon dtd\varphi \\
&\leq 2^{\frac{n}{2}}\varepsilon\sum\limits_\pm\int_{I_{\varphi_3}\setminus\overline{I_{\varphi_1}}}\int_{\mathbb{R}\setminus[-\frac{\varepsilon}{2},\frac{\varepsilon}{2}]}\frac{C_{\varphi_1}(\varepsilon+|t|)}{|t|\big|\varepsilon e^{J\varphi}\mp|t|e^{-\sgn(t)J\varphi_2}\big|\big|\varepsilon e^{-J\varphi}\mp|t|e^{-\sgn(t)J\varphi_2}\big|}|f(\varepsilon e^{I\varphi})|dtd\varphi \\
&=2^{\frac{n}{2}}\varepsilon\sum\limits_\pm\int_{I_{\varphi_3}\setminus\overline{I_{\varphi_1}}}\int_{\mathbb{R}\setminus[-\frac{\varepsilon}{2},\frac{\varepsilon}{2}]}\frac{C_{\varphi_1}(\varepsilon+|t|)}{|t||\varepsilon e^{J\varphi}-te^{\pm J\varphi_2}||\varepsilon e^{-J\varphi}-te^{\pm J\varphi_2}|}|f(\varepsilon e^{I\varphi})|dtd\varphi \\
&\leq 2^{\frac{n}{2}}C_{\varphi_1}\sum\limits_\pm\int_{I_{\varphi_3}\setminus\overline{I_{\varphi_1}}}\int_{\mathbb{R}\setminus[-\frac{1}{2},\frac{1}{2}]}\frac{1+|\tau|}{|\tau||e^{J\varphi}-\tau e^{\pm J\varphi_2}||e^{-J\varphi}-\tau e^{\pm J\varphi_2}|}d\tau d\varphi\sup\limits_{\varphi\in I_{\varphi_3}}|f(\varepsilon e^{I\varphi})|\overset{\varepsilon\rightarrow 0^+}{\longrightarrow}0,
\end{align*}
where this upper bound converges to zero because of Lemma~\ref{lem_f_convergence} and since the double integral is finite because the only singularities are located at $\tau=1$ and $\varphi=\pm\varphi_2$, and are of order one which is integrable in the two-dimensional integral. Hence, performing the limit $\varepsilon\rightarrow 0^+$ in \eqref{Eq_Independence_7} gives the independence of the imaginary unit
\begin{equation*}
\int_{\partial D_{\varphi_2}\cap\mathbb{C}_J}S_L^{-1}(s,T)ds_Jf(s)=\int_{\partial D_{\varphi_3}\cap\mathbb{C}_I}S_L^{-1}(p,T)dp_If(p)=\int_{\partial D_{\varphi_2}\cap\mathbb{C}_I}S_L^{-1}(p,T)dp_If(p),
\end{equation*}
where in the second equation we replaced $\varphi_3$ by $\varphi_2$, which is allowed by the already proven independence of the angle.
\end{proof}

Theorem~\ref{thm_Independence} now proves the well definedness of the following $\omega$-functional calculus.

\begin{defi}[$\omega$-functional calculus]\label{defi_Omega}
Let $T\in\mathcal{K}(V)$ be bisectorial of angle $\omega\in(0,\frac{\pi}{2})$. Then for every $f\in\mathcal{SH}_L^0(D_\theta)$ we define the {\it $\omega$-functional calculus} \medskip

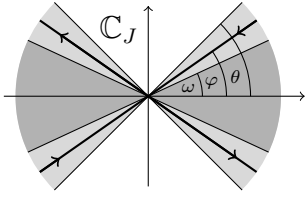
\begin{minipage}{0.4\textwidth}
\begin{center}
\begin{tikzpicture}[scale=0.8]
\fill[black!15] (0,0)--(1.56,1.56) arc (45:-45:2.2)--(0,0)--(-1.56,-1.56) arc (225:135:2.2);
\fill[black!30] (0,0)--(2,0.93) arc (25:-25:2.2)--(0,0)--(-2,-0.93) arc (205:155:2.2);
\draw (2,0.93)--(-2,-0.93);
\draw (2,-0.93)--(-2,0.93);
\draw (1.56,1.56)--(-1.56,-1.56);
\draw (1.56,-1.56)--(-1.56,1.56);
\draw (0.9,0) arc (0:25:0.9) (0.67,-0.1) node[anchor=south] {\tiny{$\omega$}};
\draw (1.3,0) arc (0:35:1.3) (1.05,-0.05) node[anchor=south] {\tiny{$\varphi$}};
\draw (1.7,0) arc (0:45:1.7) (1.45,0.05) node[anchor=south] {\tiny{$\theta$}};
\draw[thick] (1.8,1.26)--(-1.8,-1.26);
\draw[thick] (1.8,-1.26)--(-1.8,1.26);
\draw[thick,->] (1.8,1.26)--(1.47,1.03);
\draw[thick,->] (0,0)--(1.47,-1.03);
\draw[thick,->] (-1.8,-1.26)--(-1.47,-1.03);
\draw[thick,->] (0,0)--(-1.47,1.03);
\draw[->] (-2.4,0)--(2.6,0);
\draw[->] (0,-1.5)--(0,1.5) node[anchor=north east] {\large{$\mathbb{C}_J$}};
\end{tikzpicture}
\end{center}
\end{minipage}
\begin{minipage}{0.59\textwidth}
\begin{equation}\label{Eq_Omega}
f(T):=\frac{1}{2\pi}\int_{\partial D_\varphi\cap\mathbb{C}_J}S_L^{-1}(s,T)ds_Jf(s),
\end{equation}
where $\varphi\in(\omega,\theta)$ and $J\in\mathbb{S}$ are arbitrary. \medskip

By Theorem~\ref{thm_Independence}, $f(T)$ is independent of the chosen $\varphi$ and $J$ and leads to a bounded operator $f(T)\in\mathcal{B}(V)$.
\end{minipage}
\end{defi}

For every $f,g\in\mathcal{SH}_L^0(D_\theta)$ and $a\in\mathbb{R}_n$, there clearly holds
\begin{equation}\label{Eq_Linearity_omega}
(f+g)(T)=f(T)+g(T)\qquad\text{and}\qquad(fa)(T)=f(T)a.
\end{equation}

\begin{rem}\label{rem_Omega_0rho}
In the special case that $0\in\rho_S(T)$, there is also $U_r(0)\subseteq\rho_S(T)$ for some $r>0$. as. Then the estimate $\Vert S_L^{-1}(s,T)\Vert\leq\frac{C_\varphi}{|s|}$ and Lemma~\ref{lem_Integral_vanish} allow to rewrite the $\omega$-functional calculus \eqref{Eq_Omega} for every $\rho\in(0,r)$ as \medskip

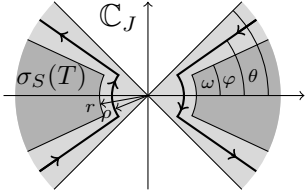
\begin{minipage}{0.35\textwidth}
\begin{center}
\begin{tikzpicture}[scale=0.8]
\fill[black!15] (0,0)--(1.56,1.56) arc (45:-45:2.2)--(0,0)--(-1.56,-1.56) arc (225:135:2.2);
\fill[black!30] (0.73,0.34)--(2,0.93) arc (25:-25:2.2)--(0.73,-0.34) arc (-25:25:0.8);
\fill[black!30] (-0.73,0.34)--(-2,0.93) arc (155:205:2.2)--(-0.73,-0.34) arc (205:155:0.8);
\draw (2,0.93)--(0.73,0.34) arc (25:-25:0.8)--(2,-0.93);
\draw (-2,0.93)--(-0.73,0.34) arc (155:205:0.8)--(-2,-0.93);
\draw (1.56,1.56)--(-1.56,-1.56);
\draw (1.56,-1.56)--(-1.56,1.56);
\draw[thick] (1.8,1.26)--(0.49,0.34) arc (35:-35:0.6)--(1.8,-1.26);
\draw[thick] (-1.8,1.26)--(-0.49,0.34) arc (145:215:0.6)--(-1.8,-1.26);
\draw[thick,->] (1.8,1.26)--(1.47,1.03);
\draw[thick,->] (1.39,-0.98)--(1.47,-1.03);
\draw[thick,->] (-1.8,-1.26)--(-1.47,-1.03);
\draw[thick,->] (-1.39,0.98)--(-1.47,1.03);
\draw[thick,->] (-0.555,0.22)--(-0.54,0.25);
\draw[thick,->] (0.555,-0.22)--(0.54,-0.25);
\draw[->] (0,0)--(-0.54,-0.22);
\draw[->] (0,0)--(-0.79,-0.14);
\draw (-0.7,-0.2) node[anchor=east] {\tiny{$r$}};
\draw (-0.95,-0.33) node[anchor=west] {\tiny{$\rho$}};
\draw[->] (-2.4,0)--(2.6,0);
\draw[->] (0,-1.56)--(0,1.56);
\draw (0,1.3) node[anchor=east] {\large{$\mathbb{C}_J$}};
\draw (-1.6,-0.08) node[anchor=south] {\small{$\sigma_S(T)$}};
\draw (1.2,0) arc (0:25:1.2) (1,-0.05) node[anchor=south] {\tiny{$\omega$}};
\draw (1.6,0) arc (0:35:1.6) (1.35,-0.03) node[anchor=south] {\tiny{$\varphi$}};
\draw (2,0) arc (0:45:2) (1.75,0.05) node[anchor=south] {\tiny{$\theta$}};
\end{tikzpicture}
\end{center}
\end{minipage}
\begin{minipage}{0.64\textwidth}
\begin{equation}\label{Eq_Omega_0rho}
f(T):=\frac{1}{2\pi}\int_{\partial(D_\varphi\setminus U_\rho(0))\cap\mathbb{C}_J}S_L^{-1}(s,T)ds_Jf(s).
\end{equation}
\end{minipage}

\medskip However, the functional calculus \eqref{Eq_Omega_0rho} is defined for a larger class of functions, namely for every bounded $f\in\mathcal{SH}_L(D_\theta)$, such that for some $R>0$ there is
\begin{equation}\label{Eq_Integrability_at_infinity}
\int_R^\infty|f(te^{J\phi})|\frac{dt}{t}<\infty,\qquad J\in\mathbb{S},\,\phi\in I_\theta.
\end{equation}
All the stated results of this section are then also valid for this functional calculus accordingly.
\end{rem}

The next proposition shows that for intrinsic functions $g\in\mathcal{N}^0(D_\theta)$, the functional calculus \eqref{Eq_Omega} can also be defined using the right $S$-resolvent operator \eqref{Eq_S_resolvent}.

\begin{prop}
Let $T\in\mathcal{K}(V)$ be bisectorial of angle $\omega\in(0,\frac{\pi}{2})$. Then for $g\in\mathcal{N}^0(D_\theta)$, $\theta\in(\omega,\frac{\pi}{2})$, the functional calculus $g(T)$ in \eqref{Eq_Omega} can be written as
\begin{equation}\label{Eq_Representation_intrinsic}
g(T)=\frac{1}{2\pi}\int_{\partial D_\varphi\cap\mathbb{C}_J}g(s)ds_JS_R^{-1}(s,T),
\end{equation}
where $\varphi\in(\omega,\theta)$ and $J\in\mathbb{S}$ are arbitrary.
\end{prop}

\begin{proof}
First, we decompose the paths $\gamma_\pm(t):=\pm|t|e^{-\sgn(t)J\varphi}$, $t\in\mathbb{R}$, which parametrizes the boundary $\partial D_\varphi\cap\mathbb{C}_J$, into
\begin{equation*}
\gamma_\pm(t)=\underbrace{\pm|t|\cos(\varphi)}_{=:x(t)}\underbrace{\mp t\sin(\varphi)}_{=:y(t)}J=x(t)+y(t)J,\qquad t\in\mathbb{R},
\end{equation*}
where we keep in mind that $x(t)$ and $y(t)$ depend on the choice of the sign $\pm$. With these functions, we can now  write the left $S$-resolvent operator \eqref{Eq_S_resolvent} as
\begin{align}
S_L^{-1}(\gamma_\pm(t),T)&=Q_{\gamma_\pm(t)}[T]^{-1}\overline{\gamma_\pm(t)}-TQ_{\gamma_\pm(t)}[T]^{-1} \notag \\
&=\underbrace{(x(t)-T)Q_{\gamma_\pm(t)}[T]^{-1}}_{=:A_\pm(t)}\underbrace{-y(t)Q_{\gamma_\pm(t)}[T]^{-1}}_{=:B_\pm(t)}J=A_\pm(t)+B_\pm(t)J. \label{Eq_Representation_intrinsic_1}
\end{align}
Moreover, we can also decompose the function $g$ according to \eqref{Eq_Holomorphic_decomposition} into
\begin{align}
\frac{\gamma_\pm'(t)}{J}g(\gamma_\pm(t))&=\frac{x'(t)+y'(t)J}{J}\big(g_0(x(t),y(t))+Jg_1(x(t),y(t))\big) \notag \\
&=\underbrace{x'(t)g_1(x(t),y(t))+y'(t)g_0(x(t),y(t))}_{=:\alpha_\pm(t)} \notag \\
&\quad+J\underbrace{\big(y'(t)g_1(x(t),y(t))-x'(t)g_0(x(t),y(t))\big)}_{=:\beta_\pm(t)}=\alpha_\pm(t)+J\beta_\pm(t).  \label{Eq_Representation_intrinsic_2}
\end{align}
With the two decompositions \eqref{Eq_Representation_intrinsic_1} and \eqref{Eq_Representation_intrinsic_2}, we can now write
\begin{align}
g(T)&=\frac{1}{2\pi}\sum_\pm\int_\mathbb{R}S_L^{-1}(\gamma_\pm(t),T)\frac{\gamma_\pm'(t)}{J}g(\gamma_\pm(t))dt \notag \\
&=\frac{1}{2\pi}\sum_\pm\int_\mathbb{R}\big(A_\pm(t)+B_\pm(t)J\big)\big(\alpha_\pm(t)+J\beta_\pm(t)\big)dt \notag \\
&=\frac{1}{\pi}\sum_\pm\int_0^\infty\big(A_\pm(t)\alpha_\pm(t)-B_\pm(t)\beta_\pm(t)\big)dt, \label{Eq_Representation_intrinsic_3}
\end{align}
where in the last line we used the properties $A_\pm(-t)=A_\pm(t)$ and $B_\pm(-t)=-B_\pm(t)$, but also $\alpha_\pm(-t)=\alpha_\pm(t)$ and $\beta_\pm(-t)=-\beta_\pm(t)$, which is a consequence of the compatibility condition \eqref{Eq_Symmetry_condition}. Next, with the same operators $A_\pm(t)$ and $B_\pm(t)$, we can also write the right $S$-resolvent operator \eqref{Eq_S_resolvent} in the form
\begin{align*}
S_R^{-1}(\gamma_\pm(t),T)&=\big(\overline{\gamma_\pm(t)}-T\big)Q_{\gamma_\pm(t)}[T]^{-1} \\
&=(x_\pm(t)-Jy_\pm(t)-T)Q_{\gamma_\pm(t)}[T]^{-1}=A_\pm(t)+JB_\pm(t).
\end{align*}
Moreover, since $g$ is intrinsic, the functions $g_0,g_1$ in \eqref{Eq_Representation_intrinsic_2} are real, which means that $\alpha_\pm,\beta_\pm$ are real valued as well. This allows us to write
\begin{equation*}
g(\gamma_\pm(t))\frac{\gamma_\pm'(t)}{J}=\alpha_\pm(t)+\beta_\pm(t)J.
\end{equation*}
Thus in the same way as in \eqref{Eq_Representation_intrinsic_3}, we also get
\begin{align}
\frac{1}{2\pi}\int_{\partial D_\varphi\cap\mathbb{C}_J}g(s)ds_JS_R^{-1}(s,T)&=\frac{1}{2\pi}\sum_\pm\int_\mathbb{R}g(\gamma_\pm(t))\frac{\gamma_\pm'(t)}{J}S_R^{-1}(\gamma_\pm(t),T)dt \notag \\
&=\frac{1}{\pi}\sum_\pm\int_0^\infty\big(A_\pm(t)\alpha_\pm(t)-B_\pm(t)\beta_\pm(t)\big)dt. \label{Eq_Representation_intrinsic_4}
\end{align}
Since the right hand sides of \eqref{Eq_Representation_intrinsic_3} and \eqref{Eq_Representation_intrinsic_4} coincide, the representation \eqref{Eq_Omega} is proven.
\end{proof}

The next proposition shows that for certain subclasses of functions, the $\omega$-functional calculus \eqref{Eq_Omega} coincides with the well established functional calculus for unbounded operators, see e.g. \cite[Definition~3.4.2]{FJBOOK} or \cite[Theorem 4.12]{ColSab2006}.

\begin{prop}\label{prop_Unbounded_omega}
Let $T\in\mathcal{K}(V)$ be bisectorial of angle $\omega\in(0,\frac{\pi}{2})$ and $f\in\mathcal{SH}_L(\mathbb{R}^{n+1}\setminus K)$ for some axially symmetric compact set $K\subseteq\mathbb{R}^{n+1}\setminus\overline{D_\omega}$. Moreover, we assume that $f(0)=0$ and $\lim_{|s|\rightarrow\infty}f(s)=0$. Then there exists some $\theta\in(\omega,\frac{\pi}{2})$ such that $f\in\mathcal{SH}_L^0(D_\theta)$ and the $\omega$-functional calculus \eqref{Eq_Omega} can be written as
\begin{equation}\label{Eq_Unbounded_omega}
f(T)=\frac{-1}{2\pi}\int_{\partial U\cap\mathbb{C}_J}S_L^{-1}(s,T)ds_Jf(s),
\end{equation}
where $J\in\mathbb{S}$ is arbitrary and $U$ is some axially symmetric open set with $K\subseteq U\subseteq\overline{U}\subseteq\rho_S(T)$.
\end{prop}

\begin{proof}
Since the compact set $K$ is located outside the closed double sector $\overline{D_\omega}$, there exists some $\theta\in(\omega,\frac{\pi}{2})$, such that $K\subseteq\mathbb{R}^{n+1}\setminus\overline{D_\theta}$. This means that $f$ can be restricted to a function $f\in\mathcal{SH}_L(D_\theta)$. Moreover, since $f(0)=0$ and $f$ is holomorphic in a neighbourhood around zero, also $\frac{1}{s}f(s)$ is holomorphic at zero and hence $f(s)=\mathcal{O}(s)$ as $s\rightarrow 0$. Analogously, since $\lim_{|s|\rightarrow\infty}f(s)=0$ and $f$ is holomorphic outside a sufficient large ball, there is $sf(s)$ bounded at infinity and hence $f(s)=\mathcal{O}(\frac{1}{|s|})$ as $s\rightarrow\infty$. With the above considerations we conclude that $f$ is bounded on $D_\theta$ and that
\begin{equation*}
\int_0^\infty|f(te^{J\phi})|\frac{dt}{t}<\infty,\qquad J\in\mathbb{S},\,\phi\in I_\theta,
\end{equation*}
which means that $f\in\mathcal{SH}_L^0(D_\theta)$. In order to show the equality \eqref{Eq_Unbounded_omega}, we write $f(T)$ as
\begin{align*}
f(T)&=\frac{1}{2\pi}\lim\limits_{\varepsilon\rightarrow 0^+}\lim\limits_{R\rightarrow\infty}\int_\gamma S_L^{-1}(s,T)ds_Jf(s) \\
&=\frac{1}{2\pi}\lim\limits_{\varepsilon\rightarrow 0^+}\lim\limits_{R\rightarrow\infty}\int_{\sigma\oplus\gamma\ominus\tau}S_L^{-1}(s,T)ds_Jf(s).
\end{align*}
Choosing $\varphi\in(\omega,\frac{\pi}{2})$ close enough to $\omega$, $\varepsilon>0$ small enough, and $R>0$ large enough, then the path $\sigma\oplus\gamma\ominus\tau$ surrounds the compact set $K$ in the negative sense. \medskip

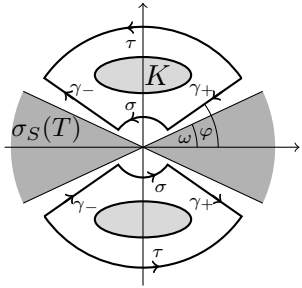
\begin{minipage}{0.34\textwidth}
\begin{center}
\begin{tikzpicture}[scale=0.8]
\fill[black!30] (0,0)--(1.99,-0.93) arc (-25:25:2.2)--(0,0)--(-1.99,-0.93) arc (205:155:2.2);
\draw[thick,fill=black!15] (0,1.2) ellipse (0.8cm and 0.3cm);
\draw[thick,fill=black!15] (0,-1.2) ellipse (0.8cm and 0.3cm);
\draw (-0.15,1.2) node[anchor=west] {$K$};
\draw[thick] (1.64,1.15)--(0.41,0.29) arc (35:145:0.5)--(-1.64,1.15) arc (145:35:2);
\draw[thick] (-1.64,-1.15)--(-0.41,-0.29) arc (215:325:0.5)--(1.64,-1.15) arc (325:215:2);
\draw (-1.99,-0.93)--(1.99,0.93);
\draw (-1.99,0.93)--(1.99,-0.93);
\draw[->] (-2.3,0)--(2.6,0);
\draw[->] (0,-2.3)--(0,2.4);
\draw[thick,->] (1.31,0.92)--(1.15,0.8);
\draw[thick,->] (1.15,-0.8)--(1.31,-0.92);
\draw[thick,->] (-1.15,0.8)--(-1.31,0.92);
\draw[thick,->] (-1.31,-0.92)--(-1.15,-0.8);
\draw (0.6,0.95) node[anchor=west] {\tiny{$\gamma_+$}};
\draw (0.6,-0.95) node[anchor=west] {\tiny{$\gamma_+$}};
\draw (-0.6,0.95) node[anchor=east] {\tiny{$\gamma_-$}};
\draw (-0.55,-1) node[anchor=east] {\tiny{$\gamma_-$}};
\draw[thick,->] (-0.2,0.455)--(-0.21,0.45);
\draw[thick,->] (0.2,-0.455)--(0.21,-0.45);
\draw (-0.2,0.45) node[anchor=south] {\tiny{$\sigma$}};
\draw (0.3,-0.45) node[anchor=north] {\tiny{$\sigma$}};
\draw[thick,->] (-0.17,1.99)--(-0.27,1.98);
\draw[thick,->] (0.17,-1.99)--(0.27,-1.98);
\draw (-0.2,1.95) node[anchor=north] {\tiny{$\tau$}};
\draw (0.25,-2) node[anchor=south] {\tiny{$\tau$}};
\draw (-1.6,-0.1) node[anchor=south] {\small{$\sigma_S(T)$}};
\draw (0.9,0) arc (0:25:0.9);
\draw (1.25,0) arc (0:35:1.25);
\draw (0.7,-0.1) node[anchor=south] {\tiny{$\omega$}};
\draw (1.05,-0.05) node[anchor=south] {\tiny{$\varphi$}};
\end{tikzpicture}
\end{center}
\end{minipage}
\begin{minipage}{0.65\textwidth}
\begin{align*}
\tau(\phi)&:=Re^{J\phi},\hspace{2.1cm}\phi\in\Big(-\frac{\pi}{2},\frac{3\pi}{2}\Big)\setminus\overline{I_\varphi}, \\
\sigma(\phi)&:=\varepsilon e^{J\phi},\hspace{2.21cm}\phi\in\Big(-\frac{\pi}{2},\frac{3\pi}{2}\Big)\setminus\overline{I_\varphi}, \\
\gamma_\pm(t)&:=\pm|t|e^{-\sgn(t)J\varphi},\qquad t\in(-R,R)\setminus[-\varepsilon,\varepsilon], \\
\gamma&:=\gamma_+\oplus\gamma_-.
\end{align*}
\end{minipage}

\medskip Since the integrand is slice hyperholomorphic on $\mathbb{R}^{n+1}\setminus(\overline{D_\omega}\cup K)$, we can change the integration path $\sigma\oplus\gamma\ominus\tau$ to the (negative oriented) boundary of $U\cap\mathbb{C}_J$. Changing also the orientation then gives the stated
\begin{equation*}
f(T)=\frac{1}{2\pi}\int_{\ominus\partial U\cap\mathbb{C}_J}S_L^{-1}(s,T)ds_Jf(s)=\frac{-1}{2\pi}\int_{\partial U\cap\mathbb{C}_J}S_L^{-1}(s,T)ds_Jf(s). \qedhere
\end{equation*}
\end{proof}

\begin{prop}\label{prop_Kernel_omega}
Let $T\in\mathcal{K}(V)$ be bisectorial of angle $\omega\in(0,\frac{\pi}{2})$, $g\in\mathcal{N}^0(D_\theta)$, $\theta\in(\omega,\frac{\pi}{2})$. Then there holds
\begin{equation*}
\ker(T)\subseteq\ker(g(T)).
\end{equation*}
\end{prop}

\begin{proof}
Let $v\in\ker(T)$. Then for every $s\in\rho_S(T)\setminus\{0\}$ we have $Q_s[T]v=|s|^2v$, consequently $Q_s[T]^{-1}v=\frac{1}{|s|^2}v$ and hence
\begin{equation*}
S_R^{-1}(s,T)v=(\overline{s}-T)Q_s[T]^{-1}v=\frac{1}{|s|^2}(\overline{s}-T)v=\frac{\overline{s}}{|s|^2}v=\frac{1}{s}v.
\end{equation*}
Hence, the representation \eqref{Eq_Representation_intrinsic} becomes
\begin{equation*}
g(T)v=\frac{1}{2\pi}\int_{\partial D_\varphi\cap\mathbb{C}_J}\frac{g(s)}{s}ds_Jv.
\end{equation*}
Using the curves \medskip

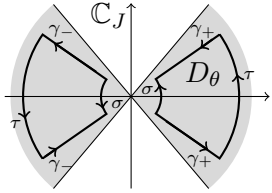
\begin{minipage}{0.39\textwidth}
\begin{center}
\begin{tikzpicture}[scale=0.8]
\fill[black!15] (0,0)--(1.29,1.53) arc (50:-50:2)--(0,0)--(-1.29,-1.53) arc (230:130:2);
\draw (1.29,1.53)--(-1.29,-1.53);
\draw (1.29,-1.53)--(-1.29,1.53);
\draw[->] (-2.1,0)--(2.3,0);
\draw[->] (0,-1.53)--(0,1.55);
\draw (0.1,1.35) node[anchor=east] {$\mathbb{C}_J$};
\draw[thick] (1.47,1.03)--(0.41,0.29) arc (35:-35:0.5)--(1.47,-1.03) arc (-35:35:1.8);
\draw[thick] (-1.47,-1.03)--(-0.41,-0.29) arc (215:145:0.5)--(-1.47,1.03) arc (145:215:1.8);
\draw[thick,->] (1.35,0.95)--(1.15,0.8);
\draw[thick,->] (0.58,-0.41)--(1.28,-0.89);
\draw[thick,->] (-1.43,-1)--(-1.23,-0.86);
\draw[thick,->] (-0.58,0.41)--(-1.28,0.89);
\draw (1.2,0.8) node[anchor=south] {\tiny{$\gamma_+$}};
\draw (1.15,-0.8) node[anchor=north] {\tiny{$\gamma_+$}};
\draw (-1.15,0.8) node[anchor=south] {\tiny{$\gamma_-$}};
\draw (-1.15,-0.88) node[anchor=north] {\tiny{$\gamma_-$}};
\draw[thick,->] (0.47,0.19)--(0.468,0.2);
\draw[thick,->] (-0.47,-0.19)--(-0.468,-0.2);
\draw (0.55,0.1) node[anchor=east] {\tiny{$\sigma$}};
\draw (-0.5,-0.13) node[anchor=west] {\tiny{$\sigma$}};
\draw[thick,->] (1.77,0.31)--(1.76,0.37);
\draw[thick,->] (-1.77,-0.31)--(-1.76,-0.37);
\draw (1.7,0.35) node[anchor=west] {\tiny{$\tau$}};
\draw (-1.65,-0.45) node[anchor=east] {\tiny{$\tau$}};
\draw (1.2,0) node[anchor=south] {$D_\theta$};
\end{tikzpicture}
\end{center}
\end{minipage}
\begin{minipage}{0.6\textwidth}
\begin{align*}
\sigma(\phi)&:=\varepsilon e^{J\phi},\hspace{2.22cm} \phi\in I_\varphi, \\
\tau(\phi)&:=Re^{J\phi},\hspace{2.11cm} \phi\in I_\varphi, \\
\gamma_\pm(t)&:=\pm|t|e^{-\sgn(t)J\varphi},\qquad t\in(-R,R)\setminus[-\varepsilon,\varepsilon], \\
\gamma&:=\gamma_+\oplus\gamma_-,
\end{align*}
\end{minipage}

\medskip we can now rewrite this integral as
\begin{equation*}
g(T)v=\frac{1}{2\pi}\lim\limits_{\varepsilon\rightarrow 0^+}\lim\limits_{R\rightarrow\infty}\int_\gamma\frac{g(s)}{s}ds_Jv=\frac{1}{2\pi}\lim\limits_{\varepsilon\rightarrow 0^+}\lim\limits_{R\rightarrow\infty}\int_{\gamma\ominus\sigma\oplus\tau}\frac{g(s)}{s}ds_Jv,
\end{equation*}
where in the second equation we used the fact that the integrals along $\sigma$ and $\tau$ vanish in the limit $\varepsilon\rightarrow 0^+$ and $R\rightarrow\infty$ due to Lemma~\ref{lem_Integral_vanish}. The only singularity $s=0$ of the integrand $\frac{g(s)}{s}$ now lies outside the integration path $\gamma\ominus\sigma\oplus\tau$, which means that
\begin{equation*}
\int_{\gamma\ominus\sigma\oplus\tau}\frac{g(s)}{s}ds_Jv=0
\end{equation*}
vanishes by the Cauchy theorem, and we conclude $g(T)v=0$.
\end{proof}

The next lemma is preparatory for the product rule in Theorem~\ref{thm_Product_omega}.

\begin{lem}
Let $g\in\mathcal{N}^0(D_\theta)$, $\theta\in(0,\frac{\pi}{2})$ and $B\in\mathcal{B}(V)$. Then for every $\varphi\in(0,\theta)$ and $J\in\mathbb{S}$, we have
\begin{equation}\label{Eq_Product_identity}
\frac{1}{2\pi}\int_{\partial D_\varphi\cap\mathbb{C}_J}g(s)ds_J(\overline{s}B-Bp)Q_s(p)^{-1}=\begin{cases} Bg(p), & p\in D_\varphi, \\ 0, & p\in\mathbb{R}^{n+1}\setminus\overline{D_\varphi}. \end{cases}
\end{equation}
\end{lem}

\begin{proof}
Consider first the case $p\in D_\varphi$. For any $0<\varepsilon<|p|<R$ let us consider the curves \medskip

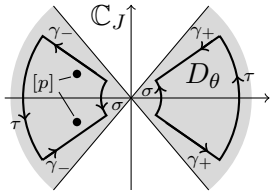
\begin{minipage}{0.39\textwidth}
\begin{center}
\begin{tikzpicture}[scale=0.8]
\fill[black!15] (0,0)--(1.29,1.53) arc (50:-50:2)--(0,0)--(-1.29,-1.53) arc (230:130:2);
\draw (1.29,1.53)--(-1.29,-1.53);
\draw (1.29,-1.53)--(-1.29,1.53);
\draw[->] (-2.1,0)--(2.3,0);
\draw[->] (0,-1.53)--(0,1.55);
\draw (0.1,1.35) node[anchor=east] {$\mathbb{C}_J$};
\draw[thick] (1.47,1.03)--(0.41,0.29) arc (35:-35:0.5)--(1.47,-1.03) arc (-35:35:1.8);
\draw[thick] (-1.47,-1.03)--(-0.41,-0.29) arc (215:145:0.5)--(-1.47,1.03) arc (145:215:1.8);
\draw[thick,->] (1.35,0.95)--(1.15,0.8);
\draw[thick,->] (0.58,-0.41)--(1.28,-0.89);
\draw[thick,->] (-1.43,-1)--(-1.23,-0.86);
\draw[thick,->] (-0.58,0.41)--(-1.28,0.89);
\draw (1.2,0.8) node[anchor=south] {\tiny{$\gamma_+$}};
\draw (1.15,-0.8) node[anchor=north] {\tiny{$\gamma_+$}};
\draw (-1.15,0.8) node[anchor=south] {\tiny{$\gamma_-$}};
\draw (-1.2,-0.88) node[anchor=north] {\tiny{$\gamma_-$}};
\draw[thick,->] (0.47,0.19)--(0.468,0.2);
\draw[thick,->] (-0.47,-0.19)--(-0.468,-0.2);
\draw (0.55,0.1) node[anchor=east] {\tiny{$\sigma$}};
\draw (-0.5,-0.13) node[anchor=west] {\tiny{$\sigma$}};
\draw[thick,->] (1.77,0.31)--(1.76,0.37);
\draw[thick,->] (-1.77,-0.31)--(-1.76,-0.37);
\draw (1.7,0.35) node[anchor=west] {\tiny{$\tau$}};
\draw (-1.65,-0.45) node[anchor=east] {\tiny{$\tau$}};
\draw (1.2,0) node[anchor=south] {$D_\theta$};
\fill[black] (-0.9,0.4) circle (0.07cm);
\fill[black] (-0.9,-0.4) circle (0.07cm);
\draw (-1,0.35)--(-1.2,0.3);
\draw (-0.95,-0.3)--(-1.2,0.1);
\draw(-1.1,0.25) node[anchor=east] {\tiny{$[p]$}};
\end{tikzpicture}
\end{center}
\end{minipage}
\begin{minipage}{0.6\textwidth}
\begin{align*}
\gamma_\pm(t)&:=\pm|t|e^{-\sgn(t)J\varphi},\qquad t\in(-R,R)\setminus[-\varepsilon,\varepsilon], \\
\sigma(\phi)&:=\varepsilon e^{J\phi},\hspace{2.22cm} \phi\in I_\varphi, \\
\tau(\phi)&:=Re^{J\phi},\hspace{2.11cm} \phi\in I_\varphi, \\
\gamma&:=\gamma_+\oplus\gamma_-.
\end{align*}
\end{minipage}

\medskip In this setting we know by \cite[Lemma 4.1.2]{CGK} that
\begin{equation}\label{Eq_Product_identity_1}
Bg(p)=\frac{1}{2\pi}\int_{\gamma\oplus\sigma\ominus\tau}g(s)ds_J(\overline{s}B-Bp)Q_s(p)^{-1}.
\end{equation}
Since the second part of the integrand behaves asymptotically as
\begin{equation*}
\big\Vert(\overline{s}B-Bp)Q_s(p)^{-1}\big\Vert\leq\frac{(|s|+|p|)\Vert B\Vert}{|Q_s(p)|}\leq\frac{(|s|+|p|)\Vert B\Vert}{(|s|-|p|)^2}=\begin{cases} \mathcal{O}(\frac{1}{|s|}), & \text{as }|s|\rightarrow\infty, \\ \mathcal{O}(1), & \text{as }|s|\rightarrow 0, \end{cases}
\end{equation*}
Lemma~\ref{lem_Integral_vanish}, which also holds for operator valued functions, shows that the integrals along the paths $\sigma$ and $\tau$ vanish in the limit $\varepsilon\rightarrow 0^+$ and $R\rightarrow\infty$, respectively. Hence \eqref{Eq_Product_identity_1} turns into
\begin{align*}
Bg(p)&=\lim\limits_{R\rightarrow\infty}\lim\limits_{\varepsilon\rightarrow 0^+}\frac{1}{2\pi}\int_\gamma g(s)ds_J(\overline{s}B-Bp)Q_s(p)^{-1} \\
&=\frac{1}{2\pi}\int_{\partial D_\varphi\cap\mathbb{C}_J}g(s)ds_J(\overline{s}B-Bp)Q_s(p)^{-1}.
\end{align*}
In the second case $p\in\mathbb{R}^{n+1}\setminus\overline{D_\varphi}$ the integral \eqref{Eq_Product_identity_1} becomes
\begin{equation*}
0=\frac{1}{2\pi}\int_{\gamma\oplus\sigma\ominus\tau}g(s)ds_J(\overline{s}B-Bp)Q_s(p)^{-1},
\end{equation*}
because the point $p$ lies outside the path $\gamma\oplus\sigma\ominus\tau$. The same calculation as above then gives
\begin{equation*}
0=\frac{1}{2\pi}\int_{\partial D_\varphi\cap\mathbb{C}_J}g(s)ds_J(\overline{s}B-Bp)Q_s(p)^{-1}. \qedhere
\end{equation*}
\end{proof}

One of the most important properties of the $\omega$-functional calculus, also in view of the upcoming generalization to a $H^\infty$-functional calculus in Section~\ref{sec_Hinfty}, is the following product rule.

\begin{thm}\label{thm_Product_omega}
Let $T\in\mathcal{K}(V)$ be bisectorial of angle $\omega\in(0,\frac{\pi}{2})$. Then, for $g\in\mathcal{N}^0(D_\theta)$ and $f\in\mathcal{SH}_L^0(D_\theta)$, $\theta\in(\omega,\frac{\pi}{2})$, we have $gf\in\mathcal{SH}_L^0(D_\theta)$ and there holds the product rule
\begin{equation}\label{Eq_Product_omega}
(gf)(T)=g(T)f(T).
\end{equation}
\end{thm}

\begin{proof}
Since $g\in\mathcal{N}^0(D_\theta)$ and $f\in\mathcal{SH}_L^0(D_\theta)$, we know that $gf$ is left slice hyperholomorphic and also bounded. Using the Hölder inequality with respect to the measure $\frac{dt}{t}$, we obtain
\begin{align*}
\int_0^\infty\big|g(te^{J\phi})f(te^{J\phi})\big|\frac{dt}{t}&\leq\bigg(\int_\mathbb{R}|g(te^{J\phi})|^2\frac{dt}{t}\bigg)^{\frac{1}{2}}\bigg(\int_\mathbb{R}|f(te^{J\varphi})|^2\frac{dt}{t}\bigg)^{\frac{1}{2}} \\
&\leq\sqrt{g_\max\,f_\max}\bigg(\int_\mathbb{R}|g(te^{J\phi})|\frac{dt}{t}\bigg)^{\frac{1}{2}}\bigg(\int_\mathbb{R}|f(te^{J\phi})|\frac{dt}{t}\bigg)^{\frac{1}{2}}<\infty,
\end{align*}
where in the last line we used that
\begin{equation*}
g_\max:=\sup\limits_{t>0}|g(te^{J\phi})|<\infty\qquad\text{and}\qquad f_\max:=\sup\limits_{t>0}|f(te^{J\phi})|<\infty
\end{equation*}
are finite by Lemma~\ref{lem_f_convergence}. This estimate shows that indeed $gf\in\mathcal{SH}_L^0(D_\theta)$. For the proof of the actual product rule \eqref{Eq_Product_omega}, let $\varphi_2<\varphi_1\in(\omega,\theta)$ and $J\in\mathbb{S}$. Since $g\in\mathcal{N}^0(D_\theta)$, we can use the representation \eqref{Eq_Representation_intrinsic} for intrinsic functions, to write the product as
\begin{align}
g(T)f(T)&=\frac{1}{4\pi^2}\int_{\partial D_{\varphi_1}\cap\mathbb{C}_J}g(s)ds_JS_R^{-1}(s,T)\int_{\partial D_{\varphi_2}\cap\mathbb{C}_J}S_L^{-1}(p,T)dp_Jf(p) \notag \\
&=\frac{1}{4\pi^2}\int_{\partial D_{\varphi_1}\cap\mathbb{C}_J}g(s)ds_J\int_{\partial D_{\varphi_2}\cap\mathbb{C}_J}\big(S_R^{-1}(s,T)p-S_L^{-1}(p,T)p-\overline{s}S_R^{-1}(s,T) \notag \\
&\hspace{7cm}+\overline{s}S_L^{-1}(p,T)\big)Q_s(p)^{-1}dp_Jf(p), \label{Eq_Product_omega_1}
\end{align}
where in the last line we used the $S$-resolvent identity \cite[Theorem 5.1.7]{FJBOOK}, namely
\begin{equation}\label{Eq_Product_omega_2}
S_R^{-1}(s,T)S_L^{-1}(p,T)=\big(S_R^{-1}(s,T)p-S_L^{-1}(p,T)p-\overline{s}S_R^{-1}(s,T)+\overline{s}S_L^{-1}(p,T)\big)Q_s(p)^{-1}.
\end{equation}
Since $\varphi_1>\varphi_2$ and $s\in\partial D_{\varphi_1}\cap\mathbb{C}_J$, we know that the mappings $p\mapsto pQ_s(p)^{-1}$ and $p\mapsto Q_s(p)^{-1}$ are intrinsic on $D_{\varphi_1}$. Hence, by the Cauchy formula there vanish the integrals
\begin{equation*}
\int_{\partial D_{\varphi_2}\cap\mathbb{C}_J}pQ_s(p)^{-1}dp_Jf(p)=0\qquad\text{and}\qquad\int_{\partial D_{\varphi_2}\cap\mathbb{C}_J}Q_s(p)^{-1}dp_Jf(p)=0,
\end{equation*}
where the path which closes $\partial D_{\varphi_2}\cap\mathbb{C}_J$ at $0$ and $\infty$ vanishes by Lemma~\ref{lem_Integral_vanish} and the asymptotics
\begin{align*}
|pQ_s(p)^{-1}|\leq\frac{|p|}{(|p|-|s|)^2}=\begin{cases} \mathcal{O}\big(\frac{1}{|p|}\big), & \text{as }|p|\rightarrow\infty, \\ \mathcal{O}(|p|), & \text{as }|p|\rightarrow 0, \end{cases} \\
|Q_s(p)^{-1}|\leq\frac{1}{(|p|-|s|)^2}=\begin{cases} \mathcal{O}\big(\frac{1}{|p|^2}\big), & \text{as }|p|\rightarrow\infty, \\ \mathcal{O}(1), & \text{as }|p|\rightarrow 0. \end{cases}
\end{align*}
This reduces \eqref{Eq_Product_omega_1} to
\begin{equation*}
g(T)f(T)=\frac{1}{4\pi^2}\int_{\partial D_{\varphi_1}\cap\mathbb{C}_J}g(s)ds_J\int_{\partial D_{\varphi_2}\cap\mathbb{C}_J}\big(\overline{s}S_L^{-1}(p,T)-S_L^{-1}(p,T)p\big)Q_s(p)^{-1}dp_Jf(p).
\end{equation*}
Since every $p\in\partial D_{\varphi_2}\cap\mathbb{C}_J$ is inside the integration path $\partial D_{\varphi_1}\cap\mathbb{C}_J$, the integral identity \eqref{Eq_Product_identity}, with the bounded operator $B=S_L^{-1}(p,T)$, further reduces this integral to
\begin{align*}
g(T)f(T)&=\frac{1}{2\pi}\int_{\partial D_{\varphi_2}\cap\mathbb{C}_J}S_L^{-1}(p,T)g(p)dp_Jf(p) \\
&=\frac{1}{2\pi}\int_{\partial D_{\varphi_2}\cap\mathbb{C}_J}S_L^{-1}(p,T)dp_Jg(p)f(p)=(gf)(T). \qedhere
\end{align*}
\end{proof}

Next we want to introduce the so called {\it polynomial functional calculus}, a calculus which simply replaces the argument $s$ of a polynomial by some operator $T$. In Proposition~\ref{prop_Product_omega_pf}, Theorem~\ref{thm_Rational_omega} and Corollary~\ref{cor_Composition_omega}, we then show that this polynomial functional calculus is well compatible with the $\omega$-functional calculus \eqref{Eq_Omega}.

\begin{defi}[Polynomial functional calculus]\label{defi_Polynomial}
Let $T\in\mathcal{K}(V)$ and consider a polynomial $p(s)=p_0+p_1s+\dots+p_ns^n$ with real coefficients $p_0,\dots,p_n\in\mathbb{R}$ and $p_n\neq 0$. Then we define the {\it polynomial functional calculus} as the operator
\begin{equation}\label{Eq_Polynomial}
p[T]:=p_0+p_1T+\dots+p_nT^n,\qquad\text{with }\dom(p[T]):=\dom(T^n).
\end{equation}
\end{defi}

Note, that the $\omega$-functional calculus is indicated by round brackets $f(T)$ and the polynomial calculus by square brackets $p[T]$.

\begin{rem}
We point out the polynomials of the type $p(s)=p_0+p_1s+\dots+p_ns^n$ are intrinsic functions when the variable $s\in\mathbb{R}^{n+1}$ is a paravector. So in the definition of the polynomial functional calculus \eqref{Eq_Polynomial}, the natural replacement of $s$ would be by an operator $T$ that is a paravector operator $T=T_0+T_1e_1+\dots+T_ne_n$. However, in view of the universality property of the $S$-functional calculus , see \cite{ADVCGKS}, we can replace $s$ by any fully Clifford operator $T=\sum_AT_Ae_A$ and still obtain the compatibility with the $\omega$-functional calculus.
\end{rem}

We start considering the polynomial functional calculus with some result which is basically a corollary of the product rule in Theorem~\ref{thm_Product_omega}.

\begin{cor}\label{cor_Composition_omega}
Let $T\in\mathcal{K}(V)$ be bisectorial of angle $\omega\in(0,\frac{\pi}{2})$, $g\in\mathcal{N}^0(D_\theta)$, $\theta\in(\omega,\frac{\pi}{2})$, and $p$ an intrinsic polynomial with $p_0=0$. Then $p\circ g\in\mathcal{SH}_L^0(D_\theta)$, and
\begin{equation*}
(p\circ g)(T)=p[g(T)].
\end{equation*}
\end{cor}

\begin{proof}
Since there is $g\in\mathcal{N}^0(D_\theta)$, there is also $g^k\in\mathcal{N}^0(D_\theta)$ for every $k\geq 1$ by Theorem~\ref{thm_Product_omega}, and hence also
\begin{equation*}
(p\circ g)(s)=p(g(s))=p_1g(s)+\dots+p_ng(s)^n\in\mathcal{N}^0(D_\theta).
\end{equation*}
The product rule \eqref{Eq_Product_omega} and the linearity \eqref{Eq_Linearity_omega} then give
\begin{equation*}
(p\circ g)(T)=(p_1g+\dots+p_ng^n)(T)=p_1g(T)+\dots+p_ng(T)^n=p[g(T)]. \qedhere
\end{equation*}
\end{proof}

\begin{prop}\label{prop_Product_omega_pf}
Let $T\in\mathcal{K}(V)$ be bisectorial of angle $\omega\in(0,\frac{\pi}{2})$, $f\in\mathcal{SH}_L^0(D_\theta)$, $\theta\in(\omega,\frac{\pi}{2})$, and $p$ an intrinsic polynomial of degree $n\in\mathbb{N}_0$. If $pf\in\mathcal{SH}_L^0(D_\theta)$, then there is $\ran(f(T))\subseteq\dom(T^n)$ and
\begin{equation}\label{Eq_Product_omega_pf}
(pf)(T)=p[T]f(T).
\end{equation}
\end{prop}

\begin{proof}
It is enough to prove \eqref{Eq_Product_omega_pf} for monomials $p(s)=s^n$, the case of general polynomials then follows by linearity. We will use induction with respect to $n\in\mathbb{N}_0$. Since the induction start $n=0$ is trivial, we only have to prove the induction step $n\rightarrow n+1$. First, since $|s|^n\leq C_n(1+|s|^{n+1})$, for some $C_n\geq 0$, it follows from $f\in\mathcal{SH}_L^0(D_\theta)$ and $s^{n+1}f\in\mathcal{SH}_L^0(D_\theta)$, that also $s^nf(s)\in\mathcal{SH}_L^0(D_\theta)$. This means, we are allowed to use the induction assumption $\ran(f(T))\subseteq\dom(T^n)$ and $(s^nf)(T)=T^nf(T)$. Next, using the $S$-resolvent identity
\begin{equation}\label{Eq_Product_omega_pf_1}
S_L^{-1}(s,T)s=TS_L^{-1}(s,T)+1,
\end{equation}
which is a direct consequence of the definition of $S_L^{-1}(s,T)$ in \eqref{Eq_S_resolvent}, we can rewrite the $\omega$-functional calculus of $(s^{n+1}f)(T)$ as
\begin{align*}
(s^{n+1}f)(T)&=\frac{1}{2\pi}\int_{\partial D_\varphi\cap\mathbb{C}_J}S_L^{-1}(s,T)ds_Js^{n+1}f(s) \\
&=\frac{1}{2\pi}\int_{\partial D_\varphi\cap\mathbb{C}_J}\big(TS_L^{-1}(s,T)+1\big)ds_Js^nf(s) \\
&=\frac{1}{2\pi}\int_{\partial D_\varphi\cap\mathbb{C}_J}TS_L^{-1}(s,T)ds_Js^nf(s).
\end{align*}
Note, that in the last line we used that due to the slice hyperholomorphicity of $s^nf$ on $D_\theta$, there vanishes the integral
\begin{equation*}
\int_{\partial D_\varphi\cap\mathbb{C}_J}1ds_Js^nf(s)=0.
\end{equation*}
Since $T\in\mathcal{K}(V)$ is a closed operator, it then follows from Hille's theorem \cite{HP57}, that there holds $\ran((s^nf)(T))=\ran\big(\int_{\partial D_\varphi\cap\mathbb{C}_J}S_L^{-1}(s,T)ds_Js^nf(s)\big)\subseteq\dom(T)$, as well as
\begin{equation*}
(s^{n+1}f)(T)=\frac{1}{2\pi}T\int_{\partial(D_\varphi\cap\mathbb{C}_J)}S_L^{-1}(s,T)ds_Js^nf(s)=T(s^nf)(T).
\end{equation*}
Combining this identity with the induction assumption $\ran(f(T))\subseteq\dom(T^n)$ and $(s^nf)(T)=T^nf(T)$, then gives $\ran(f(T))\subseteq\dom(T^{n+1})$, and
\begin{equation*}
(s^{n+1}f)(T)=TT^nf(T)=T^{n+1}f(T). \qedhere
\end{equation*}
\end{proof}

\begin{thm}\label{thm_Rational_omega}
Let $T\in\mathcal{K}(V)$ be bisectorial of angle $\omega\in(0,\frac{\pi}{2})$ and $p,q$ two intrinsic polynomials, such that

\begin{enumerate}
\item[i)] $\deg(q)\geq\deg(p)+1$,
\item[ii)] $p$ admits a zero at the origin,
\item[iii)] $q$ does not admit any zeros in $\overline{D_\omega}$.
\end{enumerate}

Then $q[T]$ is bijective, there exists some $\theta\in(\omega,\pi)$ such that $\frac{p}{q}\in\mathcal{N}^0(D_\theta)$, and there holds
\begin{equation*}
\Big(\frac{p}{q}\Big)(T)=p[T]q[T]^{-1}.
\end{equation*}
\end{thm}

\begin{proof}
It is already proven in \cite[Lemma 3.11]{MPS23}, that the operator $q[T]$ is bijective and that for any bounded, open, axially symmetric $U\subseteq\rho_S(T)$, which contains all zeros of $q$, there is
\begin{equation*}
q[T]^{-1}=\frac{-1}{2\pi}\int_{\partial U\cap\mathbb{C}_J}S_L^{-1}(s,T)ds_J\frac{1}{q(s)}.
\end{equation*}
Moreover, it can be seen in the proof of \cite[Lemma 3.11]{MPS23} that $\ran\big(q[T]^{-1}\big)\subseteq\dom(T)$ and
\begin{equation}\label{Eq_Rational_omega_2}
Tq[T]^{-1}=\frac{-1}{2\pi}\int_{\partial U\cap\mathbb{C}_J}S_L^{-1}(s,T)ds_J\frac{s}{q(s)}.
\end{equation}
Let us now choose some particular set $U$, namely the one where $\partial U\cap\mathbb{C}_J$ is parametrized by

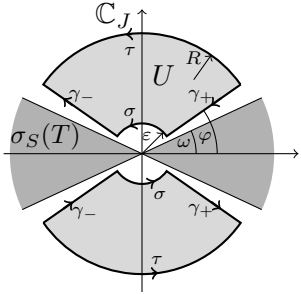
\begin{minipage}{0.34\textwidth}
\begin{center}
\begin{tikzpicture}[scale=0.8]
\fill[black!30] (0,0)--(1.99,-0.93) arc (-25:25:2.2)--(0,0)--(-1.99,-0.93) arc (205:155:2.2);
\draw[thick,fill=black!15] (1.64,1.15)--(0.41,0.29) arc (35:145:0.5)--(-1.64,1.15) arc (145:35:2);
\draw[thick,fill=black!15] (-1.64,-1.15)--(-0.41,-0.29) arc (215:325:0.5)--(1.64,-1.15) arc (325:215:2);
\draw (-1.99,-0.93)--(1.99,0.93);
\draw (-1.99,0.93)--(1.99,-0.93);
\draw[->] (-2.3,0)--(2.6,0);
\draw[->] (0,-2.3)--(0,2.4);
\draw[thick,->] (1.31,0.92)--(1.15,0.8);
\draw[thick,->] (1.15,-0.8)--(1.31,-0.92);
\draw[thick,->] (-1.15,0.8)--(-1.31,0.92);
\draw[thick,->] (-1.31,-0.92)--(-1.15,-0.8);
\draw (0.6,0.95) node[anchor=west] {\tiny{$\gamma_+$}};
\draw (0.6,-0.95) node[anchor=west] {\tiny{$\gamma_+$}};
\draw (-0.6,0.95) node[anchor=east] {\tiny{$\gamma_-$}};
\draw (-0.55,-1) node[anchor=east] {\tiny{$\gamma_-$}};
\draw[thick,->] (-0.2,0.455)--(-0.21,0.45);
\draw[thick,->] (0.2,-0.455)--(0.21,-0.45);
\draw (-0.2,0.45) node[anchor=south] {\tiny{$\sigma$}};
\draw (0.3,-0.45) node[anchor=north] {\tiny{$\sigma$}};
\draw[thick,->] (-0.17,1.99)--(-0.27,1.98);
\draw[thick,->] (0.17,-1.99)--(0.27,-1.98);
\draw (-0.2,1.95) node[anchor=north] {\tiny{$\tau$}};
\draw (0.25,-2) node[anchor=south] {\tiny{$\tau$}};
\draw (0,1.3) node[anchor=west] {$U$};
\draw (-1.6,-0.1) node[anchor=south] {\small{$\sigma_S(T)$}};
\draw (0.9,0) arc (0:25:0.9);
\draw (1.25,0) arc (0:35:1.25);
\draw (0.7,-0.05) node[anchor=south] {\tiny{$\omega$}};
\draw (1.05,0) node[anchor=south] {\tiny{$\varphi$}};
\draw[->] (0,0)--(0.35,0.35);
\draw[->] (0.86,1.23)--(1.15,1.64);
\draw (0.33,0.3) node[anchor=east] {\tiny{$\varepsilon$}};
\draw (1.15,1.6) node[anchor=east] {\tiny{$R$}};
\draw (0,2.3) node[anchor=east] {$\mathbb{C}_J$};
\end{tikzpicture}
\end{center}
\end{minipage}
\begin{minipage}{0.65\textwidth}
\begin{align*}
\gamma_\pm(t)&:=\pm|t|e^{-\sgn(t)J\varphi},\qquad t\in(-R,R)\setminus[-\varepsilon,\varepsilon], \\
\tau(\phi)&:=Re^{J\phi},\hspace{2.1cm}\phi\in\Big(-\frac{\pi}{2},\frac{3\pi}{2}\Big)\setminus\overline{I_\varphi}, \\
\sigma(\phi)&:=\varepsilon e^{J\phi},\hspace{2.22cm}\phi\in\Big(-\frac{\pi}{2},\frac{3\pi}{2}\Big)\setminus\overline{I_\varphi}, \\
\gamma&:=\gamma_+\oplus\gamma_-.
\end{align*}
\end{minipage}

Note that it is possible to choose $\varphi\in(\omega,\frac{\pi}{2})$ and $0<\varepsilon<R$ such that $U$ contains all the zeros of $q$. With this particular set $U$, the integral \eqref{Eq_Rational_omega_2} becomes
\begin{equation*}
Tq[T]^{-1}=\frac{-1}{2\pi}\int_{\tau\ominus\gamma\ominus\sigma}S_L^{-1}(s,T)ds_J\frac{s}{q(s)}.
\end{equation*}
Since $q$ does not have any zeros at the origin, we obtain the asymptotics
\begin{equation*}
\frac{s}{q(s)}=\begin{cases} \mathcal{O}(\frac{1}{|s|^{\deg(q)-1}}), & \text{as }|s|\rightarrow\infty, \\ \mathcal{O}(|s|), & \text{as }|s|\rightarrow 0^+. \end{cases}
\end{equation*}
This in particular means that $\frac{s}{q(s)}\in\mathcal{N}^0(D_\varphi)$. When we now perform the limits $\varepsilon\rightarrow 0^+$ and $R\rightarrow\infty$, the respective integrals along $\sigma$ and $\tau$ vanish by Lemma~\ref{lem_Integral_vanish}, due to the estimate $\Vert S_L^{-1}(s,T)\Vert\leq\frac{C_\varphi}{|s|}$ from \eqref{Eq_SL_estimate}. Hence we obtain the $\omega$-functional calculus
\begin{equation}\label{Eq_Rational_omega_3}
Tq[T]^{-1}=\frac{-1}{2\pi}\lim\limits_{\varepsilon\rightarrow 0^+}\lim\limits_{R\rightarrow\infty}\int_{\ominus\gamma}S_L^{-1}(s,T)ds_J\frac{s}{q(s)}=\Big(\frac{s}{q}\Big)(T).
\end{equation}
Finally, since $p$ admits a zero at $s=0$, we can write it as $p(s)=s\widetilde{p}(s)$, for another polynomial $\widetilde{p}(s)$. By the assumptions on $p$ and $q$, we now know that
\begin{equation*}
\widetilde{p}(s)\frac{s}{q(s)}=\frac{p(s)}{q(s)}=\begin{cases} \mathcal{O}(|s|), & \text{as }|s|\rightarrow 0^+, \\ \mathcal{O}(\frac{1}{|s|}), & \text{as }|s|\rightarrow\infty, \end{cases}
\end{equation*}
which in particular means that $\widetilde{p}\,\frac{s}{q}\in\mathcal{N}^0(D_\theta)$. From Proposition~\ref{prop_Product_omega_pf} there then follows that $\ran((\frac{s}{q})(T))\subseteq\dom(T^{\deg(\widetilde{p})})=\dom(T^{\deg(p)-1})$, and
\begin{equation*}
\Big(\frac{p}{q}\Big)(T)=\Big(\widetilde{p}\,\frac{s}{q}\Big)(T)=\widetilde{p}[T]\Big(\frac{s}{q}\Big)(T).
\end{equation*}
Using also \eqref{Eq_Rational_omega_3}, gives $\ran(Tq[T]^{-1})\subseteq\dom(T^{\deg(p)-1})$, i.e. $\ran(q[T]^{-1})\subseteq\dom(T^{\deg(p)})$, and
\begin{equation*}
\Big(\frac{p}{q}\Big)(T)=\widetilde{p}[T]Tq[T]^{-1}=p[T]q[T]^{-1}. \qedhere
\end{equation*}
\end{proof}

In the final part of this section we want to investigate under which circumstances operators commute with the $\omega$-functional calculus \eqref{Eq_Omega}. At this point we recall the notion of operator inclusion in \eqref{Eq_Operator_inclusion}.

\begin{thm}\label{thm_Commutation_omega}
Let $T\in\mathcal{K}(V)$ be bisectorial of angle $\omega\in(0,\frac{\pi}{2})$ and $A\in\mathcal{K}(V)$, such that
\begin{equation}\label{Eq_Commutation_omega_assumption}
AQ_s[T]^{-1}\supseteq Q_s[T]^{-1}A\qquad\text{and}\qquad ATQ_s[T]^{-1}\supseteq TQ_s[T]^{-1}A,
\end{equation}
for every $s\in\mathbb{R}^{n+1}\setminus\overline{D_\omega}$. Then for every $g\in\mathcal{N}^0(D_\theta)$, $\theta\in(\omega,\frac{\pi}{2})$, there commutes
\begin{equation}\label{Eq_Commutation_omega}
Ag(T)\supseteq g(T)A.
\end{equation}
\end{thm}

\begin{proof}
According to \eqref{Eq_Representation_intrinsic_3}, we can write
\begin{equation*}
g(T)=\frac{1}{\pi}\sum_\pm\int_0^\infty\big(A_\pm(t)\alpha_\pm(t)-B_\pm(t)\beta_\pm(t)\big)dt,
\end{equation*}
using the real valued functions $\alpha_\pm,\beta_\pm$ from \eqref{Eq_Representation_intrinsic_2} and the operators \eqref{Eq_Representation_intrinsic_1}
\begin{equation*}
A_\pm(t)=(x_\pm(t)-T)Q_{\gamma_\pm(t)}[T]^{-1}\qquad\text{and}\qquad B_\pm(t)=-y_\pm(t)Q_{\gamma_\pm(t)}[T]^{-1}.
\end{equation*}
Due to the assumption \eqref{Eq_Commutation_omega_assumption}, there hold the commutations relations
\begin{equation*}
AA_\pm(t)\supseteq A_\pm(t)A\qquad\text{and}\qquad AB_\pm(t)\supseteq B_\pm(t)A,\qquad t>0.
\end{equation*}
Using these relations in \eqref{Eq_Commutation_omega}, we obtain
\begin{equation*}
g(T)A\subseteq\frac{1}{\pi}\sum_\pm\int_0^\infty A\big(A_\pm(t)\alpha_\pm(t)-B_\pm(t)\beta_\pm(t)\big)dt.
\end{equation*}
Since $A$ is closed, Hille's theorem gives $\ran\big(\int_0^\infty\big(A_\pm(t)\alpha_\pm(t)-B_\pm(t)\beta_\pm(t)\big)dt\big)\subseteq\dom(A)$, and
\begin{equation*}
g(T)A\subseteq\frac{1}{\pi}\sum_\pm A\int_0^\infty\big(A_\pm(t)\alpha_\pm(t)-B_\pm(t)\beta_\pm(t)\big)dt=Ag(T). \qedhere
\end{equation*}
\end{proof}

\begin{cor}\label{cor_Commutation_omega_bounded}
Let $T\in\mathcal{K}(V)$ be bisectorial of angle $\omega\in(0,\frac{\pi}{2})$, $g\in\mathcal{N}^0(D_\theta)$, $\theta\in(\omega,\frac{\pi}{2})$. \medskip

\begin{enumerate}
\item[i)] For every $B\in\mathcal{B}(V)$ with $TB\supseteq BT$, we have
\begin{equation}\label{Eq_Commutation_omega_bounded}
Bg(T)=g(T)B.
\end{equation}

\item[ii)] For every intrinsic polynomial $p$, we get
\begin{equation*}
p[T]g(T)\supseteq g(T)p[T].
\end{equation*}

\item[iii)] For every $f\in\mathcal{N}^0(D_\theta)$, we obtain
\begin{equation*}
g(T)f(T)=f(T)g(T).
\end{equation*}
\end{enumerate}
\end{cor}

\begin{proof}
i)\;\;For every $s\in\mathbb{R}^{n+1}\setminus\overline{D_\omega}$ we have the commutation relation
\begin{equation*}
Q_s[T]B=(T^2-2s_0T+|s|^2)B\supseteq B(T^2-2s_0T+|s|^2)=BQ_s[T].
\end{equation*}
Using also $Q_s[T]^{-1}Q_s[T]\subseteq 1$ and $Q_s[T]Q_s[T]^{-1}=1$, we get
\begin{equation*}
BQ_s[T]^{-1}\supseteq Q_s[T]^{-1}Q_s[T]BQ_s[T]^{-1}\supseteq Q_s[T]^{-1}BQ_s[T]Q_s[T]^{-1}=Q_s[T]^{-1}B.
\end{equation*}
Moreover, since $TB\supseteq BT$ implies $TB=BT$ on $\dom(T)$, and since $\ran(Q_s[T]^{-1})\subseteq\dom(T)$, we get the additional commutation relation
\begin{equation}\label{Eq_Commutation_omega_bounded_1}
BTQ_s[T]^{-1}=TBQ_s[T]^{-1}\supseteq TQ_s[T]^{-1}B.
\end{equation}
Hence the assumptions \eqref{Eq_Commutation_omega_assumption} of Theorem~\ref{thm_Commutation_omega} are satisfied, and we get $Bg(T)\supseteq g(T)B$. However, since the right hand side of \eqref{Eq_Commutation_omega_bounded_1} is everywhere defined, we get the equality \eqref{Eq_Commutation_omega_bounded}. \medskip

ii)\;\;For every $s\in\mathbb{R}^{n+1}$, there is $TQ_s[T]=Q_s[T]T$. Using also $Q_s[T]^{-1}Q_s[T]\subseteq 1$ and $Q_s[T]Q_s[T]^{-1}=1$, we obtain  the commutation relation
\begin{equation*}
TQ_s[T]^{-1}\supseteq Q_s[T]^{-1}Q_s[T]TQ_s[T]^{-1}=Q_s[T]^{-1}TQ_s[T]Q_s[T]^{-1}=Q_s[T]^{-1}T.
\end{equation*}
Consequently, we also get
\begin{equation*}
TTQ_s[T]^{-1}\supseteq TQ_s[T]^{-1}T,
\end{equation*}
and hence the assumptions of Theorem~\ref{thm_Commutation_omega} are satisfied for $A=T$ and we conclude
\begin{equation*}
Tg(T)\supseteq g(T)T.
\end{equation*}
The commutation $p[T]g(T)\supseteq g(T)p[T]$ then follows inductively. \medskip

iii)\;\;In ii) we have already shown that $Tg(T)\supseteq g(T)T$. Since $g(T)$ is a bounded operator, it then follows from i), that $g(T)f(T)=g(T)f(T)$.
\end{proof}

\section{Extended $\omega$-functional calculus}
\label{sec_Extended}

One main purpose of the $\omega$-functional calculus in Section~\ref{sec_Omega} is, to serve as a functional calculus for the regularizers in the $H^\infty$-functional calculus in Section~\ref{sec_Hinfty}. However, since $\mathcal{SH}_L^0(D_\theta)$ only includes functions which vanish at zero and at infinity, this class is too small, in particular for non-injective operators $T$. This is why in this section we will also give meaning to $f(T)$ for functions which have finite values at zero and at infinity, see Definition~\ref{defi_SH_bnd}. Another reason why the extended $\omega$-functional makes sense, is that it gives a larger class of functions for which the operator $f(T)$ is still bounded. See also \cite{CGdiffusion2018}, where this method was already used.

\begin{defi}\label{defi_SH_bnd}
For every $\theta\in(0,\frac{\pi}{2})$ let $D_\theta$ be the double sector from \eqref{Eq_Domega}. Then we define the function spaces
\begin{align*}
\text{i)}\;\;&\mathcal{SH}_L^\bnd(D_\theta):=\Big\{f(s)=f_\infty+\frac{1}{1+s^2}(f_0-f_\infty)+\widetilde{f}(s)\;\Big|\;f_0,f_\infty\in\mathbb{R}_n,\,\widetilde{f}\in\mathcal{SH}_L^0(D_\theta)\Big\}, \\
\text{ii)}\;\;&\mathcal{N}^\bnd(D_\theta):=\Big\{g(s)=g_\infty+\frac{1}{1+s^2}(g_0-g_\infty)+\widetilde{g}(s)\;\Big|\;g_0,g_\infty\in\mathbb{R},\,\widetilde{g}\in\mathcal{N}^0(D_\theta)\Big\}.
\end{align*}
\end{defi}

Note, that for every $f\in\mathcal{SH}_L^\bnd(D_\theta)$ the components $f_0,f_\infty\in\mathbb{R}_n$, $\widetilde{f}\in\mathcal{SH}_L^0(D_\theta)$ are unique. Moreover, by Lemma~\ref{lem_f_convergence}, the function $f$ clearly admits for every $\varphi\in(0,\theta)$ the limits
\begin{equation*}
\lim\limits_{t\rightarrow 0^+}\sup\limits_{J\in\mathbb{S},\phi\in I_\varphi}|f(te^{J\phi})-f_0|=0\qquad\text{and}\qquad\lim\limits_{t\rightarrow\infty}\sup\limits_{J\in\mathbb{S},\phi\in I_\varphi}|f(te^{J\phi})-f_\infty|=0.
\end{equation*}

The next lemma gives a useful characterization of the spaces $\mathcal{SH}_L^\bnd(D_\theta)$ and $\mathcal{N}^\bnd(D_\theta)$.

\begin{lem}\label{lem_Characterization_SHbnd}
Let $\theta\in(0,\frac{\pi}{2})$ and $f\in\mathcal{SH}_L(D_\theta)$ (resp. $f\in\mathcal{N}(D_\theta)$) be bounded. Then there is $f\in\mathcal{SH}_L^\bnd(D_\theta)$ (resp. $f\in\mathcal{N}^\bnd(D_\theta)$) if and only if there exists $0<r<R$ and $f_0,f_\infty\in\mathbb{R}_n$ (resp. $f_0,f_\infty\in\mathbb{R}$) such that for every $J\in\mathbb{S}$ and $\phi\in I_\theta$ there holds
\begin{equation}\label{Eq_Characterization_SHbnd}
\int_0^r|f(te^{J\phi})-f_0|\frac{dt}{t}<\infty\qquad\text{and}\qquad\int_R^\infty|f(te^{J\phi})-f_\infty|\frac{dt}{t}<\infty.
\end{equation}
\end{lem}

\begin{proof}
For the first implication assume that $f\in\mathcal{SH}_L^\bnd(D_\theta)$, i.e., we can write
\begin{equation*}
f(s)=f_\infty+\frac{1}{1+s^2}(f_0-f_\infty)+\widetilde{f}(s),
\end{equation*}
for some $f_0,f_n\in\mathbb{R}_n$ and $\widetilde{f}\in\mathcal{SH}_L^0(D_\theta)$. Rewriting this equation, gives
\begin{equation}\label{Eq_Characterization_SHbnd_1}
f(s)-f_0=\frac{s^2}{1+s^2}(f_\infty-f_0)+\widetilde{f}(s)\qquad\text{and}\qquad f(s)-f_\infty=\frac{1}{1+s^2}(f_0-f_\infty)+\widetilde{f}(s).
\end{equation}
Since $\widetilde{f}\in\mathcal{SH}_L^0(D_\theta)$ we know that $\int_0^\infty|\widetilde{f}(te^{J\phi})|\frac{dt}{t}<\infty$, and hence the representations \eqref{Eq_Characterization_SHbnd_1} show that also the integrals \eqref{Eq_Characterization_SHbnd} are finite. \medskip

For the inverse implication, we assume that there exists $f_0,f_\infty\in\mathbb{R}_n$, such that the integrals \eqref{Eq_Characterization_SHbnd} are finite. Let us define
\begin{equation*}
\widetilde{f}(z):=f(z)-f_\infty-\frac{1}{1+z^2}(f_0-f_\infty).
\end{equation*}
This function then satisfies
\begin{align*}
\int_0^r|\widetilde{f}(te^{J\phi})|\frac{dt}{t}&=\int_0^r|f(te^{J\phi})-f_0|\frac{dt}{t}+|f_0-f_\infty|\int_0^r\frac{t^2}{|1+t^2e^{2J\phi}|}\frac{dt}{t}<\infty, \\
\int_r^R|\widetilde{f}(te^{J\phi})|\frac{dt}{t}&<\infty, \\
\int_R^\infty|\widetilde{f}(te^{J\phi})|\frac{dt}{t}&=\int_R^\infty|f(te^{J\phi})-f_\infty|\frac{dt}{t}+|f_0-f_\infty|\int_R^\infty\frac{2}{|1+t^2e^{2J\phi}|}\frac{dt}{t}<\infty,
\end{align*}
which means that $\widetilde{f}\in\mathcal{SH}_L^0(D_\theta)$. Hence we can write $f(z)=f_\infty+\frac{1}{1+z^2}(f_0-f_\infty)+\widetilde{f}(z)$ for $f_0,f_\infty\in\mathbb{R}_n$ and $\widetilde{f}\in\mathcal{SH}_L^0(D_\theta)$, which means that $f\in\mathcal{SH}_L^\bnd(D_\theta)$.
\end{proof}

\begin{defi}[Extended $\omega$-functional calculus]\label{defi_Extended}
Let $T\in\mathcal{K}(V)$ be bisectorial of angle $\omega\in(0,\frac{\pi}{2})$. Then for every $f\in\mathcal{SH}_L^\bnd(D_\theta)$ we define the {\it extended $\omega$-functional calculus}
\begin{equation}\label{Eq_Extended}
f(T):=f_\infty+(1+T^2)^{-1}(f_0-f_\infty)+\widetilde{f}(T),
\end{equation}
where $f_0,f_\infty\in\mathbb{R}_n$, $\widetilde{f}\in\mathcal{SH}_L^0(D_\theta)$ are the coefficients from Definition~\ref{defi_SH_bnd}. Moreover, $\widetilde{f}(T)$ is understood as the $\omega$-functional calculus in Definition~\ref{defi_Omega}, and the inverse $(1+T^2)^{-1}\in\mathcal{B}(V)$ is bounded since $Q_J[T]=T^2+1$ is boundedly invertible because $\mathbb{S}\subseteq\rho_S(T)$ by Definition~\ref{defi_Bisectorial_operators}.
\end{defi}

For every $f,g\in\mathcal{SH}_L^\bnd(D_\theta)$ and $a\in\mathbb{R}_n$ there holds
\begin{equation*}
(f+g)(T)=f(T)+g(T)\qquad\text{and}\qquad(fa)(T)=f(T)a.
\end{equation*}

\begin{rem}\label{rem_Extended_0rho}
In the special case that $0\in\rho_S(T)$, we already mentioned in Remark~\ref{rem_Omega_0rho} that we are able to extend the $\omega$-functional calculus to functions where only the integrability at $\infty$ matters. Consequently, if $0\in\rho_S(T)$, we can also extend the extended $\omega$-functional calculus to functions $f$ for which there exists some $f_\infty\in\mathbb{R}_n$, such that $\widehat{f}(s):=f(s)-f_\infty$ satisfies the integrability condition \eqref{Eq_Integrability_at_infinity}. Then $\widehat{f}(T)$ is understood as in \eqref{Eq_Omega_0rho}, and we can define
\begin{equation}\label{Eq_Extended_0rho}
f(T):=f_\infty+\widehat{f}(T).
\end{equation}
\end{rem}

The next proposition shows that for certain subclasses of functions, the extended $\omega$-functional calculus coincides with the well established functional calculus for unbounded operators, see e.g. \cite[Definition 3.4.2]{FJBOOK} or \cite[Theorem 4.12]{ColSab2006}.

\begin{prop}
Let $T\in\mathcal{K}(V)$ be bisectorial of angle $\omega\in(0,\frac{\pi}{2})$ and $f\in\mathcal{SH}_L(\mathbb{R}^{n+1}\setminus K)$ for some axially symmetric compact $K\subseteq\mathbb{R}^{n+1}\setminus\overline{D_\omega}$. Moreover, assume that $\lim_{|s|\rightarrow\infty}f(s)=f_\infty$, for some $f_\infty\in\mathbb{R}_n$. Then there exists some $\theta\in(\omega,\frac{\pi}{2})$ such that $f\in\mathcal{SH}_L^\bnd(D_\theta)$ and the extended $\omega$-functional calculus \eqref{Eq_Extended} can be written as
\begin{equation}\label{Eq_Unbounded_extended}
f(T)=f_\infty-\frac{1}{2\pi}\int_{\partial U\cap\mathbb{C}_J}S_L^{-1}(s,T)ds_Jf(s),
\end{equation}
where $J\in\mathbb{S}$ is arbitrary and $U$ is some axially symmetric open set with $K\subseteq U\subseteq\overline{U}\subseteq\rho_S(T)$.
\end{prop}

\begin{proof}
Since the compact set $K$ is located outside the closed double sector $\overline{D_\omega}$, there exists some $\theta\in(\omega,\frac{\pi}{2})$, such that $K\subseteq\mathbb{R}^{n+1}\setminus\overline{D_\theta}$. This means that $f$ can be restricted to a function $f\in\mathcal{SH}_L(D_\theta)$. Since $f$ is slice hyperholomorphic on $\mathbb{R}^{n+1}\setminus K$, the value $f_0:=f(0)$ exists. Let us now define the set $\widetilde{K}:=K\cup[J]$, as well as the function
\begin{equation}\label{Eq_Unbounded_extended_2}
\widetilde{f}(s):=f(s)-f_\infty-\frac{1}{1+s^2}(f_0-f_\infty),\qquad s\in\mathbb{R}^{n+1}\setminus\widetilde{K}.
\end{equation}
Then $\widetilde{f}\in\mathcal{SH}_L(\mathbb{R}^{n+1}\setminus\widetilde{K})$ with values $\widetilde{f}(0)=0$ and $\lim_{|s|\rightarrow\infty}\widetilde{f}(s)=0$. In the same way as in Proposition~\ref{prop_Unbounded_omega}, we conclude that $\widetilde{f}\in\mathcal{SH}_L^0(D_\theta)$. Consequently there is $f\in\mathcal{SH}_L^\bnd(D_\theta)$. According to Definition~\ref{defi_Extended} and Proposition~\ref{prop_Unbounded_omega}, the extended $\omega$ functional calculus for $f$ is given by
\begin{equation}\label{Eq_Unbounded_extended_1}
f(T)=f_\infty+(1+T^2)^{-1}(f_0-f_\infty)-\frac{1}{2\pi}\int_{\partial\widetilde{U}\cap\mathbb{C}_J}S_L^{-1}(s,T)ds_J\widetilde{f}(s),
\end{equation}
where $\widetilde{U}$ is some axially symmetric open set with $\widetilde{K}\subseteq\widetilde{U}\subseteq\overline{\widetilde{U}}\subseteq\rho_S(T)$. Moreover, it follows from the hyperholomorphicity of the constant function $1$, that
\begin{equation*}
\int_{\partial\widetilde{U}\cap\mathbb{C}_J}S_L^{-1}(s,T)ds_J=0,
\end{equation*}
and from \cite[Equation (54)]{MPS23}, that
\begin{equation*}
\frac{-1}{2\pi}\int_{\partial\widetilde{U}\cap\mathbb{C}_J}S_L^{-1}(s,T)ds_J\frac{1}{1+s^2}=(1+T^2)^{-1}.
\end{equation*}
These two identities, together with the definition of $\widetilde{f}$ in \eqref{Eq_Unbounded_extended_2}, reduce \eqref{Eq_Unbounded_extended_1} to
\begin{equation*}
f(T)=f_\infty-\frac{1}{2\pi}\int_{\partial\widetilde{U}\cap\mathbb{C}_J}S_L^{-1}(s,T)ds_Jf(s).
\end{equation*}
Finally, since $f$ is slice hyperholomorphic on $\mathbb{R}^{n+1}\setminus K$, we can replace the set $\widetilde{U}$, which is a superset of $\widetilde{K}=K\cup[J]$, by any axially symmetric open set $U$ with $K\subseteq U\subseteq\overline{U}\subseteq\rho_S(T)$. So we have proven the formula \eqref{Eq_Unbounded_extended}.
\end{proof}

\begin{prop}\label{prop_Kernel_extended}
Let $T\in\mathcal{K}(V)$ be bisectorial of angle $\omega\in(0,\frac{\pi}{2})$ and $g\in\mathcal{N}^\bnd(D_\theta)$, $\theta\in(\omega,\frac{\pi}{2})$. Then, with $g_0\in\mathbb{R}$ from Definition~\ref{defi_SH_bnd}, there holds
\begin{equation*}
\ker(T)\subseteq\ker(g(T)-g_0).
\end{equation*}
\end{prop}

\begin{proof}
Let $v\in\ker(T)$. Then $(1+T^2)v=v$ and hence $(1+T^2)^{-1}v=v$. Moreover, there is also $\widetilde{g}(T)v=0$ by Proposition~\ref{prop_Kernel_omega}. Using this in the definition \eqref{Eq_Extended} of $g(T)$, gives
\begin{equation*}
g(T)v=g_\infty v+(1+T^2)^{-1}(g_0-g_\infty)v+\widetilde{g}(T)v=g_\infty v+(g_0-g_\infty)v=g_0v. \qedhere
\end{equation*}
\end{proof}

Next, we generalize the product rule of Theorem~\ref{thm_Product_omega} to the extended $\omega$-functional calculus.

\begin{thm}\label{thm_Product_extended}
Let $T\in\mathcal{K}(V)$ be bisectorial of angle $\omega\in(0,\frac{\pi}{2})$. Then, for $g\in\mathcal{N}^\bnd(D_\theta)$ and $f\in\mathcal{SH}_L^\bnd(D_\theta)$, $\theta\in(\omega,\frac{\pi}{2})$, we have $gf\in\mathcal{SH}_L^\bnd(D_\theta)$ and there holds the product rule
\begin{equation}\label{Eq_Product_extended}
(gf)(T)=g(T)f(T).
\end{equation}
\end{thm}

\begin{proof}
Since $g\in\mathcal{N}^\bnd(D_\theta)$ and $f\in\mathcal{SH}_L^\bnd(D_\theta)$, there exists $g_0,g_\infty\in\mathbb{R}$, $\widetilde{g}\in\mathcal{N}^0(D_\theta)$ and $f_0,f_\infty\in\mathbb{R}_n$, $\widetilde{f}\in\mathcal{SH}_L^0(D_\theta)$, such that
\begin{equation*}
g(s)=g_\infty+\frac{g_0-g_\infty}{1+s^2}+\widetilde{g}(s)\qquad\text{and}\qquad f(s)=f_\infty+\frac{1}{1+s^2}(f_0-f_\infty)+\widetilde{f}(s).
\end{equation*}
Hence the product $gf$ admits the decomposition
\begin{align}
g(s)f(s)&=g_\infty f_\infty+\frac{1}{1+s^2}(g_0f_0-g_\infty f_\infty)+g_\infty\widetilde{f}(s)+\widetilde{g}(s)f_\infty+\widetilde{g}(s)\widetilde{f}(s) \notag \\
&\quad+\underbrace{\frac{1}{1+s^2}\widetilde{f}(s)}_{=:\widetilde{h}_1(s)}(g_0-g_\infty)+\underbrace{\frac{\widetilde{g}(s)}{1+s^2}}_{=:\widetilde{h}_2(s)}(f_0-f_\infty)-\underbrace{\frac{s^2}{(1+s^2)^2}}_{=:\widetilde{h}_3(s)}(g_0-g_\infty)(f_0-f_\infty). \label{Eq_Product_extended_1}
\end{align}
Since $\widetilde{g}\in\mathcal{N}^0(D_\theta),\widetilde{f}\in\mathcal{SH}_L^0(D_\theta)$, there is also $\widetilde{h}_1,\widetilde{h}_2,\widetilde{h}_3,\widetilde{g}\widetilde{f}\in\mathcal{SH}_L^0(D_\theta)$ by Theorem~\ref{thm_Product_omega}. Hence the representation \eqref{Eq_Product_extended_1} shows that $gf\in\mathcal{SH}_L^\bnd(D_\theta)$, with coefficients $h_0=g_0f_0\in\mathbb{R}_n$, $h_\infty=g_\infty f_\infty\in\mathbb{R}_n$, and
\begin{equation*}
\widetilde{h}=g_\infty\widetilde{f}+\widetilde{g}f_\infty+\widetilde{g}\widetilde{f}+\widetilde{h}_1(g_0-g_\infty)+\widetilde{h}_2(f_0-f_\infty)-\widetilde{h}_3(g_0-g_\infty)(f_0-f_\infty)\in\mathcal{SH}_L^0(D_\theta).
\end{equation*}
Hence, we can write the extended $\omega$-functional calculus of the product $gf$ as
\begin{equation}\label{Eq_Product_extended_2}
(gf)(T)=g_\infty f_\infty+(1+T^2)^{-1}(g_0f_0-g_\infty f_\infty)+\widetilde{h}(T).
\end{equation}
Let us now further investigate, how to rewrite $\widetilde{h}_1(T)$, $\widetilde{h}_2(T)$ and $\widetilde{h}_3(T)$. First, we note that $(1+s^2)\widetilde{h}_1=\widetilde{f}\in\mathcal{SH}_L^0(D_\theta)$ and hence by Proposition~\ref{prop_Product_omega_pf} there is $\ran(\widetilde{h}_1(T))\subseteq\dom(T^2)$ and
\begin{equation*}
(1+T^2)\widetilde{h}_1(T)=((1+s^2)\widetilde{h}_1)(T)=\widetilde{f}(T).
\end{equation*}
Since $1+T^2$ is boundedly invertible, there is $\widetilde{h}_1(T)=(1+T^2)^{-1}\widetilde{f}(T)$. Analogously, we also get $\widetilde{h}_2(T)=(1+T^2)^{-1}\widetilde{g}(T)$. Finally, by Theorem~\ref{thm_Rational_omega} there is $\widetilde{h}_3(T)=T^2(1+T^2)^{-2}$. Plugging these three operators into \eqref{Eq_Product_extended_2}, we obtain
\begin{align}
(gf)(T)&=g_\infty f_\infty+(1+T^2)^{-1}(g_0f_0-g_\infty f_\infty)+g_\infty\widetilde{f}(T)+\widetilde{g}(T)f_\infty+(\widetilde{g}\widetilde{f})(T) \notag \\
&\quad+\widetilde{h}_1(T)(g_0-g_\infty)+\widetilde{h}_2(T)(f_0-f_\infty)-\widetilde{h}_3(T)(g_0-g_\infty)(f_0-f_\infty) \notag \\
&=g_\infty f_\infty+(1+T^2)^{-1}(g_0f_0-g_\infty f_\infty)+g_\infty\widetilde{f}(T)+\widetilde{g}(T)f_\infty+\widetilde{g}(T)\widetilde{f}(T) \notag \\
&\quad+(1+T^2)^{-1}\widetilde{f}(T)(g_0-g_\infty)+(1+T^2)^{-1}\widetilde{g}(T)(f_0-f_\infty) \notag \\
&\quad-T^2(1+T^2)^{-2}(g_0-g_\infty)(f_0-f_\infty). \label{Eq_Product_extended_3}
\end{align}
On the other hand, we also have
\begin{align}
g&(T)f(T)=\big(g_\infty+(1+T^2)^{-1}(g_0-g_\infty)+\widetilde{g}(T)\big)\big(f_\infty+(1+T^2)^{-1}(f_0-f_\infty)+\widetilde{f}(T)\big) \notag \\
&=g_\infty f_\infty+(1+T^2)^{-2}(g_0-g_\infty)(f_0-f_\infty)+g_\infty\widetilde{f}(T)+\widetilde{g}(T)f_\infty+\widetilde{g}(T)\widetilde{f}(T) \notag \\
&\quad+(1+T^2)^{-1}\big(g_\infty(f_0-f_\infty)+(g_0-g_\infty)f_\infty+(g_0-g_\infty)\widetilde{f}(T)+\widetilde{g}(T)(f_0-f_\infty)\big), \label{Eq_Product_extended_4}
\end{align}
where in the second equation we used that $\widetilde{g}(T)$ and $(1+T^2)^{-1}$ commute by Corollary~\ref{cor_Commutation_omega_bounded}~(i). Since the right hand sides of \eqref{Eq_Product_extended_3} and \eqref{Eq_Product_extended_4} coincide, we deduce the product rule \eqref{Eq_Product_extended}.
\end{proof}

\begin{cor}\label{cor_Composition_extended}
Let $T\in\mathcal{K}(V)$ be bisectorial of angle $\omega\in(0,\frac{\pi}{2})$, $g\in\mathcal{N}^\bnd(D_\theta)$, $\theta\in(\omega,\frac{\pi}{2})$, and $p$ an intrinsic polynomial. Then there is $p\circ g\in\mathcal{N}^\bnd(D_\theta)$ and
\begin{equation*}
(p\circ g)(T)=p[g(T)].
\end{equation*}
\end{cor}

\begin{proof}
The proof is the same as the one of Corollary~\ref{cor_Composition_omega}, with the only difference that here we also allow the constant term $p_0$ in the polynomial. However, this term is clearly contained in the space $\mathcal{N}^\bnd(D_\theta)$.
\end{proof}

The next proposition shows the compatibility of the extended $\omega$-functional calculus \eqref{Eq_Extended} with the polynomial functional calculus.

\begin{prop}\label{prop_Product_extended_pf}
Let $T\in\mathcal{K}(V)$ be bisectorial of angle $\omega\in(0,\frac{\pi}{2})$, $p$ an intrinsic polynomial of degree $n\in\mathbb{N}_0$, and $f\in\mathcal{SH}_L^\bnd(D_\theta)$, $\theta\in(\omega,\frac{\pi}{2})$. If $pf\in\mathcal{SH}_L^\bnd(D_\theta)$, then there is $\ran(f(T))\subseteq\dom(T^n)$ and
\begin{equation}\label{Eq_Product_extended_pf}
(pf)(T)=p[T]f(T).
\end{equation}
\end{prop}

\begin{proof}
Note, that it is enough to prove \eqref{Eq_Product_extended_pf} for monomials $p(s)=s^n$. The case of general polynomials then follows by linearity. We will prove this fact via induction with respect to $n\in\mathbb{N}_0$. The induction start $n=0$ is trivial. For the induction step $n\rightarrow n+1$, we have $f\in\mathcal{SH}_L^\bnd(D_\theta)$ and $s^{n+1}f\in\mathcal{SH}_L^\bnd(D_\theta)$, which means that we can write
\begin{equation}\label{Eq_Product_extended_pf_5}
f(s)=f_\infty+\frac{1}{1+s^2}(f_0-f_\infty)+\widetilde{f}(s)\quad\text{and}\quad s^{n+1}f(s)=g_\infty+\frac{1}{1+s^2}(g_0-g_\infty)+\widetilde{g}(s).
\end{equation}
If we multiply the first equation with $s^{n+1}$ and compare both equations, we obtain
\begin{equation*}
s^{n+1}\Big(f_\infty+\frac{1}{1+s^2}(f_0-f_\infty)+\widetilde{f}(s)\Big)=g_\infty+\frac{1}{1+s^2}(g_0-g_\infty)+\widetilde{g}(s).
\end{equation*}
Since this equation has to be true for every $s\in D_\theta$, we conclude $g_0=0$ when we take the limit $|s|\rightarrow 0^+$. Moreover, in the limit $|s|\rightarrow\infty$ we get
\begin{equation*}
\lim\limits_{|s|\rightarrow\infty}s^{n+1}\Big(f_\infty+\frac{1}{1+s^2}(f_0-f_\infty)+\widetilde{f}(s)\Big)
=g_\infty.
\end{equation*}
The fact that this limit is finite, implies that
\begin{equation*}
\lim_{|s|\rightarrow\infty}\Big(f_\infty+\frac{1}{1+s^2}(f_0-f_\infty)+\widetilde{f}(s)\Big)=0,
\end{equation*}
i.e. $f_\infty=0$ by Lemma~\ref{lem_f_convergence}. The two values $g_0=f_\infty=0$ now reduce the equations \eqref{Eq_Product_extended_pf_5} to
\begin{equation}\label{Eq_Product_extended_pf_6}
f(s)=\frac{1}{1+s^2}f_0+\widetilde{f}(s)\quad\text{and}\quad s^{n+1}f(s)=g_\infty-\frac{1}{1+s^2}g_\infty+\widetilde{g}(s).
\end{equation}
Hence the corresponding extended $\omega$-functional calculi \eqref{Eq_Extended} are given by
\begin{equation}\label{Eq_Product_extended_pf_7}
\begin{split}
f(T)=(1+T^2)^{-1}f_0+\widetilde{f}(T)\quad\text{and}\quad(s^{n+1}f)(T)&=g_\infty-(1+T^2)^{-1}g_\infty+\widetilde{g}(T) \\
&=T^2(1+T^2)^{-1}g_\infty+\widetilde{g}(T).
\end{split}
\end{equation}
Next, we will show that for $h(s):=sf(s)\in\mathcal{SH}_L^\bnd(D_\theta)$ there holds $\ran(f(T))\subseteq\dom(T)$ with
\begin{equation}\label{Eq_Product_extended_pf_1}
h(T)=Tf(T).
\end{equation}
$\circ$\;\;If $n=0$, we can use the two equations \eqref{Eq_Product_extended_pf_6}, to write
\begin{equation*}
s\Big(\widetilde{f}(s)-\frac{s}{1+s^2}g_\infty\Big)=\widetilde{g}(s)-\frac{s}{1+s^2}f_0\in\mathcal{SH}_L^0(D_\theta).
\end{equation*}
Since also $\widetilde{f}(s)-\frac{s}{1+s^2}g_\infty\in\mathcal{SH}_L^0(D_\theta)$, we have $\ran(\widetilde{f}(T)-T(1+T^2)^{-1}g_\infty)\subseteq\dom(T)$ by Proposition~\ref{prop_Product_omega_pf} and
\begin{equation*}
T\big(\widetilde{f}(T)-T(1+T^2)^{-1}g_\infty\big)=\widetilde{g}(T)-T(1+T^2)^{-1}f_0,
\end{equation*}
where in the second equation we used the identity $(\frac{s}{1+s^2})(T)=T(1+T^2)^{-1}$ from the rational functional calculus in Theorem~\ref{thm_Rational_omega}. Hence we have $\ran(\widetilde{f}(T))\subseteq\dom(T)$ and using $f(T)$ in \eqref{Eq_Product_extended_pf_7} also $\ran(f(T))\subseteq\dom(T)$, with
\begin{equation*}
Tf(T)=T(1+T^2)^{-1}f_0+T\widetilde{f}(T)=\widetilde{g}(T)+T^2(1+T^2)^{-1}g_\infty=(sf)(T)=h(T).
\end{equation*}
$\circ$\;\;If $n\geq 1$ we can use \eqref{Eq_Product_extended_pf_6} to rewrite
\begin{align*}
h(s)&=\frac{1}{s^n}\Big(\frac{s^2}{1+s^2}g_\infty+\widetilde{g}(s)\Big)=\mathcal{O}\Big(\frac{1}{|s|^n}\Big),\quad\text{as }|s|\rightarrow\infty, \\
h(s)&=\frac{s}{1+s^2}f_0+s\widetilde{f}(s)=\mathcal{O}(|s|),\quad\text{as }|s|\rightarrow 0.
\end{align*}
These asymptotics prove that $h\in\mathcal{SH}_L^0(D_\theta)$, and hence by \eqref{Eq_Product_extended_pf_6} also
\begin{equation*}
s\widetilde{f}(s)=h(s)-\frac{s}{1+s^2}f_0\in\mathcal{SH}_L^0(D_\theta).
\end{equation*}
Then, by Proposition~\ref{prop_Product_omega_pf} there is $\ran(\widetilde{f}(T))\subseteq\dom(T)$, hence $\ran(f(T))\subseteq\dom(T)$ by the representation \eqref{Eq_Product_extended_pf_7}, and
\begin{equation*}
h(T)=T\widetilde{f}(T)+T(1+T^2)^{-1}f_0=Tf(T).
\end{equation*}
In both cases we now have verified that $\ran(f(T))\subseteq\dom(T)$ and the identity \eqref{Eq_Product_extended_pf_1}. Since we have also shown that $h\in\mathcal{SH}_L^\bnd(D_\theta)$ and $s^nh=s^{n+1}f\in\mathcal{SH}_L^\bnd(D_\theta)$, it follows from the induction assumption that $\ran(h(T))\subseteq\dom(T^n)$ and $T^nh(T)=(s^nh)(T)$. From this we finally conclude that $\ran(f(T))\subseteq\dom(T^{n+1})$ and
\begin{equation*}
T^{n+1}f(T)=T^nTf(T)=T^nh(T)=(s^nh)(T)=(s^{n+1}f)(T). \qedhere
\end{equation*}
\end{proof}

\begin{thm}\label{thm_Rational_extended}
Let $T\in\mathcal{K}(V)$ be bisectorial of angle $\omega\in(0,\frac{\pi}{2})$ and $p,q$ two intrinsic polynomials, such that

\begin{enumerate}
\item[i)] $\deg(q)\geq\deg(p)$,
\item[ii)] $q$ does not admit any zeros in $\overline{D_\omega}$.
\end{enumerate}

Then $q[T]$ is bijective, there exists some $\theta\in(\omega,\pi)$ such that $\frac{p}{q}\in\mathcal{N}^\bnd(D_\theta)$, and there holds
\begin{equation*}
\Big(\frac{p}{q}\Big)(T)=p[T]q[T]^{-1}.
\end{equation*}
\end{thm}

\begin{proof}
If we write the polynomials $p$ and $q$ as
\begin{equation*}
p(s)=p_0+\dots+p_ns^n\qquad\text{and}\qquad q(s)=q_0+\dots+q_ns^n,
\end{equation*}
with $n=\deg(q)\geq\deg(p)$, i.e. it is $q_n\neq 0$, but possibly $p_n=0$. Since $q$ does not admit any zeros in $\overline{D_\omega}$, it in particular does not admit a zero in $s=0$, which means that $q_0\neq 0$. Hence we can choose $e_0:=\frac{p_0}{q_0}$ and $e_\infty:=\frac{p_n}{q_n}$, and rewrite the rational function $\frac{p(s)}{q(s)}$ as
\begin{equation*}
\frac{p(s)}{q(s)}=e_\infty+\frac{e_0-e_\infty}{1+s^2}+\frac{p(s)(1+s^2)-q(s)(e_\infty s^2+e_0)}{q(s)(1+s^2)}=e_\infty+\frac{e_0-e_\infty}{1+s^2}+\frac{r(s)}{t(s)},
\end{equation*}
using the polynomials
\begin{equation*}
r(s):=p(s)(1+s^2)-q(s)(e_\infty s^2+e_0)\qquad\text{and}\qquad t(s):=q(s)(1+s^2).
\end{equation*}
The $0$-th and $(n+2)$-nd coefficient of the polynomial $r$ are given by
\begin{equation*}
r_0=p_0-q_0e_0=0\qquad\text{and}\qquad r_{n+2}=p_n-q_ne_\infty=0.
\end{equation*}
Hence these polynomials satisfy $\deg(t)=n+2$ and $\deg(r)\leq n+1$. Since $r$ also admits a zero at the origin, there is $\frac{r}{t}\in\mathcal{N}^0(D_\theta)$ and consequently $\frac{p}{q}\in\mathcal{N}^\bnd(D_\theta)$. The extended $\omega$-functional calculus is then given by
\begin{equation*}
\Big(\frac{p}{q}\Big)(T)=e_\infty+(1+T^2)^{-1}(e_0-e_\infty)+\Big(\frac{r}{t}\Big)(T).
\end{equation*}
Using now also Theorem~\ref{thm_Rational_omega} to write $(\frac{r}{t})(T)=r[T]t[T]^{-1}$, we finally get
\begin{align*}
\Big(\frac{p}{q}\Big)(T)&=e_\infty+(1+T^2)^{-1}(e_0-e_\infty)+r[T]t[T]^{-1} \\
&=e_\infty+(1+T^2)^{-1}(e_0-e_\infty)+\big(p[T](1+T^2)-q[T](e_\infty T^2+e_0)\big)(1+T^2)^{-1}q[T]^{-1} \\
&=e_\infty+(1+T^2)^{-1}(e_0-e_\infty)+p[T]q[T]^{-1}-(e_\infty T^2+e_0)(1+T^2)^{-1} \\
&=p[T]q[T]^{-1}. \qedhere
\end{align*}
\end{proof}

Next, we investigate the commutation of operators with the extended $\omega$-functional calculus. The operator inclusions are always understood in the sense \eqref{Eq_Operator_inclusion}.

\begin{thm}
Let $T\in\mathcal{K}(V)$ be bisectorial of angle $\omega\in(0,\frac{\pi}{2})$ and $A\in\mathcal{K}(V)$, such that
\begin{equation}\label{Eq_Commutation_extended_assumption}
AQ_s[T]^{-1}\supseteq Q_s[T]^{-1}A\qquad\text{and}\qquad ATQ_s[T]^{-1}\supseteq TQ_s[T]^{-1}A,
\end{equation}
for every $s\in\mathbb{R}^{n+1}\setminus\overline{D_\omega}$. Then for every $g\in\mathcal{N}^\bnd(D_\theta)$, $\theta\in(0,\frac{\pi}{2})$, there commutes
\begin{equation*}
Ag(T)\supseteq g(T)A.
\end{equation*}
\end{thm}

\begin{proof}
The extended $\omega$-functional calculus of $g\in\mathcal{N}^\bnd(D_\theta)$ is given by
\begin{equation}\label{Eq_Commutation_extended_1}
g(T)=g_\infty+(1+T^2)^{-1}(g_0-g_\infty)+\widetilde{g}(T),
\end{equation}
with values $g_0,g_\infty\in\mathbb{R}$ and $\widetilde{g}\in\mathcal{N}^0(D_\theta)$. By Theorem~\ref{thm_Commutation_omega}, we have
\begin{equation}\label{Eq_Commutation_extended_2}
A\widetilde{g}(T)\supseteq\widetilde{g}(T)A.
\end{equation}
For some arbitrary $J\in\mathbb{S}$ we have $J\in\mathbb{R}^{n+1}\setminus\overline{D_\omega}$, and hence we get by \eqref{Eq_Commutation_extended_assumption}
\begin{equation}\label{Eq_Commutation_extended_3}
A(1+T^2)^{-1}=AQ_J[T]^{-1}\supseteq Q_J[T]^{-1}A=(1+T^2)^{-1}A,
\end{equation}
Altogether, using \eqref{Eq_Commutation_extended_2} and \eqref{Eq_Commutation_extended_3} in \eqref{Eq_Commutation_extended_1}, gives $Ag(T)\supseteq g(T)A$.
\end{proof}

\begin{cor}\label{cor_Commutation_extended}
Let $T\in\mathcal{K}(V)$ be bisectorial of angle $\omega\in(0,\frac{\pi}{2})$, $g\in\mathcal{N}^\bnd(D_\theta)$, $\theta\in(\omega,\frac{\pi}{2})$.

\begin{enumerate}
\item[i)] For every $B\in\mathcal{B}(V)$ with $TB\supseteq BT$, we have
\begin{equation*}
Bg(T)=g(T)B.
\end{equation*}

\item[ii)] For every intrinsic polynomial $p$, we get
\begin{equation*}
p[T]g(T)\supseteq g(T)p[T].
\end{equation*}

\item[iii)] For every $f\in\mathcal{N}^\bnd(D_\theta)$ we obtain
\begin{equation*}
g(T)f(T)=f(T)g(T).
\end{equation*}
\end{enumerate}
\end{cor}

\begin{proof}
The proof is the same as the one of Corollary~\ref{cor_Commutation_omega_bounded}.
\end{proof}

We conclude this section with the important spectral mapping theorem, Theorem~\ref{thm_Spectral_mapping}. To do so we will need some preparatory lemmata. The first one helps us to locate resolvent set inside the double sector $\overline{D_\omega}$.

\begin{lem}\label{lem_fT_bijective}
Let $T\in\mathcal{K}(V)$ be bisectorial of angle $\omega\in(0,\frac{\pi}{2})$, and $p\in\overline{D_\omega}\setminus\{0\}$. If there exists some $g\in\mathcal{N}^\bnd(D_\theta)$, $\theta\in(\omega,\frac{\pi}{2})$, with $g(p)=0$ and $g(T)$ bijective, then $p\in\rho_S(T)$.
\end{lem}

\begin{proof}
Let us distinguish the cases $p\notin\mathbb{R}$ and $p\in\mathbb{R}\setminus\{0\}$. \medskip

$\circ$\;\;If $p\notin\mathbb{R}$, let us consider the function
\begin{equation*}
h(s):=\frac{g(s)}{Q_p(s)}=\frac{g(s)}{s^2-2p_0s+|p|^2},\qquad s\in D_\theta\setminus[p].
\end{equation*}
Since $g\in\mathcal{N}^\bnd(D_\theta)$, we can decompose $h$ into
\begin{equation*}
h(s)=\frac{1}{Q_p(s)}\Big(g_\infty+\frac{g_0-g_\infty}{1+s^2}\Big)+\frac{\widetilde{g}(s)}{Q_p(s)},\qquad s\in D_\theta\setminus[p],
\end{equation*}
for some $g_0,g_\infty\in\mathbb{R}$, $\widetilde{g}\in\mathcal{N}^0(D_\theta)$. The term $\frac{1}{Q_p(s)}$ now admits the asymptotics
\begin{equation*}
\frac{1}{Q_p(s)}=\begin{cases} \frac{1}{|p|^2}+\mathcal{O}(|s|), & \text{as }|s|\rightarrow 0^+, \\ \mathcal{O}(\frac{1}{|s|^2}), & \text{as }|s|\rightarrow\infty, \end{cases}
\end{equation*}
and consequently also
\begin{equation*}
\frac{1}{Q_p(s)}\Big(g_\infty+\frac{g_0-g_\infty}{1+s^2}\Big)=\begin{cases} \frac{g_0}{|p|^2}+\mathcal{O}(|s|), & \text{as }|s|\rightarrow 0^+, \\ \mathcal{O}(\frac{1}{|s|^2}), & \text{as }|s|\rightarrow\infty. \end{cases}
\end{equation*}
From these two asymptotics we then conclude for any $0<r<|p|<R$ the integrability
\begin{equation*}
\int_0^r\Big|h(te^{J\phi})-\frac{g_0}{|p|^2}\Big|\frac{dt}{t}<\infty\qquad\text{and}\qquad\int_R^\infty|h(te^{J\phi})|\frac{dt}{t}<\infty,
\end{equation*}
for every $J\in\mathbb{S}$, $\phi\in I_\theta$. Moreover, since $g(p)=0$ by assumption, and $g$ is intrinsic, there is also $g(s)=0$ for every $s\in[p]$. This means, that $h$ extends to an intrinsic function on all of $D_\theta$ (denoted again by $h$). From Lemma~\ref{lem_Characterization_SHbnd} it then follows that $h\in\mathcal{N}^\bnd(D_\theta)$. This in particular means that $h(T)$ is defined via the extended $\omega$-functional calculus and by Proposition~\ref{prop_Product_extended_pf} and the commutation relation in Corollary~\ref{cor_Commutation_extended} ii), we obtain $\ran(h(T))\subseteq\dom(T^2)$ and
\begin{equation*}
g(T)=(Q_ph)(T)=Q_p[T]h(T)\supseteq h(T)Q_p[T].
\end{equation*}
From this operator inclusion and the assumption that $g(T)$ is bijective, it then follows that $Q_p[T]$ is bijective with inverse $Q_p[T]^{-1}=h(T)g(T)^{-1}$. Moreover, since $g(T)$ is bounded and bijective, its inverse $g(T)^{-1}$ is closed and everywhere defined, hence $g(T)^{-1}\in\mathcal{B}(V)$ by the closed graph theorem. This means that also $Q_p[T]^{-1}\in\mathcal{B}(V)$ is bounded, i.e. $p\in\rho_S(T)$. \medskip

$\circ$\;\;If $p\in\mathbb{R}\setminus\{0\}$, let us consider the function
\begin{equation*}
h(s):=\frac{g(s)}{s-p},\qquad s\in D_\theta\setminus\{p\}.
\end{equation*}
In a similar way as above one shows that $h$ extends to an intrinsic function on all of $D_\theta$. By the product rule Proposition~\ref{prop_Product_extended_pf} and the commutation relation Corollary~\ref{cor_Commutation_extended}~ii), we obtain $\ran(h(T))\subseteq\dom(T)$ and
\begin{equation*}
g(T)=((s-p)h)(T)=(T-p)h(T)\supseteq h(T)(T-p).
\end{equation*}
From this operator inclusion and the assumption that $g(T)$ is bijective, it then follows that $T-p$ is bijective with inverse $(T-p)^{-1}=h(T)g(T)^{-1}$. Moreover, since $g(T)$ is bounded and bijective, its inverse $g(T)^{-1}$ is closed and everywhere defined, hence $g(T)^{-1}\in\mathcal{B}(V)$ by the closed graph theorem. This means that also $(T-p)^{-1}\in\mathcal{B}(V)$ is bounded. Consequently, $Q_p(T)=(T-p)^2$ is bijective and $Q_p[T]^{-1}=(T-p)^{-2}\in\mathcal{B}(V)$, i.e. $p\in\rho_S(T)$.
\end{proof}

\begin{lem}
Let $T\in\mathcal{K}(V)$ be bisectorial of angle $\omega\in(0,\frac{\pi}{2})$, such that $\sigma_S(T)\cap U_r(0)=\{0\}$, for some $r>0$. Then the operator
\begin{equation}\label{Eq_Spectral_projection}
E_0:=\frac{1}{2\pi}\int_{\partial U_\rho(0)\cap\mathbb{C}_J}S_L^{-1}(s,T)ds_J
\end{equation}
is independent of the choice $J\in\mathbb{S}$, $\rho\in(0,r)$, and for every $g\in\mathcal{N}^\bnd(D_\theta)$ there is
\begin{equation}\label{Eq_Spectral_projection_property}
g(T)E_0=g_0E_0,
\end{equation}
with $g_0\in\mathbb{R}$ from Definition~\ref{defi_SH_bnd}.
\end{lem}

\begin{proof}
The proof that $E_0$ is independent of $\rho\in(0,r)$ and $J\in\mathbb{S}$ follows the same idea as in \cite[Theorem~3.12]{ColSab2006}. Since $g\in\mathcal{N}^\bnd(D_\theta)$, there exists $g_0,g_\infty\in\mathbb{R}$, $\widetilde{g}\in\mathcal{N}^0(D_\theta)$, such that
\begin{equation*}
g(s)=g_\infty+\frac{g_0-g_\infty}{1+s^2}+\widetilde{g}(s),\qquad s\in D_\theta.
\end{equation*}
In the {\it first step} we will calculate the product of $\widetilde{g}(T)E_0$. To do so, let us consider the curves

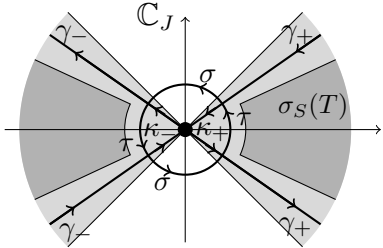
\begin{minipage}{0.44\textwidth}
\begin{center}
\begin{tikzpicture}
\fill[black!15] (0,0)--(1.56,1.56) arc (45:-45:2.2)--(0,0)--(-1.56,-1.56) arc (225:135:2.2);
\fill[black!30] (0.73,0.34)--(2,0.93) arc (25:-25:2.2)--(0.73,-0.34) arc (-25:25:0.8);
\fill[black!30] (-0.73,0.34)--(-2,0.93) arc (155:205:2.2)--(-0.73,-0.34) arc (205:155:0.8);
\draw (2,0.93)--(0.73,0.34) arc (25:-25:0.8)--(2,-0.93);
\draw (-2,0.93)--(-0.73,0.34) arc (155:205:0.8)--(-2,-0.93);
\draw (1.56,1.56)--(-1.56,-1.56);
\draw (1.56,-1.56)--(-1.56,1.56);
\draw[thick] (1.8,1.26)--(-1.8,-1.26);
\draw[thick] (1.8,-1.26)--(-1.8,1.26);
\draw[thick,->] (1.8,1.26)--(1.47,1.03);
\draw[thick,->] (0,0)--(1.47,-1.03);
\draw[thick,->] (-1.8,-1.26)--(-1.47,-1.03);
\draw[thick,->] (0,0)--(-1.47,1.03);
\draw (1.5,1) node[anchor=south] {$\gamma_+$};
\draw (1.47,-1.03) node[anchor=north] {$\gamma_+$};
\draw (-1.45,-1.11) node[anchor=north] {$\gamma_-$};
\draw (-1.5,1) node[anchor=south] {$\gamma_-$};
\draw[thick,->] (0.41,0.29)--(0.27,0.2);
\draw[thick,->] (0.27,-0.2)--(0.41,-0.29);
\draw[thick,->] (-0.27,0.2)--(-0.41,0.29);
\draw[thick,->] (-0.41,-0.29)--(-0.27,-0.2);
\draw (0,-0.03) node[anchor=west] {$\kappa_+$};
\draw (0.05,-0.04) node[anchor=east] {$\kappa_-$};
\draw[->] (-2.4,0)--(2.6,0);
\draw[->] (0,-1.5)--(0,1.5) node[anchor=east] {\large{$\mathbb{C}_J$}};
\fill[black] (0,0) circle (0.1cm);
\draw[thick] (0,0) circle (0.6cm);
\draw[thick,->] (0.21,0.56)--(0.18,0.57);
\draw[thick,->] (-0.21,-0.56)--(-0.18,-0.57);
\draw (0.3,0.5) node[anchor=south] {$\sigma$};
\draw (-0.3,-0.5) node[anchor=north] {$\sigma$};
\draw[thick,->] (0.555,0.22)--(0.54,0.25);
\draw[thick,->] (-0.555,-0.22)--(-0.54,-0.25);
\draw (0.54,0.15) node[anchor=west] {$\tau$};
\draw (-0.54,-0.2) node[anchor=east] {$\tau$};
\draw (1.7,0) node[anchor=south] {\small{$\sigma_S(T)$}};
\end{tikzpicture}
\end{center}
\end{minipage}
\begin{minipage}{0.54\textwidth}
\begin{align*}
\gamma_\pm(t)&:=\pm|t|e^{-\sgn(t)J\varphi},\qquad t\in\mathbb{R}\setminus[-\rho,\rho], \\
\kappa_\pm(t)&:=\pm|t|e^{-\sgn(t)J\varphi},\qquad t\in(-\rho,\rho), \\
\sigma(\phi)&:=\rho e^{J\phi},\hspace{2.2cm}\phi\in\Big(-\frac{\pi}{2},\frac{3\pi}{2}\Big)\setminus\overline{I_\varphi}, \\
\tau(\phi)&:=\rho e^{J\phi},\hspace{2.2cm}\phi\in I_\varphi, \\
\gamma&:=\gamma_+\oplus\gamma_-, \\
\kappa&:=\kappa_+\oplus\kappa_-.
\end{align*}
\end{minipage}

\medskip With the representation \eqref{Eq_Representation_intrinsic} and the resolvent identity \eqref{Eq_Product_omega_2}, we get
\begin{align}
\widetilde{g}(T)E_0&=\frac{1}{4\pi^2}\int_{\gamma\oplus\kappa}\widetilde{g}(s)ds_JS_R^{-1}(s,T)\int_{\sigma\oplus\tau}S_L^{-1}(p,T)dp_J \notag \\
&=\frac{1}{4\pi^2}\int_{\gamma\oplus\kappa}\widetilde{g}(s)ds_J\int_{\sigma\oplus\tau}\Big(S_R^{-1}(s,T)p-S_L^{-1}(p,T)p \notag \\
&\hspace{4.8cm}-\overline{s}S_R^{-1}(s,T)+\overline{s}S_L^{-1}(p,T)\Big)Q_s(p)^{-1}dp_J. \label{Eq_Spectral_projection_3}
\end{align}
Since both values $p,s\in\mathbb{C}_J$ are in the same complex plane, we can use the classical residue theorem to derive the explicit value of the integral
\begin{align}
\frac{1}{2\pi}\int_{\sigma\oplus\tau}pQ_s(p)^{-1}dp_J&=\frac{1}{2\pi}\int_{\sigma\oplus\tau}\frac{p}{(p-s)(p-\overline{s})}dp_J \notag \\
&=\begin{cases} \frac{s}{s-\overline{s}}+\frac{\overline{s}}{\overline{s}-s}, & \text{if }s\in\ran(\kappa), \\ 0, & \text{if }s\in\ran(\gamma), \end{cases}=\begin{cases} 1, & \text{if }s\in\ran(\kappa), \\ 0, & \text{if }s\in\ran(\gamma). \end{cases} \label{Eq_Spectral_projection_1}
\end{align}
In the same way, we also get
\begin{align}
\frac{1}{2\pi}\int_{\sigma\oplus\tau}Q_s(p)^{-1}dp_J&=\frac{1}{2\pi}\int_{\sigma\oplus\tau}\frac{1}{(p-s)(p-\overline{s})}dp_J \notag \\
&=\begin{cases} \frac{1}{s-\overline{s}}+\frac{1}{\overline{s}-s}, & \text{if }s\in\ran(\kappa), \\ 0, & \text{if }s\in\ran(\gamma), \end{cases}=\begin{cases} 0, & \text{if }s\in\ran(\kappa), \\ 0, & \text{if }s\in\ran(\gamma). \end{cases} \label{Eq_Spectral_projection_2}
\end{align}
Using now the integrals \eqref{Eq_Spectral_projection_1} and \eqref{Eq_Spectral_projection_2} in \eqref{Eq_Spectral_projection_3}, gives
\begin{align}
\widetilde{g}(T)E_0&=\frac{1}{2\pi}\int_\kappa\widetilde{g}(s)ds_JS_R^{-1}(s,T) \notag \\
&+\frac{1}{4\pi^2}\int_{\gamma\oplus\kappa}\widetilde{g}(s)ds_J\int_{\sigma\oplus\tau}\big(\overline{s}S_L^{-1}(p,T)-S_L^{-1}(p,T)p\big)Q_s(p)^{-1}dp_J. \label{Eq_Spectral_projection_4}
\end{align}
Next, we note that the identity \eqref{Eq_Product_identity} with $B=S_L^{-1}(p,T)$ in our setting becomes
\begin{equation*}
\frac{1}{2\pi}\int_{\gamma\oplus\kappa}\widetilde{g}(s)ds_J\big(\overline{s}S_L^{-1}(p,T)-S_L^{-1}(p,T)p\big)Q_s(p)^{-1}=\begin{cases} S_L^{-1}(p,T)\widetilde{g}(p), & \text{if }p\in\ran(\tau), \\ 0, & \text{if }p\in\ran(\sigma). \end{cases}
\end{equation*}
Using this in \eqref{Eq_Spectral_projection_4}, simplifies the right hand side to
\begin{equation}\label{Eq_Spectral_projection_5}
\widetilde{g}(T)E_0=\frac{1}{2\pi}\int_\kappa\widetilde{g}(s)ds_JS_R^{-1}(s,T)+\frac{1}{2\pi}\int_\tau S_L^{-1}(p,T)\widetilde{g}(p)dp_J.
\end{equation}
Since $\rho\in(0,R)$ is arbitrary and the left hand side is independent of $\rho$, we can take the limit $\rho\rightarrow 0^+$ of the right hand side. However, since $\Vert S_L^{-1}(p,T)\Vert\leq\frac{C_\varphi}{|p|}$ by Definition~\ref{defi_Bisectorial_operators}, the second integral vanishes by Lemma~\ref{lem_Integral_vanish}~i), which obviously also holds true for operator valued function $S_L^{-1}(p,T)$. Moreover, with the estimate \eqref{Eq_SR_estimate}, also the first integral in \eqref{Eq_Spectral_projection_5} vanishes, namely
\begin{align*}
\bigg|\int_\kappa\widetilde{g}(s)ds_JS_R^{-1}(s,T))\bigg|&\leq\sum\limits_{\phi\in\{-\varphi,\varphi,\pi-\varphi,\pi+\varphi\}}\bigg|\int_0^\rho\widetilde{g}(te^{J\phi})\frac{e^{J\phi}}{J}S_R^{-1}(te^{J\phi},T)dt\bigg| \\
&\leq 2C_\varphi\sum\limits_{\phi\in\{-\varphi,\varphi,\pi-\varphi,\pi+\varphi\}}\int_0^\rho|\widetilde{g}(te^{J\phi})|\frac{dt}{t}\overset{\rho\rightarrow 0^+}{\longrightarrow}0.
\end{align*}
Altogether, we found that the right hand side of \eqref{Eq_Spectral_projection_5} vanishes in the limit $\rho\rightarrow 0^+$, i.e.
\begin{equation}\label{Eq_Spectral_projection_6}
\widetilde{g}(T)E_0=0.
\end{equation}
In the {\it second step} we consider the product $(1+T^2)^{-1}E_0$. Starting with $T^2E_0$, we use the resolvent identity \eqref{Eq_Product_omega_pf_1} to rewrite
\begin{equation*}
\int_{\sigma\oplus\tau}S_L^{-1}(p,T)p^2dp_J=\int_{\sigma\oplus\tau}\big(TS_L^{-1}(p,T)+1\big)pdp_J=\int_{\sigma\oplus\tau}TS_L^{-1}(p,T)pdp_J.
\end{equation*}
By Hille's theorem this means that $\ran\big(\int_{\sigma\oplus\tau}S_L^{-1}(p,T)pdp_J\big)\subseteq\dom(T)$, and
\begin{equation*}
\int_{\sigma\oplus\tau}S_L^{-1}(p,T)p^2dp_J=T\int_{\sigma\oplus\tau}S_L^{-1}(p,T)pdp_J.
\end{equation*}
The same computation applied once more, shows $\ran\big(\int_{\sigma\oplus\tau}S_L^{-1}(p,T)dp_J\big)\subseteq\dom(T^2)$, and
\begin{equation*}
\int_{\sigma\oplus\tau}S_L^{-1}(p,T)p^2dp_J=T^2\int_{\sigma\oplus\tau}S_L^{-1}(p,T)dp_J=T^2E_0.
\end{equation*}
However, since $0<\rho<r$ is arbitrary, and only the left hand side depends on $\rho$, we can take the limit $\rho\rightarrow 0^+$ and obtain
\begin{equation*}
\Vert T^2E_0\Vert=\lim\limits_{\rho\rightarrow 0^+}\bigg\Vert\int_0^{2\pi}S_L^{-1}(\rho e^{J\phi},T)(\rho e^{J\phi})^2\rho e^{J\phi}d\phi\bigg\Vert\leq\lim\limits_{\rho\rightarrow 0^+}\int_0^{2\pi}\frac{C_\varphi}{\rho}\rho^3d\phi=0.
\end{equation*}
This proves that $T^2E_0=0$, and since $1+T^2$ is invertible, also
\begin{equation}\label{Eq_Spectral_projection_7}
(1+T^2)^{-1}E_0=E_0.
\end{equation}
From the identities \eqref{Eq_Spectral_projection_6} and \eqref{Eq_Spectral_projection_7} we now conclude
\begin{equation*}
g(T)E_0=g_\infty E_0+(g_0-g_\infty)(1+T^2)^{-1}E_0+\widetilde{g}(T)E_0=g_\infty E_0+(g_0-g_\infty)E_0=g_0E_0. \qedhere
\end{equation*}
\end{proof}

\begin{thm}[Spectral mapping theorem]\label{thm_Spectral_mapping}
Let $T\in\mathcal{K}(V)$ be bisectorial of angle $\omega\in(0,\frac{\pi}{2})$, and $g\in\mathcal{N}^\bnd(D_\theta)$, $\theta\in(\omega,\frac{\pi}{2})$. with the values $g_0,g_\infty\in\mathbb{R}$ from Definition~\ref{defi_SH_bnd}. If we extend the function $g$ with the value $g(0):=g_0$, then we have
\begin{equation}\label{Eq_Spectral_mapping}
\sigma_S(g(T))\supseteq g(\sigma_S(T))\qquad\text{and}\qquad\sigma_S(g(T))\subseteq g(\sigma_S(T))\cup\{g_\infty\}.
\end{equation}
\end{thm}

\begin{proof}
For the first inclusion let $c\in g(\sigma_S(T))$, i.e., we can write $c=g(p)$ for some $p\in\sigma_S(T)$. Let us distinguish three cases. \medskip

Case 1: If $p\neq 0$, then $p\in\sigma_S(T)\setminus\{0\}\subseteq\overline{D_\omega}\setminus\{0\}\subseteq D_\theta$ and consequently by Corollary~\ref{cor_Composition_extended} there is $Q_c\circ g\in\mathcal{N}^\bnd(D_\theta)$ and
\begin{equation*}
(Q_c\circ g)(T)=Q_c[g(T)].
\end{equation*}
Let us assume that $c\in\rho_S(g(T))$. Then $(Q_c\circ g)(T)^{-1}\in\mathcal{B}(V)$ and since also
\begin{equation*}
(Q_c\circ g)(p)=g(p)^2-2c_0g(p)+|c|^2=0,
\end{equation*}
because of $c=g(p)$, it follows from Lemma~\ref{lem_fT_bijective}, that $p\in\rho_S(T)$. This is a contradiction to the choice $p\in\sigma_S(T)$, and we have proven $c\in\sigma_S(g(T))$. \medskip

Case 2: If $p=0$ and there exists a sequence $(p_n)_{n\in\mathbb{N}}\in\sigma_S(T)$ with $\lim_{n\rightarrow\infty}p_n=0$. Then we also have $\lim_{n\rightarrow\infty}g(p_n)=g_0$ and hence there is
\begin{equation*}
c=g(0)=g_0\in\overline{g(\sigma_S(T)\setminus\{0\})}\subseteq\overline{\sigma_S(g(T))}=\sigma_S(g(T)),
\end{equation*}
where we used the fact that the $S$-spectrum is closed and that $g(\sigma_S(T)\setminus\{0\})\subseteq\sigma_S(g(T))$,
which we have already proven above. \medskip

Case 3: If $p=0$ and $U_r(0)\cap\sigma_S(T)=\{0\}$ for some $r>0$, let us assume that $c\in\rho_S(g(T))$, i.e. $Q_c[g(T)]$ is boundedly invertible. Using the spectral projection $E_0$ from \eqref{Eq_Spectral_projection}, with the property $g(T)E_0=g_0E_0$ from \eqref{Eq_Spectral_projection_property}, we then get
\begin{equation*}
Q_c[g(T)]E_0=(g(T)^2-2c_0g(T)+|c|^2)E_0=(g_0^2-2c_0g_0+|c|^2)E_0=0,
\end{equation*}
where the last bracket vanishes because of $c=g(0)=g_0$. However, since $Q_c[g(T)]$ is bijective, this implies that $E_0=0$ is the zero operator, i.e., $V_0:=\ran(E_0)=\{0\}$ is the trivial vector space. Moreover, by \cite[Theorem 3.7.8 (iii)]{FJBOOK}, there is $\sigma_S(T|_{V_0})=\{0\}$. However, since $V_0=\{0\}$, the restriction $T|_{V_0}$ is the trivial operator which always has empty spectrum $\sigma_S(T|_{V_0})=\emptyset$. Hence, our assumption $c\in\rho_S(g(T))$ was wrong and we conclude $c\in\sigma_S(g(T))$. \medskip

For the second inclusion in \eqref{Eq_Spectral_mapping}, let $c\in\sigma_S(g(T))$, and distinguish four cases. \medskip

Case 1: If $c=g_\infty$, then clearly $c\in g(\sigma_S(T))\cup\{g_\infty\}$. \medskip

Case 2: If $c=g_0\neq g_\infty$ and $0\in\sigma_S(T)$, we clearly have $c=g_0=g(0)\in g(\sigma_S(T))$. \medskip

Case 3: If $c=g_0\neq g_\infty$ and $0\notin\sigma_S(T)$, we are in the situation of Remark~\ref{rem_Omega_0rho} and Remark~\ref{rem_Extended_0rho}, which allow us to define the $\omega$- and the extended $\omega$-functional calculus for a larger class of functions. Let us consider
\begin{equation}\label{Eq_Spectral_mapping_1}
h(s):=(Q_c\circ g)(s)=g(s)^2-2c_0g(s)+|c|^2,\qquad s\in D_\theta.
\end{equation}
Then $h\in\mathcal{N}^\bnd(D_\theta)$ by Theorem~\ref{thm_Product_extended}, with $h_0=(g_0-c)^2=0$ and $h_\infty=(g_\infty-c)^2\neq 0$. An interior zero of infinite order, or zeros which accumulate at an interior point, would cause $h$ to vanish identically on this connected component, but this is not possible since $\lim_{|s|\rightarrow\infty}h(s)=(g_\infty-c)^2\neq 0$. Hence, there exist finitely many points $p_1,\dots,p_n\in\mathbb{R}\setminus\{0\}$ and $q_1,\dots,q_m\in D_\theta\setminus\mathbb{R}$, such that the zeros of $h$ are given by $p_1,\dots,p_n,[q_1],\dots,[q_m]$. Let us  call the multiplicities of those zeros $k_1,\dots,k_n$ and $l_1,\dots,l_m$, respectively. Let us then consider the function
\begin{equation}\label{Eq_Spectral_mapping_2}
r(s)=\prod\limits_{i=1}^n\Big(\frac{s-p_i}{1+s^2}\Big)^{k_i}\prod\limits_{j=1}^m\Big(\frac{Q_{q_j}(s)}{1+s^2}\Big)^{l_j},\qquad s\in D_\theta.
\end{equation}
It is then clear that every factor $\frac{s-p_i}{1+s^2},\frac{Q_{q_j}(s)}{1+s^2}\in\mathcal{N}^\bnd(D_\theta)$ and hence also $r\in\mathcal{N}^\bnd(D_\theta)$. Then, by Lemma~\ref{lem_Characterization_SHbnd}, the shifted functions
\begin{equation*}
\widehat{h}(s):=h(s)-h_\infty\qquad\text{and}\qquad\widehat{r}(s):=r(s)-r_\infty,\qquad s\in D_\theta,
\end{equation*}
satisfy the integrability condition \eqref{Eq_Integrability_at_infinity}. For every $s\in D_\theta\setminus\big(\{p_1\},\dots,\{p_n\}\cup[q_1]\cup\dots\cup[q_m]\big)$, we can then define
\begin{equation*}
k(s):=\frac{r(s)}{h(s)}=\big(r_\infty+\widehat{r}(s)\big)\Big(\frac{1}{h_\infty}-\frac{\widehat{h}(s)}{h_\infty h(s)}\Big)=\frac{r_\infty}{h_\infty}+\underbrace{\frac{\widehat{r}(s)}{h_\infty}-\frac{r_\infty\widehat{h}(s)}{h_\infty h(s)}-\frac{\widehat{r}(s)\widehat{h}(s)}{h_\infty h(s)}}_{=:\widehat{k}(s)}.
\end{equation*}
This function $\widehat{k}$ then also satisfies the condition \eqref{Eq_Integrability_at_infinity}. Moreover, since $r$ and $h$ have the same zeros with the same multiplicity, $k$ extends to an intrinsic function on $D_\theta$. For this extension we can then define the functional calculus $k(T)$ as in \eqref{Eq_Extended_0rho}, and get
\begin{equation}\label{Eq_Spectral_mapping_3}
r(T)=(hk)(T)=h(T)k(T)=(Q_c\circ g)(T)k(T)=Q_c[g(T)]k(T),
\end{equation}
where in the last equation we used Corollary~\ref{cor_Composition_extended}. \medskip

Let us now assume that $c\notin g(\sigma_S(T))$. Since $g$ is intrinsic and $\sigma_S(T)$ axially symmetric, this means $g(s)\notin[c]$ for every $s\in\sigma_S(T)$. However, since $p_1,\dots,p_n,[q_1],\dots,[q_m]$ are zeros of $Q_c\circ g$, there is $g(p_1),\dots,g(p_n),g(q_1),\dots,g(q_m)\in[c]$, from which we conclude that $p_1,\dots,p_n,q_1,\dots,q_m\notin\sigma_S(T)$. Hence the operator
\begin{equation*}
r(T)=\prod\limits_{i=1}^n\big((T-p_i)(1+T^2)^{-1}\big)^{k_i}\prod\limits_{j=1}^m\big(Q_{q_j}[T](1+T^2)^{-1}\big)^{l_j},
\end{equation*}
is bijective and by \eqref{Eq_Spectral_mapping_3} also $Q_c[g(T)]k(T)$. Since moreover $Q_c[g(T)]k(T)=k(T)Q_c[g(T)]$ commute, there follows the bijectivity of $Q_c[g(T)]$. Since $Q_c[g(T)]$ is bounded, the inverse $Q_c[g(T)]^{-1}$ is bounded as well and hence $c\in\rho_S(g(T))$. This is a contradiction to $c\in\sigma_S(g(T))$, which means that the assumption $c\notin g(\sigma_S(T))$ was wrong and we have proven $c\in g(\sigma_S(T))$. \medskip

Case 4: If $c\notin\{g_0,g_\infty\}$, we again consider the function $h$ from \eqref{Eq_Spectral_mapping_1} and $r$ from \eqref{Eq_Spectral_mapping_2}. Contrary to above, here the values $h_0=(g_0-c)^2\neq 0$ and $h_\infty=(g_\infty-c)^2\neq 0$ are both nonzero. It is now straight forward to show that in the decomposition of Definition~\ref{defi_SH_bnd}, we can write for every $s\in D_\theta\setminus\big(\{p_1\},\dots,\{p_n\}\cup[q_1]\cup\dots\cup[q_m]\big)$ the function $k(s):=\frac{r(s)}{h(s)}$ as
\begin{align*}
k(s)-\frac{r_0}{h_0}&=\frac{\frac{s^2}{1+s^2}(r_\infty h_0-r_0h_\infty)+\widetilde{r}(s)h_0-r_0\widetilde{h}(s)}{h_0h(s)},\qquad\text{and} \\
k(s)-\frac{r_\infty}{h_\infty}&=\frac{\frac{1}{1+s^2}(r_0h_\infty-r_\infty h_0)+\widetilde{r}(s)h_\infty-r_\infty\widetilde{h}(s)}{h_\infty h(s)}.
\end{align*}
This form shows that there exists some $0<r<R$, such that the integrability conditions \eqref{Eq_Characterization_SHbnd} are satisfied. As above, the function $k$ extends to an intrinsic function on $D_\theta$ and from Lemma~\ref{lem_Characterization_SHbnd} we then get $k\in\mathcal{N}^\bnd(D_\theta)$. This means that $k(T)$ is defined via the extended $\omega$-functional calculus and there is
\begin{equation*}
r(T)=(hk)(T)=h(T)k(T)=(Q_c\circ g)(T)k(T)=Q_c[g(T)]k(T),
\end{equation*}
where in the last equation we used Corollary~\ref{cor_Composition_extended}. Since this is the same formula as in \eqref{Eq_Spectral_mapping_3}, we also conclude by the same argument that $c\in g(\sigma_S(T))$.
\end{proof}

\section{The $H^\infty$-functional calculus}\label{sec_Hinfty}

Taking advantage of the extended $\omega$-functional calculus from Section~\ref{sec_Extended}, we will now define the $H^\infty$-functional calculus for the so called class of regularizable functions. In particular, functions with polynomial growth at infinity will be covered by this calculus, and under the additional assumption that $T$ is injective, also a polynomial growth at zero is allowed.

\begin{defi}\label{defi_SH_reg}
For every $\theta\in(0,\frac{\pi}{2})$ let $D_\theta$ be the double sector from \eqref{Eq_Domega}. Then we define the function spaces of

\begin{enumerate}
\item[i)] $\mathcal{SH}_L^\reg(D_\theta):=\Big\{f\in\mathcal{SH}_L(D_\theta)\;\Big|\;\exists e\in\mathcal{N}^\bnd(D_\theta): ef\in\mathcal{SH}_L^\bnd(D_\theta),\,e(T)\text{ injective}\Big\}$
\item[ii)] $\mathcal{N}^\reg(D_\theta):=\Big\{g\in\mathcal{N}(D_\theta)\;\Big|\;\exists e\in\mathcal{N}^\bnd(D_\theta): eg\in\mathcal{N}^\bnd(D_\theta),\,e(T)\text{ injective}\Big\}$
\end{enumerate}
For every function $f\in\mathcal{SH}_L^\reg(D_\theta)$, the function $e$ is called a {\it regularizer} of $f$. Note that the operator $e(T)$ is well defined via the extended $\omega$-functional calculus \eqref{Eq_Extended}.
\end{defi}

\begin{prop}\label{prop_Polynomial_functions}
Let $T\in\mathcal{K}(V)$ be bisectorial of angle $\omega\in(0,\frac{\pi}{2})$. Then for every $\theta\in(\omega,\frac{\pi}{2})$, we have
\begin{equation*}
\bigg\{f\in\mathcal{SH}_L(D_\theta)\;\bigg|\;\begin{array}{l} \exists r>0,f_0\in\mathbb{R}_n: \int_0^r|f(te^{J\phi})-f_0|\frac{dt}{t}<\infty,\;J\in\mathbb{S},\phi\in I_\theta, \\ \exists k\in\mathbb{N}_0,\,C\geq 0: |f(s)|\leq C(1+|s|^k),\,s\in D_\theta. \end{array}\bigg\}\subseteq\mathcal{SH}_L^\reg(D_\theta).
\end{equation*}
If $T$ is injective, we even have
\begin{equation*}
\bigg\{f\in\mathcal{SH}_L(D_\theta)\;\bigg|\;\exists k\in\mathbb{N}_0,\,C\geq 0: |f(s)|\leq C\Big(\frac{1}{|s|^k}+|s|^k\Big),\,s\in D_\theta\bigg\}\subseteq\mathcal{SH}_L^\reg(D_\theta).
\end{equation*}
\end{prop}

\begin{proof}
Let $f\in\mathcal{SH}_L(D_\theta)$, such that there exists $r>0$, $f_0\in\mathbb{R}_n$, $k\in\mathbb{N}_0$, $C\geq 0$ with
\begin{equation*}
|f(s)|\leq C(1+|s|^k)\qquad\text{and}\qquad\int_0^r|f(te^{J\phi})-f_0|\frac{dt}{t}<\infty,\qquad J\in\mathbb{S},\,\phi\in I_\theta.
\end{equation*}
Then let us choose the regularizer
\begin{equation*}
e(s):=\frac{1}{(1+s^2)^{k+1}},
\end{equation*}
which is clearly an element $e\in\mathcal{N}^\bnd(D_\theta)$. The corresponding functional calculus is according to Theorem~\ref{thm_Rational_extended} given by $e(T)=(1+T^2)^{-k-1}$. Since $1+T^2$ is bijective due to $J\in\rho_S(T)$, for every $J\in\mathbb{S}$, the operator $e(T)$ is injective. Moreover, we can write the product $ef$ as
\begin{align}
e(s)f(s)&=\frac{1}{1+s^2}f_0+\frac{1}{(1+s^2)^{k+1}}\big(f(s)-(1+s^2)^kf_0\big) \notag \\
&=\frac{1}{1+s^2}f_0+\underbrace{\frac{1}{(1+s^2)^{k+1}}(f(s)-f_0)}_{=:\widetilde{g_0}(s)}-\sum\limits_{l=1}^k{k\choose l}\underbrace{\frac{s^{2l}}{(1+s^2)^{k+1}}}_{=:\widetilde{g}_l(s)}f_0. \label{Eq_Polynomial_functions_1}
\end{align}
In order to show $ef\in\mathcal{N}^\bnd(D_\theta)$, we have to verify $\widetilde{g}_l\in\mathcal{N}^0(D_\theta)$ for every $l\in\{0,\dots,k\}$. The function $\widetilde{g}_0$ can be estimated in two ways, namely
\begin{equation*}
|\widetilde{g}_0(s)|=\frac{|f(s)-f_0|}{|1+s^2|^{k+1}}\leq\Big(\frac{2}{1+\cos(2\theta)}\Big)^{\frac{k+1}{2}}\frac{|f(s)-f_0|}{(1+|s|^2)^{k+1}}\leq\Big(\frac{2}{1+\cos(2\theta)}\Big)^{\frac{k+1}{2}}\begin{cases} |f(s)-f_0|, \\ \frac{C(1+|s|^k)+|f_0|}{(1+|s|^2)^{k+1}}, \end{cases}
\end{equation*}
where we used that $\sup_{s\in D_\theta}\frac{1+|s|^2}{|1+s^2|}=\big(\frac{2}{1+\cos(2\theta)}\big)^{1/2}$, which can easily be verified. This estimate shows that the integral
\begin{align*}
\int_0^\infty&|\widetilde{g}_0(te^{J\phi})|\frac{dt}{t}=\int_0^r|\widetilde{g}_0(te^{J\phi})|\frac{dt}{t}+\int_r^\infty|\widetilde{g}_0(te^{J\phi})|\frac{dt}{t} \\
&\leq\Big(\frac{2}{1+\cos(2\theta)}\Big)^{\frac{k+1}{2}}\bigg(\int_0^r|f(te^{J\phi})-f_0|\frac{dt}{t}+\int_r^\infty\frac{C(1+|t|^k)+|f_0|}{(1+|t|^2)^{k+1}}\frac{dt}{t}\bigg)<\infty,
\end{align*}
is finite and hence $\widetilde{g}_0\in\mathcal{N}^0(D_\theta)$. Next, the functions $\widetilde{g}_l$, $l\in\{1,\dots,k\}$, can be estimated by
\begin{equation*}
|\widetilde{g}_l(s)|=\frac{|s|^{2l}}{|1+s^2|^{k+1}}\leq\Big(\frac{2}{1+\cos(2\theta)}\Big)^{\frac{k+1}{2}}\frac{|s|^{2l}}{(1+|s|^2)^{k+1}}.
\end{equation*}
This estimate shows that $\widetilde{g}_l\in\mathcal{N}^0(D_\theta)$. Altogether we have proven that the product $ef$ in the representation \eqref{Eq_Polynomial_functions_1} is an element $ef\in\mathcal{SH}_L^\bnd(D_\theta)$. This shows that $e$ is a regularizer of $f$ according to Definition~\ref{defi_SH_reg} and hence $f\in\mathcal{SH}_L^\reg(D_\theta)$. \medskip

Let us now additionally assume that $T$ is injective. Then, for every $f\in\mathcal{SH}_L(D_\theta)$, which satisfies the estimate
\begin{equation*}
|f(s)|\leq C\Big(\frac{1}{|s|^k}+|s|^k\Big),
\end{equation*}
we choose the regularizer
\begin{equation*}
e(s):=\frac{s^{k+1}}{(1+s^2)^{k+1}}.
\end{equation*}
Then clearly $e\in\mathcal{N}^0(D_\theta)\subseteq\mathcal{N}^\bnd(D_\theta)$ and by Theorem~\ref{thm_Rational_extended}, the corresponding functional calculus is given by $e(T)=T^{k+1}(1+T^2)^{-k-1}$. Since $1+T^2$ is bijective due to $J\in\rho_S(T)$ for every $J\in\mathbb{S}$, and since we also assumed that $T$ is injective, also $e(T)$ is injective. Finally, we can estimate the product $ef$ by
\begin{equation*}
|e(s)f(s)|\leq C\frac{|s|^{k+1}}{|1+s^2|^{k+1}}\Big(\frac{1}{|s|^k}+|s|^k\Big)\leq C\Big(\frac{2}{1+\cos(2\theta)}\Big)^{\frac{k+1}{2}}\frac{|s|+|s|^{2k+1}}{(1+|s|^2)^{k+1}}.
\end{equation*}
This estimate shows that $ef\in\mathcal{N}^0(D_\theta)\subseteq\mathcal{N}^\bnd(D_\theta)$. Hence $e$ is a regularizer of $f$ according to Definition~\ref{defi_SH_reg} and we conclude that $f\in\mathcal{N}^\reg(D_\theta)$.
\end{proof}

For those regularizable functions in Definition~\ref{defi_SH_reg} we now define the $H^\infty$-functional calculus.

\begin{defi}[$H^\infty$-functional calculus]
Let $T\in\mathcal{K}(V)$ be bisectorial of angle $\omega\in(0,\frac{\pi}{2})$. Then for every $f\in\mathcal{SH}_L^\reg(D_\theta)$, $\theta\in(\omega,\frac{\pi}{2})$, we define the {\it $H^\infty$-functional calculus}
\begin{equation}\label{Eq_Hinfty}
f(T):=e(T)^{-1}(ef)(T),
\end{equation}
where $e$ is a regularizer of $f$ according to Definition~\ref{defi_SH_reg}.
\end{defi}

Note, that definition \eqref{Eq_Hinfty} is independent of the choice of the regularizer $e$. This is because from the product rule \eqref{Eq_Product_extended} we obtain for any two regularizers $e_1,e_2$ the equality
\begin{align*}
e_1(T)^{-1}(e_1f)(T)&=e_1(T)^{-1}e_2(T)^{-1}e_2(T)(e_1f)(T)=(e_2e_1)(T)^{-1}(e_2e_1f)(T) \\
&=(e_1e_2)(T)^{-1}(e_1e_2f)(T)=e_2(T)^{-1}e_1(T)^{-1}e_1(T)(e_2f)(T) \\
&=e_2(T)^{-1}(e_2f)(T).
\end{align*}

\begin{rem}\label{rem_Hinfty_right}
Note that the natural choice of the $H^\infty$-functional calculus \eqref{Eq_Hinfty} for right slice  hyperholomorphic functions, would be $f(T):=(fe)(T)e(T)^{-1}$. However, this is not a good definition since $f(T)$ is defined on $\ran(e(T))$ and hence depends on the choice of the regularizer. It is an open problem how the $H^\infty$-functional calculus for right slice hyperholomorphic functions has to be defined.
\end{rem}

Let us now collect some basic properties of the $H^\infty$-functional calculus \eqref{Eq_Hinfty}.

\begin{lem}
Let $T\in\mathcal{K}(V)$ be bisectorial of angle $\omega\in(0,\frac{\pi}{2})$, and $f,g\in\mathcal{SH}_L^\reg(D_\theta)$, $\theta\in(\omega,\frac{\pi}{2})$. Then we have:

\begin{enumerate}
\item[i)] The operator $f(T)\in\mathcal{K}(V)$ is right-linear and closed, with domain
\begin{equation}\label{Eq_Domain_Hinfty}
\dom(f(T))=\big\{v\in V\;|\;(ef)(T)v\in\ran(e(T))\big\},
\end{equation}
where $e\in\mathcal{N}^\bnd(D_\theta)$ is an arbitrary regularizer of $f$.

\item[ii)] For every $a\in\mathbb{R}_n$ there holds the linearity properties
\begin{equation}\label{Eq_Linearity_Hinfty}
(f+g)(T)\supseteq f(T)+g(T)\qquad\text{and}\qquad(fa)(T)=f(T)a.
\end{equation}
\end{enumerate}
\end{lem}

\begin{proof}
i)\;\;The operators $e(T)$ and $(ef)(T)$ in the definition \eqref{Eq_Hinfty} are bounded, and $e(T)^{-1}$ is closed as the inverse of a bounded operator. Hence $f(T)=e(T)^{-1}(ef)(T)$ is closed as the product of a closed and a bounded operator. The explicit form of the domain \eqref{Eq_Domain_Hinfty} is exactly the one of the product of unbounded operators. \medskip

ii)\;\;If $e_1$ is a regularizer of $f$ and $e_2$ a regularizer of $g$, then it is clear that the product $e_1e_2$ is a regularizer for both functions $f$ and $g$ as well as for their sum $f+g$. Hence we can write the $H^\infty$-functional calculus of the sum as
\begin{align*}
(f+g)(T)&=(e_1e_2)^{-1}(e_1e_2(f+g))(T) \\
&=(e_1e_2)^{-1}\big((e_1e_2f)(T)+(e_1e_2g)(T)\big) \\
&\supseteq(e_1e_2)^{-1}(e_1e_2f)(T)+(e_1e_2)^{-1}(e_1e_2g)(T)=f(T)+g(T),
\end{align*}
where from the second to the third line we used the inclusion $A(B+C)\supseteq AB+AC$ of unbounded operators. \medskip

For the second equation in \eqref{Eq_Linearity_Hinfty} we note that the regularizer $e_1$ of $f$ is also a regularizer of $fa$. Hence we can write the $H^\infty$-functional calculus \eqref{Eq_Hinfty} for $fa$ as
\begin{equation*}
(fa)(T)=e_1(T)^{-1}(e_1fa)(T)=e_1(T)(e_1f)(T)a=f(T)a. \qedhere
\end{equation*}
\end{proof}

\begin{prop}
Let $T\in\mathcal{K}(V)$ be bisectorial of angle $\omega\in(0,\frac{\pi}{2})$. If $T$ is not injective, then for every $f\in\mathcal{SH}_L^\reg(D_\theta)$ there exists the limit
\begin{equation*}
f_0:=\lim\limits_{|s|\rightarrow 0^+}f(s).
\end{equation*}
Moreover, if $f\in\mathcal{N}^\reg(D_\theta)$ is intrinsic, we also have
\begin{equation*}
\ker(T)\subseteq\ker(f(T)-f_0).
\end{equation*}
\end{prop}

\begin{proof}
Since $f\in\mathcal{SH}_L^\reg(D_\theta)$, there exists some $e\in\mathcal{N}^\bnd(D_\theta)$, such that $ef\in\mathcal{SH}_L^\bnd(D_\theta)$ and $e(T)$ is injective. Then by Proposition~\ref{prop_Kernel_extended} there holds the inclusion
\begin{equation}\label{Eq_Kernel_Hinfty_1}
\ker(T)\subseteq\ker(e(T)-e_0).
\end{equation}
Since $e(T)$ is injective and $T$ is not injective, there has to be $e_0\neq 0$ and consequently also $e(s)\neq 0$ for $|s|$ small enough. Hence, there exists the limit
\begin{equation*}
\lim\limits_{|s|\rightarrow 0^+}f(s)=\lim\limits_{|s|\rightarrow 0^+}\frac{1}{e(s)}(ef)(s)=\frac{1}{e_0}(ef)_0.
\end{equation*}
Next, for every $v\in\ker(T)$ we have $e(T)v=e_0v$ by \eqref{Eq_Kernel_Hinfty_1} and hence $v\in\dom(e(T)^{-1})$ with
\begin{equation*}
e(T)^{-1}v=\frac{1}{e_0}v.
\end{equation*}
Moreover, since $f\in\mathcal{N}^\reg(D_\theta)$ is intrinsic, there also holds
\begin{equation*}
(ef)(T)v=(ef)_0v
\end{equation*}
by Proposition~\ref{prop_Kernel_extended}, and we obtain
\begin{equation*}
f(T)v=e(T)^{-1}(ef)(T)v=e(T)^{-1}(ef)_0v=\frac{1}{e_0}(ef)_0v=f_0v. \qedhere
\end{equation*}
\end{proof}

\begin{thm}\label{thm_Product_Hinfty}
Let $T\in\mathcal{K}(V)$ bisectorial of angle $\omega\in(0,\frac{\pi}{2})$, $g\in\mathcal{N}^\reg(D_\theta)$, $f\in\mathcal{SH}_L^\reg(D_\theta)$, $\theta\in(\omega,\frac{\pi}{2})$. Then there holds the product rule
\begin{equation}\label{Eq_Product_Hinfty}
(gf)(T)\supseteq g(T)f(T),
\end{equation}
and the domains of the both sides are connected by
\begin{equation}\label{Eq_Product_domain}
\dom(g(T)f(T))=\dom((gf)(T))\cap\dom(f(T)).
\end{equation}
\end{thm}

\begin{proof}
Let $e_1$ be a regularizer of $f$ and $e_2$ a regularizer of $g$. Then it is clear that $e_1e_2$ is a regularizer of the product $fg$. Since $g$ is intrinsic, there holds the commutation relation
\begin{equation*}
e_1(T)(e_2g)(T)=(e_2g)(T)e_1(T)
\end{equation*}
from Corollary~\ref{cor_Commutation_extended}~iii). Consequently, we get
\begin{align}
e_1(T)^{-1}(e_2g)(T)&\supseteq e_1(T)^{-1}(e_2g)(T)e_1(T)e_1(T)^{-1} \notag \\
&=e_1(T)^{-1}e_1(T)(e_2g)(T)e_1(T)^{-1}=(e_2g)(T)e_1(T)^{-1}. \label{Eq_Product_Hinfty_1}
\end{align}
Hence the $H^\infty$-functional calculus \eqref{Eq_Hinfty} of the product $(gf)(T)$ can be written as
\begin{align*}
(gf)(T)&=(e_1e_2)(T)^{-1}(e_1e_2gf)(T)=e_2(T)^{-1}e_1(T)^{-1}(e_2g)(T)(e_1f)(T) \\
&\supseteq e_2(T)^{-1}(e_2g)(T)e_1(T)^{-1}(e_1f)(T)=g(T)f(T).
\end{align*}
For the domain identity \eqref{Eq_Product_domain} we start with $v\in\dom(g(T)f(T))$. Then there is clearly $v\in\dom(f(T))$ and by the already proven product rule \eqref{Eq_Product_Hinfty} also $v\in\dom((gf)(T))$. \medskip

For the inverse inclusion in \eqref{Eq_Product_domain}, let $v\in\dom((gf)(T))\cap\dom(f(T))$. Then there is clearly $v\in\dom(f(T))$, and from the commutation relation \eqref{Eq_Product_Hinfty_1} there follows the identity
\begin{align}
(e_2g)(T)f(T)v&=(e_2g)(T)e_1(T)^{-1}(e_1f)(T)v \notag \\
&=e_1(T)^{-1}(e_2g)(T)(e_1f)(T)v=e_1(T)^{-1}(e_1e_2gf)(T)v. \label{Eq_Product_Hinfty_2}
\end{align}
Since also $v\in\dom((gf)(T))$, there is $(e_1e_2gf)(T)v\in\ran((e_1e_2)(T))$ and hence
\begin{equation}\label{Eq_Product_Hinfty_3}
e_1(T)^{-1}(e_1e_2gf)(T)v\in\ran(e_2(T)).
\end{equation}
Combining now the identity \eqref{Eq_Product_Hinfty_2} with \eqref{Eq_Product_Hinfty_3}, shows that $(e_2g)(T)f(T)v\in\ran(e_2(T))$, which means $f(T)\in\dom(g(T))$. Hence we have proven $v\in\dom(g(T)f(T))$.
\end{proof}

\begin{cor}\label{cor_Composition_Hinfty}
Let $T\in\mathcal{K}(V)$ be bisectorial of angle $\omega\in(0,\frac{\pi}{2})$, $g\in\mathcal{N}^\reg(D_\theta)$, $\theta\in(\omega,\frac{\pi}{2})$, and $p$ an intrinsic polynomial. Then there is $p\circ g\in\mathcal{N}^\reg(D_\theta)$ and
\begin{equation*}
(p\circ g)(T)\supseteq p[g(T)].
\end{equation*}
\end{cor}

\begin{proof}
Note that by the linearity property \eqref{Eq_Linearity_Hinfty} it is enough to consider monomials $p(s)=s^n$. Let $e$ be a regularizer of $g$. Then it is clear that $e^n$ is a regularizer of $g^n$. Using the relations
\begin{equation*}
(e^n)(T)=e(T)^n\qquad\text{and}\qquad(e^ng^n)(T)=(eg)(T)^n,
\end{equation*}
which follow from the product rule \eqref{Eq_Product_extended}, as well as the inclusion
\begin{equation*}
e(T)^{-1}(eg)(T)\supseteq(eg)(T)e(T)^{-1},
\end{equation*}
from \eqref{Eq_Product_Hinfty_1}, we can write the $H^\infty$-functional calculus of of $g^n$ as
\begin{equation*}
(g^n)(T)=(e^n)(T)^{-1}(e^ng^n)(T)=e(T)^{-n}(eg)(T)^n\supseteq\big(e(T)^{-1}(eg)(T)\big)^n=g(T)^n. \qedhere
\end{equation*}
\end{proof}

\begin{thm}\label{thm_Rational_Hinfty}
Let $T\in\mathcal{K}(V)$ be bisectorial of angle $\omega\in(0,\frac{\pi}{2})$ and $p,q$ be two intrinsic polynomials, such that $q$ does not admit any zeros in $\overline{D_\omega}$. Then, $q[T]$ is bijective, there exists some $\theta\in(\omega,\pi)$ such that $\frac{p}{q}\in\mathcal{N}^\reg(D_\theta)$, and we have
\begin{equation}\label{Eq_Rational_Hinfty}
\Big(\frac{p}{q}\Big)(T)=p[T]q[T]^{-1}.
\end{equation}
\end{thm}

\begin{proof}
If we denote $n:=\deg(p)$, we choose the regularizer
\begin{equation*}
e(s):=\frac{1}{(1+s^2)^n}.
\end{equation*}
Then it is clear that $e\in\mathcal{N}^\bnd(D_\theta)$. Since $q$ has no zero at $s=0$, we obtain the asymptotics
\begin{equation*}
e(s)\frac{p(s)}{q(s)}=\begin{cases} \mathcal{O}(1), & \text{as }|s|\rightarrow 0^+, \\ \mathcal{O}(\frac{1}{|s|^n}), & \text{as }|s|\rightarrow\infty, \end{cases}
\end{equation*}
which implies that $e\frac{p}{q}\in\mathcal{N}^\bnd(D_\theta)$. Since, by Theorem~\ref{thm_Rational_extended}, $e(T)=(1+T^2)^{-n}$ is also injective, $e$ is a regularizer of $\frac{p}{q}$. Using also Theorem~\ref{thm_Rational_extended}, the $H^\infty$-functional calculus of $\frac{p}{q}$ is then given by
\begin{align*}
\Big(\frac{p}{q}\Big)(T)&=e(T)^{-1}\Big(e\frac{p}{q}\Big)(T) \\
&=(1+T^2)^np[T]\big((1+s^2)^nq\big)[T]^{-1} \\
&=(1+T^2)^np[T](1+T^2)^{-n}q[T]^{-1} \\
&=p[T](1+T^2)^n(1+T^2)^{-n}q[T]^{-1}=p[T]q[T]^{-1}. \qedhere
\end{align*}
\end{proof}

\begin{prop}\label{prop_Product_Hinfty_pf}
Let $T\in\mathcal{K}(V)$ be bisectorial of angle $\omega\in(0,\frac{\pi}{2})$, $p$ an intrinsic polynomial, and $f\in\mathcal{SH}_L^\reg(D_\theta)$, $\theta\in(\omega,\frac{\pi}{2})$. Then we have
\begin{equation*}
(pf)(T)\supseteq p[T]f(T).
\end{equation*}
\end{prop}

\begin{proof}
If we use the product rule \eqref{Eq_Product_Hinfty} and the rational functional calculus \eqref{Eq_Rational_Hinfty} with $q=1$, we immediately get
\begin{equation*}
(pf)(T)\supseteq p(T)f(T)=p[T]f(T). \qedhere
\end{equation*}
\end{proof}

\begin{cor}\label{cor_Commutation_Hinfty}
Let $T\in\mathcal{K}(V)$ be bisectorial of angle $\omega\in(0,\frac{\pi}{2})$, $g\in\mathcal{N}^\reg(D_\theta)$, $\theta\in(\omega,\frac{\pi}{2})$. Then we have the following commutation relations: \medskip

\begin{enumerate}
\item[i)] For every $B\in\mathcal{B}(V)$ with $TB\supseteq BT$, we have
\begin{equation*}
Bg(T)\subseteq g(T)B.
\end{equation*}

\item[ii)] For every $f\in\mathcal{N}^\bnd(D_\theta)$, we have
\begin{equation*}
f(T)g(T)\subseteq g(T)f(T).
\end{equation*}
\end{enumerate}
\end{cor}

\begin{proof}
Let $e$ be a regularizer of $g$. \medskip

i)\;\;By Corollary~\ref{cor_Commutation_extended} i) we have $B(eg)(T)=(eg)(T)B$ and $Be(T)=e(T)B$. Hence, we get
\begin{equation*}
e(T)^{-1}B\supseteq e(T)^{-1}Be(T)e(T)^{-1}=e(T)^{-1}e(T)Be(T)^{-1}=Be(T)^{-1},
\end{equation*}
from which we conclude
\begin{equation*}
g(T)B=e(T)^{-1}(eg)(T)B=e(T)^{-1}B(eg)(T)\supseteq Be(T)^{-1}(eg)(T)=Bg(T).
\end{equation*}
ii)\;\;From the operator inclusion \eqref{Eq_Product_Hinfty_1}, we obtain
\begin{align*}
g(T)f(T)&=e(T)^{-1}(eg)(T)f(T)=e(T)^{-1}f(T)(eg)(T) \\
&\supseteq f(T)e(T)^{-1}(eg)(T)=f(T)g(T). \qedhere
\end{align*}
\end{proof}

\end{document}